\documentclass[12pt,reqno]{amsart}

\usepackage[colorlinks,linkcolor=blue,citecolor=blue]{hyperref}
\usepackage[abs]{overpic}
\usepackage[normalem]{ulem} 

\usepackage{amsmath, amscd, amssymb, amsthm,amsfonts}
\usepackage{euscript, mathrsfs, latexsym,mathabx,stmaryrd}
\usepackage{color}
\usepackage{hyperref} 

\usepackage{lmodern}
\usepackage[T1]{fontenc} 

\usepackage{multicol}

\ifx\pdfoutput\undefined
\usepackage{graphicx}
\else
\fi

\ifx\xetex
\usepackage{fontspec}
\usefont{EU1}{lmr}{m}{n}
\else
\fi

\usepackage{url} 
\usepackage{tikz}
\usetikzlibrary{matrix,arrows,positioning}
\tikzset{node distance=2cm, auto}

\usepackage[nodayofweek]{datetime}

\parskip=\medskipamount

\DeclareMathOperator{\colim}{colim}

\newcommand{\conj}[1]{\quad\textnormal{ #1 }\quad}

\newcommand{\inp}[1]{\ensuremath{\langle #1 \rangle}}

\newcommand{\module}[0]{\operatorname{-mod}}

\newcommand{\normaltext}[1]{\textnormal{#1}}

\makeatletter
\def\imod#1{\allowbreak\mkern2.5mu({\operator@font mod}\,#1)}
\makeatother

\renewcommand{\a}{\alpha}
\renewcommand{\b}{\beta}

\newcommand{\ott}{\otimes}

\newcommand{\s}{\sigma}

\newcommand{\Ga}{\Gamma}
\newcommand{\ga}{\gamma}

\renewcommand{\d}{\delta}

\newcommand{\CPPic}[1]{
\begin{minipage}{.8in}
\includegraphics[scale=1.1]{#1}
\end{minipage}
}

\newcommand{\aA}{\mathcal{A}}

\newcommand{\aC}{\mathcal{C}}
\newcommand{\aD}{\mathcal{D}}

\newcommand{\aF}{\mathcal{F}}

\newcommand{\aM}{\mathcal{M}}

\newcommand{\aS}{\mathcal{S}}
\newcommand{\aT}{\mathcal{T}}

\newcommand{\bn}{\mathbf{n}}

\newcommand{\fo}{\mathfrak{o}}

\usepackage{bbm}

\newcommand{\CC}{\mathbb{C}}

\newcommand{\FF}{\mathbb{F}}

\newcommand{\NN}{\mathbb{N}}

\newcommand{\PP}{\mathbb{P}}
\newcommand{\QQ}{\mathbb{Q}}

\newcommand{\ZZ}{\mathbb{Z}}

\newcommand{\kk}{\mathbb{k}}

\theoremstyle{plain}

\newtheorem{thm}{Theorem}[section]
\newtheorem*{introthm}{Theorem}
\newtheorem{theorem}[thm]{Theorem}

\newtheorem{conjecture}[thm]{Conjecture}
\newtheorem*{introconjecture}{Conjecture}

\newtheorem{proposition}[thm]{Proposition}
\newtheorem{prop}[thm]{Proposition}

\newtheorem{corollary}[thm]{Corollary}
\newtheorem{cor}[thm]{Corollary}
\newtheorem{lemma}[thm]{Lemma}

\theoremstyle{remark}

\theoremstyle{definition}

\newtheorem{example}[thm]{Example}

\newtheorem{defn}[thm]{Definition}
\newtheorem*{introdefn}{Definition}
\newtheorem{definition}[thm]{Definition}

\newtheorem{notation}[thm]{Notation}
\newtheorem{rmk}[thm]{Remark}
\newtheorem{remark}[thm]{Remark}
\newtheorem{terminology}[thm]{Terminology}
\numberwithin{equation}{section}

\makeatletter
\def\imod#1{\allowbreak\mkern2.5mu({\operator@font mod}\,#1)}
\makeatother

\makeatletter
\def\namedlabel#1#2{\begingroup
   \def\@currentlabel{#2}%
   \label{#1}\endgroup
}
\makeatother

\newcommand{\aprod}{\varoast}

\newcommand{\M}{\aM}
\newcommand{\Ch}{Ch}
\newcommand{\Ainf}{A_\infty}

\newcommand{\z}{z}

\newcommand{\h}{h}

\newcommand{\A}{\Lambda} 

\renewcommand{\kk}{k}

\newcommand{\lab}[1]{\normaltext{#1}}

\newcommand{\vnp}[1]{\lvert #1 \rvert}

\newcommand{\xto}[1]{\xrightarrow{#1}}

\newcommand{\from}{\leftarrow}
\usepackage{extarrows}

\theoremstyle{plain}

\newcommand{\NAlg}{\bn}

\newcommand{\Ho}{H^0}

\newcommand{\End}{End}

\newcommand{\Aut}{Aut}
\newcommand{\Endo}{End}

\newcommand{\Obj}{Ob}
\newcommand{\Ob}{\Obj}

\renewcommand{\s}{\sigma}

\newcommand{\Hom}{Hom}
\newcommand{\Ext}{Ext}

\newcommand{\Alg}{{\bf f}}
\newcommand{\SAlg}{{\bf F}}
\newcommand{\DHa}{\mathcal{D}\!\mathcal{H}}
\newcommand{\tDHa}{D\!H\!a}
\newcommand{\Ha}{H\!a}

\renewcommand{\L}{L}
\newcommand{\Rep}{Rep}

\newcommand{\E}{E}

\newcommand{\zim}{z_{i,n}}
\newcommand{\zjm}{z_{j,n}}
\newcommand{\zjmp}{z_{j,n+1}}
\newcommand{\zjmk}{z_{j,n+k}}
\newcommand{\dqt}{D^b(\Rep_\kk(A_2))}

\newcommand{\Tw}{T\!w}
\newcommand{\Si}{\Sigma}
\newcommand{\F}{\aF}
\newcommand{\FS}{\aF\!\aS}

\newcommand{\da}{(D^2,\A)}

\newcommand{\dfda}{D^\pi\F\da}

\newcommand{\fol}[3]{\inp{#1, #2}_{#3}}

\newcommand{\U}{\mathcal {U}}
\newcommand{\g}{\mathfrak{g}}

\begin{document}

\title[Hall algebras of surfaces]{The Hall algebras of surfaces I}
\author[Benjamin Cooper]{Benjamin Cooper}
\author[Peter Samuelson]{Peter Samuelson}
\address{University of Iowa, Department of Mathematics, 14 MacLean Hall, Iowa City, IA 52242-1419 USA}
\email{ben-cooper\char 64 uiowa.edu}
\address{University of Edinburgh, Department of Mathematics, 4F4 185 Causewayside Edinburgh, EH9 1PH, UK}
\email{peter.samuelson\char 64 ed.ac.uk}

\begin{abstract}
We study the derived Hall algebra of the partially wrapped Fukaya category
of a surface. We give an explicit description of the Hall algebra for the
disk with $m$ marked intervals
and we give a conjectural description of the Hall algebras of all surfaces with enough marked intervals. Then we use a functoriality result to show
that a graded version of the HOMFLY-PT skein relation holds among certain 
arcs in the Hall algebras of general surfaces.
\end{abstract}
\vspace{.5in}
\dedicatory{To Charlie Frohman, for his 61\raisebox{0.2em}{st} birthday.}

\maketitle
\setcounter{tocdepth}{1}
\setcounter{secnumdepth}{3}
\vspace{.5in}
\tableofcontents

\section{Introduction}\label{introsec}
The Hall algebra construction of quantum groups has led to many insights in
representation theory. Fukaya categories are central to the study of
symplectic geometry and topological field theory. In this section we
introduce and motivate our study of the Hall algebras of Fukaya categories of surfaces
and their relationship to the HOMFLY-PT skein algebras.

\pagebreak

\subsection{Hall algebras}
The Hall algebra $\Ha(\aA)$ is an invariant of abelian categories
that captures information about the extensions between objects. As a vector space, $\Ha(\aA)$ has a basis given by isomorphism classes of objects in $\aA$.  If $A,C\in\Obj(\aA)/iso$ then
their product in $\Ha(\aA)$ is given by
\begin{equation}\label{halleq}A\cdot C = \sum_{B}\frac{\vnp{\Ext^1(A,C)_B}}{\vnp{\Hom(A,C)}} B
\end{equation}
where $Ext^1(A,C)_B \subset Ext^1(A,C)$ is the subset corresponding to short exact sequences of the form $0\to C\to B \to A\to 0$, see \cite[\S 2.3]{Bri13}.

  Ringel proved that the positive half of the quantum group
  $U^+_q(\mathfrak{g})$ associated to a simply laced Dynkin quiver $Q$ is isomorphic to 
  the Hall algebra of the corresponding abelian category of quiver
  representations
  \[
  U^+_q(\mathfrak{g}) \cong \Ha(\Rep_{\FF_q}(Q))
  \] 
see
  \cite{Ringel}. 
  Lusztig found a geometric interpretation
  for this construction which led to many important results in representation
  theory \cite{LQG, Lus90}.  The higher
  categorical structure present in Lusztig's work inspired Crane and Frenkel
  to conjecture the existence of corresponding 4-dimensional topological
  field theories \cite{CraneFrenkel}. These conjectures represent one of the
  driving forces behind present interest in categorified quantum groups
  \cite{KL09,Rou08,VV}.

  Abelian categories of coherent sheaves on curves are similar to categories of
  quiver representations because both have homological dimension one.  By
  virtue of Ringel's theorem, the Hall algebras of coherent sheaves on curves are analogues of quantum groups.
  Although the Hall algebra of the category 
  $Coh(\mathbb{P}^1)$ can be understood terms of the affine algebra
  $U_q(\widehat{\mathfrak{sl}}_2)$, the elliptic Hall algebra, obtained from
  coherent sheaves on curves of genus one, has new features and has been
  the subject of much recent study. 
  Burban and Schiffmann gave an explicit description of this algebra
  \cite{BS} and showed the relations could be written as Laurent
  polynomials in parameters $q$ and $t$, which are the eigenvalues of the Frobenius
  operator acting on the first cohomology of the curve. Subsequent work has related
  this algebra to the equivariant $K$-theory of Hilbert schemes of points in
  $\CC^2$ \cite{SV13, FT11, Neg14}, triply graded knot homology \cite{GN15,
    Che13}, positivity conjectures in algebraic combinatorics \cite{BGSX16} and
  the AGT conjecture in mathematical physics \cite{SV13agt}. Recently,
  Schiffmann used the Hall algebras associated to curves of higher genus to
  compute the Betti numbers of moduli spaces of Higgs bundles \cite{Sch16}.
  
\subsection{Fukaya categories}
Homological mirror symmetry postulates a relationship between the algebraic
geometry and the symplectic geometry of a space $X$ and a potential
function $W : X^\vee\to \CC$ on a mirror space. 
Each object determines two
triangulated categories: a B-model, associated to a category of coherent
sheaves, and an A-model, associated to a Fukaya category \cite{Seidel}. 
The algebraic geometry underlying the B-model can be understood in terms of vector bundles
because a coherent sheaf is the cokernel of a map $f : E_1 \to E_2$ between
algebraic vector bundles.  On the other hand, Fukaya categories are defined
in terms of symplectic geometry. The objects are Lagrangian submanifolds and
a map $p : L_1 \to L_2$ is a point $p$ in the  intersection $L_1\cap L_2$. The
composition of maps is associative up to a collection of coherent
homotopies, called an $\Ainf$-structure, which is defined in terms of moduli
spaces of pseudo-holomorphic curves.  For some surfaces, a collection of
Lagrangians, which decompose the surface into disks, determine a
presentation called the topological Fukaya category. Many authors have studied Fukaya categories of surfaces 
\cite{DK, Bok, Nadler, Abo08, STZ14}, 
in this paper we follow the exposition in \cite{HKK}.

When homological mirror symmetry holds, there are isomorphisms between pairs
of categories:
\begin{equation}\label{hmseqn}
  D^b Coh(X) \xto{\sim} D^\pi \FS(X^\vee,W) \conj{ and } D^\pi\F(X) \xto{\sim} D^b_{sing}Coh(X^\vee,W)
\end{equation}  
where the categories on the right-hand side are constructions which account for singularities \cite{Kapustin}.
Since these categories are defined in terms of different geometries on different spaces, the isomorphisms above determine a dictionary between the subjects which underly their constructions.

From the perspective of homological mirror symmetry, the Hall algebras of
coherent sheaves on curves are invariants of the B-model in dimension 1. In
this paper, we study invariants of the A-model in dimension 1: the Hall
algebras of Fukaya categories of surfaces. We relate these algebras to the
HOMFLY-PT skein algebras.

\subsection{Skein algebras}\label{skintrosec}
The HOMFLY-PT skein algebras $H_q(S)$ are constructed topologically and 
have deep connections to moduli spaces of flat connections on surfaces and
quantum invariants of 3-manifolds.  The Goldman Lie
algebra of a surface $S$ is the universal Lie algebra which maps to the Poisson Lie algebras
of functions on the $GL_n(\CC)$-representation varieties of $S$
\cite{Goldman}.  Turaev showed that the algebra $H_q(S)$ is
a deformation of the universal enveloping algebra of the Goldman Lie algebra
\cite{Turaevskein}.  On the other hand, after setting the variables in
the ground ring to roots of unity, the skein algebras are known to be intimately related to 
the $SU(k)_{\ell}$ Witten-Reshetikhin-Turaev invariants of 3-manifolds \cite{Turaevbook, BHMV, Witten}.
Although skein algebras and skein modules contain a wealth of
structure, their definition is relatively short.

\begin{introdefn}
If $M$ is an oriented 3-manifold then the {\em HOMFLY-PT skein
  module} $H_q(M)$ consists of $R$-linear combinations of isotopy classes of
framed oriented links in $M$ subject to the HOMFLY-PT skein relation:
\begin{equation}\label{homflypteq}
\CPPic{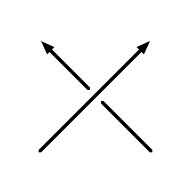} -  \CPPic{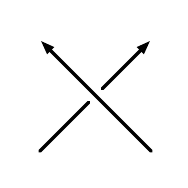} = (q-q^{-1})\CPPic{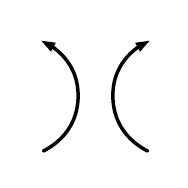}
\end{equation}
and the framing relations
$$\CPPic{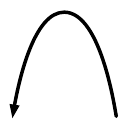} = \quad v \CPPic{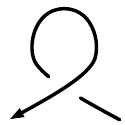} \conj{ and }\quad \CPPic{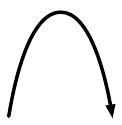} = \quad v^{-1} \CPPic{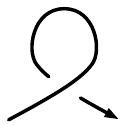}$$
where $R=k[v^{\pm 1}, q^{\pm 1}, (q-q^{-1})^{-1}]$ is the coefficient ring. If $S$ is an
orientable surface then the skein module $H_q(S)= H_q(S\times [0,1])$ is an
algebra with a product $H_q(S)\otimes_R H_q(S)\to H_q(S)$ determined by the stacking map
$S\times [0,1] \sqcup S\times [1,2]\xto{\sim} S\times [0,1]$.
\end{introdefn}

One goal of this paper is to establish a relationship between these
skein algebras and the Hall algebras of the Fukaya categories of
surfaces. The remainder of this section contains motivation and context for
the main construction.

Morton and the second author constructed an isomorphism between the HOMFLY-PT skein algebra of the torus $T^2$ and the $t=1/q$ specialization of the elliptic Hall algebra  \cite{MortonSamuelson}.  
The proof of this isomorphism was by brute force; however, it seems that 
it is a rough analogue of an instance of Kontsevich's homological mirror symmetry conjecture. In this particular case, 
mirror symmetry states
that the derived category of coherent sheaves on a complex elliptic curve $X$ is equivalent to the Fukaya category of a symplectic torus \cite{PZ}. This gives a hint that the Hall algebra of the Fukaya category may be related to the skein algebra, at least for the torus.

There is a second way to guess that the Hall algebras of Fukaya
categories should be related to skein algebras.  If all of the extensions
between two Lagrangians $L_1$ and $L_2$ in the Fukaya category can be realized by Lagrangian surgeries then the commutator $L_1\cdot L_2-L_2\cdot L_1$ in the Hall
algebra 
is given by a sum over all possible Lagrangians that can be obtained by surgery at the points
of intersection $L_1\cap L_2$. For surfaces, the terms in this sum are the same as those in the formula defining the bracket $[L_1,L_2]$ in the Goldman Lie algebra.
This suggests that the Hall algebra of a surface is a deformation of
the universal enveloping algebra of the Goldman Lie algebra. By Turaev
\cite{Turaevskein}, a candidate for such an algebra is the skein algebra.

There are many technical obstacles to making the intuitions expressed in the
last two paragraphs mathematically meaningful.  The next section contains
precise statements of results.

\subsection{Statements of results}\label{introcontentsec}
Two issues  complicate defining the Hall algebra of the
Fukaya category.  The Fukaya categories are situated within the context of
homotopy associative categories and the standard Hall algebra construction
\eqref{halleq} is not homotopy invariant.  Moreover, the Hall algebra is
only defined for categories which satisfy several finiteness properties.
For these reasons, we apply To\"{e}n's derived Hall algebra construction 
to the derived categories of topological Fukaya categories, see
\cite{Toen,HKK}. Sections \ref{dhasec} and \ref{fuksec} contain detailed
discussions of these definitions.

In Proposition \ref{finiteimpliesfiniteprop}, the topological Fukaya
category $\F(S,A)$ of a marked surface $(S,M)$ is shown to satisfy the
finiteness conditions required by To\"{e}n's construction when $S$ is a {\em
  finitary surface}: a compact surface with boundary where each boundary component
contains at least one marked interval.

Although the categories $\F(S,A)$ are defined in terms of a system of arcs $A$ which cut $S$ into a union of 
disks, the derived categories $D^\pi\F(S,A)$ are all equivalent.
We define algebras $\Alg(S,A)$, naturally associated to each arc system, which are
related by the isomorphisms induced by the equivalences between derived
categories, see Cor. \ref{sheafpropcor}. 
To eliminate the dependence on arcs, we define the algebra $\SAlg(S,M)$ as the colimit of the functor determined by the  assignment $A\mapsto \Alg(S,A)$.
  $$
  \SAlg(S,M) := \underset{(S,A)\in \M(S,M)}{\colim} \Alg(S,A)
  $$
For brevity, we call this the Fukaya Hall algebra.

Within this context the Fukaya Hall algebra of the disk $D^2$ with $m$ marked
intervals is studied in detail.  
The relationship between the Fukaya category $\F(D^2,\A_m) $ of the disk equipped with minimal arc system $\A_m$ and the category of $A_{m-1}$-quiver representations determines a presentation for the Hall algebra.

\begin{introthm}{(\ref{minarcthm})}
Suppose that the disk with $m$ marked intervals is equipped with a minimal arc system $\A_m$ and foliation data $\h : \A_m \to \ZZ$.  Then the algebra $\Alg(D^2,\A_m)$ is the 
generated by the suspensions of the arcs
$$\ZZ \A_m = \{ \E_{i,n} : n\in\ZZ, \E_{i}\in \A_m\}\conj{ where } \E_{i,n} = \s^n\E_{i}$$
subject to the relations listed below.\\
\vspace{-.2in}
\begin{description}
  \item[(R1)] Self-extension:
\begin{align*}
[\E_{i,0},\E_{i,k}]_{q^{2(-1)^k}} &=\d_{k,1}\frac{q^{-1}}{q^2-1}\conj{ for } k \geq 1,
\end{align*}

  \item[(R2)] Adjacent commutativity and convolution:
\begin{align*}[\E_{i+1,k},\E_{i,h(i)}]_{q^{(-1)^{k+1}}} &= 0  \conj{ for } k > 1,\\
[\E_{i+1,k},\E_{i,h(i)}]_{q^{(-1)^{k}}} &= 0 \conj{ for } k<1\notag, 
\end{align*}
\begin{equation*}
\E_{i,h(i)} = [\E_{i+m-1,\tau^i\inp{m-1}}, \ldots, \E_{i+2,\tau^i\inp{2}}, \E_{i+1,\tau^i\inp{1}}]_q,
\end{equation*}

 \item[(R3)] Far-commutativity: 
\begin{align*}[\E_{i,0},\E_{j,k}]_1 &= 0\conj{ for } k\in\ZZ \quad\normaltext{ and }\quad \vnp{i-j} \geq 2.
\end{align*}
\end{description}
\end{introthm}
(See Notation \ref{disknotation} for the necessary definitions.)

The proof of this theorem uses the work of Hernandez and Leclerc, who gave a description of the derived Hall algebras of categories of quiver representations \cite[\S 8]{HL}.
Theorem \ref{gluethm} constructs presentations for the Hall algebras
$\Alg(D^2,A)$ associated to Fukaya categories of arbitrary arc systems  
by gluing the algebras associated to minimal arc systems.

In order to study the behavior of the Fukaya Hall algebras of surfaces using the Hall algebra of the disk, we establish conditions under which an embedding of marked surfaces induces a monomorphism of algebras.

\begin{introthm}{(\S \ref{hallsurfsec})}
Suppose that $f : (S,M) \to (S',M')$ is an embedding of marked surfaces 
which induces an injection $\pi_0(M) \to \pi_0(M')$ between the sets of components of marked intervals. Then $f$ induces a monomorphism $f_* : \SAlg(S,M) \to \SAlg(S',M')$ of Fukaya Hall algebras.
\end{introthm}

Given a marked surface $(S,M)$ with a full arc system $A$ that cuts the surface into disks $(D_k,\A_{m_k})$, the gluing procedure for disks allows us to define a  \emph{naive algebra} $\NAlg(S,A)$, see Def. \ref{def:naive}. This algebra is the free product of the disk algebras $\Alg(D_k,\A_{m_k})$ subject to the gluing relations in Def. 
\ref{alggluedef}. 
If $(S,M)$ has enough marked intervals then Thm. \ref{prop:naivemaps} uses the functoriality statement above to construct the map below.
\begin{introthm}{(\ref{prop:naivemaps})}
There is a surjective homomorphism $\gamma$ from the naive algebra to the composition subalgebra.
\begin{equation}\label{eq:naivemap}
\gamma: \NAlg(S,A) \twoheadrightarrow \Alg(S,A)
\end{equation}
\end{introthm}
This map is shown to be an isomorphism for disks in Prop. \ref{fullarcprop}. We conjecture that this holds whenever $(S,M)$ has enough marked intervals; a property which holds for many surfaces.
\begin{introconjecture}{(\ref{conjconj})}
The map $\gamma$ is an algebra isomorphism.
\end{introconjecture}

The remainder of the paper shows that a graded analogue of 
the HOMFLY-PT skein relation appears as a commutator relation between certain curves in 
the Fukaya Hall algebras $\SAlg(S,M)$ of all surfaces. If $X$ and $Y$ are two arcs that intersect then they represent generators in distinct Hall algebras.
However, $X$ and $Y$ play an equal role within the Fukaya Hall algebra
$\SAlg(S,M)$.  It is within this context that a graded version of the
HOMFLY-PT skein relation appears.

\begin{introthm}{(\ref{skeinrelcor})}
  Suppose that $X$ and $Y$ are graded arcs whose endpoints intersect four
  distinct marked intervals in a finitary marked surface $(S,M)$.  If $X$
  and $Y$ intersect uniquely at a point $p$ then the 
  relation
\begin{equation}\label{grskeineq}
[X,Y]_1 = \left\{ \begin{array}{cl}
(q-q^{-1})X\#_p Y & \textrm{if } i_p(X,Y)=1\\
0&\textrm{if } i_p(X,Y) > 1
\end{array}
\right.
\end{equation}
holds in the Fukaya Hall algebra $\SAlg(S,M)$. 
\end{introthm}

The idea of the proof is to enclose the two arcs 
within a disk $(D^2,4)$
with four marked intervals. By the functoriality theorem, relations
between $X$ and $Y$ are implied by relations in 
the Hall algebra $\SAlg(D^2,4)$
of the disk. This local computation in the disk is given in Lemma \ref{skeinthmthm}.

The commutator relation above is called the graded HOMFLY-PT skein relation
because the curves in $\F(S,A)$ are graded, not oriented. If the foliation
used to define the grading comes from a 1-form then each graded curve
inherits an orientation and Eqn. \eqref{grskeineq}
can be pictured in the same manner as the oriented HOMFLY-PT relation in
Eqn. \eqref{homflypteq}.

Materials in this paper suggest many extensions and
variations of the literature. For example, 
we expect that one can extend the Goldman Lie algebras and
the HOMFLY-PT skein algebras to surfaces with boundary and marked intervals
using Thm. \ref{minarcthm}, 
see Remark \ref{rmk:skeinskein}. 
String topology was introduced by Chas and Sullivan to provide a context for the
Goldman Lie algebra. A relationship between string topology and the Hall
algebras of surfaces would be enlightening.
Very recently, a preprint by Lekili and Polishchuk studies homological
mirror symmetry between partially wrapped Fukaya categories and certain
nodal stacky curves \cite{Lekili}. Our construction suggests that a skein relation
holds within the Hall algebras of derived categories of coherent sheaves on
these curves. It would be interesting to study these algebras from the
B-model perspective.

\subsection{Organization}\label{orgsec}
This paper is organized in the following way.

Section 2 begins with a review of
basic facts about
Hall algebras and the twisted derived Hall algebra $\tDHa(\aT)$.
 Section \ref{latticealgsec} recalls the computation
of the derived Hall algebras of simply-laced quivers performed by
Hernandez and Leclerc \cite{HL}.

Section 3 contains a detailed study of the quiver algebra $\tDHa(A_{m-1})$
introduced in \S \ref{latticealgsec}. 
Families of special elements
$z_{(a,b),n} \in \tDHa(A_{m-1})$ are introduced and many relations between
them are proven. 
A summary appears in Appendix II,
see Remark \ref{arcrelssec}.

Section 4 discusses the topology of graded surfaces
\S\ref{gradedsurfsec}, marked surfaces \S\ref{markedsec}, and  the concept of a
finitary surface \S\ref{markedsec}.  Section \S\ref{ainfsec}
contains the Fukaya category $\F(S,A)$
and  
derived Fukaya category $D^\pi\F(S,A)$. The section concludes with
Prop. \ref{finiteimpliesfiniteprop} which shows that $D^\pi\F(S,A)$ is a
finitary triangulated category when $S$ is a finitary surface. Section
\ref{disksec} features several key results about the Fukaya categories of
disks.

Section 5 introduces the composition subalgebras $\Alg(S,A)$ in
Def. \ref{compalgdef}. These algebras are combined to form the 
Fukaya Hall algebras
$\SAlg(S,M)$ by Def. \ref{skdef}. Section \ref{hallsurfsec} contains 
functoriality results, see Cor. \ref{skdiskembcor}. The remainder
of Section 5 studies the disk. Theorem \ref{minarcthm} produces a presentation for $\Alg(D^2,\A_m)$
the disk with minimal arc system and Thm. \ref{gluethm} extends this result
to produce a presentation for $\Alg(D^2,A)$.

Section 6 contains Theorem \ref{skeinrelcor} which shows that certain 
arcs in the algebras $\SAlg(S,M)$ satisfy the graded HOMFLY-PT skein relation.

Appendix I relates the elements $z_{(a,b),n}\in \tDHa(A_{m-1})$ to Lusztig's
PBW basis. Appendix II contains a useful index of notation, a summary of
relations which hold among the elements $z_{(a,b),n}$ and some remarks about
$q$-algebra.

\subsection{Acknowledgments} 
The authors would like to thank M.\ Abouzaid, G. Arlie,
A. Appel, B. Gammage, F. Goodman, R. Lipshitz, Y. Lekili,
O. Schiffmann and E. Zaslow for helpful conversations. 
We would especially like to thank F. Haiden for patiently answering many questions and C.\ Frohman for bringing the authors to Iowa. The work of the second author was funded in part by the European Research Council grant no. 637618.

\vskip .25in

\section{Hall algebras}

\subsection{Hall algebras of triangulated categories}\label{dhasec}
If $\aT$ is a triangulated category satisfying some modest finiteness
conditions then the derived Hall algebra $\DHa(\aT)$ is an 
invariant of $\aT$. The multiplication in this algebra encodes the structure of
extensions between various objects in the category $\aT$ \cite{Toen,
  XX}. There is a version $\tDHa(\aT)$ of this construction which is
obtained by twisting the multiplication. Our main interest is in the
invariant $\tDHa(\aT)$ when $\aT$ is a triangulated category obtained from
the Fukaya category of a surface.

\begin{notation}\label{ringnotation1}
In this section, the ring $\kk$ is required to be a finite field.
\end{notation}

\begin{defn}\label{finitenessdef}
An additive $\kk$-linear category $\aT$ is {\em finitary} when the two properties below hold.
\begin{enumerate}
\item For any two objects $X,Y\in\Ob(\aT)$, the cardinality of the set $\Hom(X,Y)$ is finite:
$$ \vnp{\Hom(X,Y)} < \infty.$$ 
\item If $X\in\Ob(\aT)$ is an indecomposable object then the ring of endomorphisms $\Endo(X)$ is a finite dimensional local $\kk$-algebra.
\end{enumerate}
A {\em finitary triangulated category $\aT$} is 
finitary as an additive $\kk$-linear category and {\em left homologically finite}: for any two objects $X,Y\in\Ob(\aT)$, the sum below is finite.
$$\sum_{n\geq 0} \dim_k \Ext^{-n}(X,Y) < \infty$$
\end{defn}

\begin{remark}
  When a triangulated category $\aT$ is finitary, each set $\Ext^n(X,Y)$ has finite cardinality for all $n\in\ZZ$, and the total cardinality $\Ext^{\leq 0}(X,Y)$ in negative degree is finite.
\end{remark}

\begin{notation}\label{homnote}
The following notation simplifies the formula for the product below.
$$\{X,Y\} := \prod_{n > 0} \vnp{\Ext^{-n}(X,Y)}^{(-1)^n}$$ 
$$\Hom(X,L)_Y := \{ f\in \Hom(X,L) : C(f)\cong Y \}$$
\end{notation}

\begin{definition}\label{toenalgdef}
In the notation above, {\em To\"{e}n's algebra structure constants} are defined to be:
$$F^L_{X,Y} := \frac{\vnp{\Hom(X,L)_Y}}{\vnp{\Aut(X)}}\cdot \frac{\{X,L\}}{\{X,X\}} \conj{ for } X,Y,L \in \Ob(\aT)/iso.$$
\end{definition}

The analogous formula for the Hall algebra of an abelian category does not
contain the terms in the curly braces above. Roughly speaking, these terms
compensate for the difference between the short exact sequences which are
used to define the Hall product for abelian categories and the distinguished
triangles, which represent long exact sequences, used to define the derived
Hall product. Using the constants $F^L_{X,Y}$ above, the derived Hall
algebra is introduced by the theorem below.

\begin{thm}[\cite{XX, Toen}]\label{dhadef}
Suppose that $\aT$ is a finitary triangulated category. Then {\em the derived Hall algebra of $\aT$} is the associative algebra determined by the abelian group $\DHa(\aT) = \QQ\inp{\Ob(\aT)/iso}$ spanned by isomorphism classes of objects in $\aT$ together with the multiplication below.
$$[X]\cdot [Y] := \sum_{[L]} F_{X,Y}^L [L]$$
Moreover, an exact equivalence $f : \aT \to \aT'$ of triangulated categories induces an isomorphism of derived Hall algebras: $f_* : \DHa(\aT) \xto{\sim} \DHa(\aT')$.
\end{thm}
\begin{proof}
  The statement that $\DHa(\aT)$ is an associative algebra is the main
  theorem of \cite{XX}. Invariance under exact equivalences of triangulated
  categories follows from the observation that such equivalences induce natural
  bijections among the sets which determine the structural constants above.
\end{proof}

We introduce a twisted analogue, $\tDHa(\aT)$, of the derived Hall algebra. This twisting of the Hall product by the Euler form goes back to Ringel \cite{Ringel}; without it the Serre relations appear somewhat mangled.

\begin{defn}\label{grotheulerdef}
Suppose that $\aT$ is a $\kk$-linear triangulated category. Then the 
{\em Grothendieck group $K_0(\aT)$ of $\aT$} is given by
$$K_0(\aT) := \ZZ\inp{\Ob(\aT)}/\sim$$
where the relation $[B] \sim [A] + [C]$ is imposed whenever there is a distinguished triangle of the form $A \to B \to C \to A[1]$. 
The {\em Euler pairing}:
$$\inp{X,Y} : K_0(\aT)\ott K_0(\aT) \to \ZZ \conj{is} \inp{X,Y} := \sum_{n\in\ZZ} (-1)^n \dim_k \Ext^n(X,Y).$$
\end{defn}

Since isomorphic objects in $\aT$ are identified by the relation $\sim$, the quotient map $\Ob(\aT)\to K_0(\aT)$ factors through a map $\Ob(\aT)/iso \to K_0(\aT)$. By extending this map additively to a map $\ZZ\inp{\Ob(\aT)/iso}\to K_0(\aT)$, the Euler pairing can be applied to elements of the derived Hall algebra $\DHa(\aT)$. 

\begin{defn}\label{dha2def}
The {\em twisted derived Hall algebra} $\tDHa(\aT)$ is determined by an extension of the abelian group
\begin{equation}\label{dhaabeqn}
\tDHa(\aT) := \DHa(\aT)\ott_{\QQ} \QQ(\sqrt{q})
\end{equation}
 where $q$ is the
characteristic of the field $\kk$. However, multiplication is scaled by the
Euler pairing $\inp{-,-}$ introduced by the previous definition:
$$[X]*[Y] := q^{\inp{Y,X}/2} [X]\cdot [Y].$$
\end{defn}

\begin{terminology}
Throughout the remainder of the paper we will refer to the twisted derived Hall algebra of a triangulated category $\aT$ as the Hall algebra or the derived Hall algebra of $\aT$.
\end{terminology}

\begin{rmk}\label{dhafunrmk}
  The Hall algebra can be understood as a functor from the groupoid category
  of finitary triangulated categories and exact equivalences. This is a
  restatement of invariance. The construction remains functorial under
  inclusions between categories which are full, faithful and exact, but
  functoriality in any broader sense is not expected.
\end{rmk}

\subsection{Hall algebras of quivers}\label{latticealgsec}

Suppose that $Q$ is a simply-laced Dynkin quiver. Denote by $I$ the vertices of
$Q$ and $\cdot : I\times I \to \ZZ$ the Cartan pairing. Within the category
of finite dimensional representations $\Rep_\kk(Q)$ there is a unique
1-dimensional simple module $z_i$ associated to each vertex $i\in I$. Since
$\Rep_{\kk}(Q)$ is hereditary, the derived category of bounded cochain
complexes $D^b(\Rep_\kk(Q))$ is a finitary triangulated category when $\kk$
is a finite field. In the proposition below, we recall a presentation for
the derived Hall algebra
\begin{equation}\label{tdhaqdef}
\tDHa(Q) := \tDHa (D^b\Rep_\kk(Q))
\end{equation}
 which was computed by D. Hernandez and B. Leclerc, see \cite[Prop. 8.1]{HL}.

There are objects $\zim = z_i[n]$ in the derived category $D^b(\Rep_\kk(Q))$ of cochain complexes of $Q$-representations associated to the $n$-fold suspension of the simple module at the $i$th vertex. Their presentation uses elements $\zim \in \tDHa(Q)$ associated to these objects.

\begin{prop}[{\cite{HL}}]\label{dhaanprop}
  The derived Hall algebra $\tDHa(Q)$ is generated by the elements
  $z_{i,n}$, for $i\in I, n\in \ZZ$, which are subject to the relations:
\begin{description}
\item[(H1)\namedlabel{h1:itm}{\lab{H1}}] for $n\in \ZZ$,
\begin{align*}
\zim \zjm - \zjm \zim &=0 \conj{if} i\cdot j = 0,\\
\zim^2 \zjm - (q + q^{-1}) \zim\zjm\zim + \zjm \zim^2  &= 0  \conj{if} i\cdot j = -1,
\end{align*}
\item[(H2)\namedlabel{h2:itm}{\lab{H2}}] for $m\in\ZZ$ and $i,j\in I$,
\begin{align*}
\zim\zjmp = q^{-i\cdot j} \zjmp\zim + \d_{ij}\frac{q^{-1}}{q^2 -1},
\end{align*}
\item[(H3)\namedlabel{h3:itm}{\lab{H3}}] for $k > 1$ and $i,j\in I$,
  \begin{align*}
\zim\zjmk = q^{(-1)^k i\cdot j} \zjmk \zim.
  \end{align*}
\end{description}

\end{prop}

\begin{rmk}\label{dhaamrmk}
When the quiver $Q$ is $A_{m-1}$, the generators $\zim$ of $\tDHa(A_{m-1})$ are indexed by vertices $I =\{ i\in\ZZ : 1\leq i < m\}$ and the Cartan matrix is determined by the assignments: 
$$i\cdot j = \left\{\begin{array}{cl} 
2 & \vnp{i-j} = 0,\\ 
-1 & \vnp{i-j}=1, \\ 
 0  & \vnp{i-j} \geq 2.
\end{array}\right.$$

In this case, the relations above can be organized as follows:
\begin{align}
   [z_{i,n},z_{j,n}]_1 &=0  & \textrm{ when } i \cdot j &= 0 \label{eq:z1}\tag{Z1}\\
  [z_{i,n}, [z_{i,n},z_{j,n}]_{q^{\mp 1}}]_{q^{\pm 1}} &= 0  & \textrm{ when }  i \cdot j &= -1\notag\\
[z_{i,n},z_{j,n+k}]_{q^{(-1)^k i\cdot j}}&= \delta_{i,j}\delta_{k,1}\frac{q^{-1}}{q^2-1} & \textrm{ when }  k \geq &1\label{eq:z2}\tag{Z2}
\end{align}
The relations \eqref{eq:z1} are equivalent to \eqref{h1:itm}. The second equation \eqref{eq:z2} combines \eqref{h2:itm} and \eqref{h3:itm} from the definition above.
\end{rmk}

\begin{rmk}
  B. To\"{e}n remarked on the similarities between his (untwisted) derived
  Hall algebra and the lattice algebra $\L(\Rep_{\kk}(Q))$ introduced by
  M. Kapranov, see \cite{Kapranov} and \cite[\S 7]{Toen}. 
\end{rmk}

\vskip 1.25in
\section{Relations in the quiver algebra}\label{algsec}

In this section we study relations among certain elements
$z_{(a,b),n} \in \tDHa(A_{m-1})$ where $1\leq a < b \leq m$ and $n\in\ZZ$. Each such element is an iterated
commutator of adjacent generators $z_{i,n}$:
$$z_{(a,b),n} = [z_{b-1,n}, [z_{b-2,n}, [\cdots, [z_{a+1,n},z_{a,n}]_q]_q]_q]_q.$$
When interpreted as an element in the Hall algebra of the Fukaya category of the disk using the parameterization $\phi_i$ from
Prop. \ref{dereqprop}, the element $z_{(a,b),n}$ is a graded arc which bisects the disk $(D^2,m)$. The relations shown to hold in this section are those of arcs in the Hall algebra of the Fukaya category which are needed to establish the isomorphism in Thm. \ref{minarcthm}.

\begin{defn}\label{bracketnotationdef}
If $A$ is a $\ZZ[q]$-algebra and $\{a_i\}_{i=1}^{n}\subset A$ is a collection of elements then the {\em right $q$-bracketing} of the word $a_na_{n-1}\cdots a_2a_1$ will be denoted by
$$[a_n,a_{n-1},\ldots, a_2,a_1]_q := [a_n, [a_{n-1}, [\cdots, [a_2,a_1]_q]_q]_q]_q$$
and, when $n=1$, $[a_n]_q = a_n$.
\end{defn}

The proposition below implies that most of the iterated $q$-bracketings
considered in this paper will not depend on the order in which the $q$-Lie brackets
are applied. The notation introduced above was chosen to emphasize this
observation.

\begin{prop}{(Rebracketing)}\namedlabel{orderprop}{\lab{RB}}
Suppose that $A$ is a $\ZZ[q]$-algebra with elements $\{a_i\}_{i=1}^n \subset A$ which satisfy the property:
$$[a_i,a_j]_1=0\conj{when} \vnp{i-j} \geq 2.$$
Then any two $q$-Lie bracketings of the word $a_1\cdots a_n$ are equal.
\end{prop}
\begin{proof}
Equation \eqref{eq:comm4usedlaterdonotdelete} with $f=g=q$ shows that if $[x,z]_1 = 0$ then
\begin{equation}\label{eq:treecomm}
[x,[y,z]_q]_q = [[x,y]_q,z]_q\conj{when} [x,z]_1=0.
\end{equation}

Now the assumption $[a_i,a_j]_1=0$ when $\vnp{i-j} \geq 2$ implies that if
$x$ is any Lie bracketing of the word $a_i\cdots a_k$ and $z$ is any Lie
bracketing of the word $a_{k+2}\cdots a_\ell$ then $[x,z]_1=0$ since each
monomial in $x$ commutes with each monomial in $z$.

Any Lie bracketing of the word $a_1\cdots a_n$ can be related to any
other by a sequence of reassociations:
$$[x,[y,z]_q]_q \leftrightarrow [[x,y]_q,z]_q$$
which change the order in which the $q$-bracket is applied. The previous
paragraph implies that the assumption of Equation \eqref{eq:treecomm} holds
for any such reassociation and Equation \eqref{eq:treecomm} shows that
performing this operation doesn't change the resulting element in $A$.
\end{proof}

The corollary below follows from the proposition above and the relation
\eqref{h1:itm}.

\begin{cor}\label{assoccor}
  Any two $q$-Lie bracketings of the word
  $\z_{b-1,n} \z_{b-2,n} \cdots \z_{a+1,n} \z_{a,n}$ are equal in
  $\tDHa(A_{m-1})$, where $1\leq a < b\leq m$.
\end{cor}

\begin{defn}\label{usefulcor}
In the algebra $\tDHa(A_{m-1})$, there are elements $z_{(a,b),n}$ defined by
$$\z_{(a,b),n} := [\z_{b-1,n},\z_{b-2,n},\ldots, \z_{a+1,n}, \z_{a,n}]_q \conj{ when } 1 \leq a < b \leq m.$$
\end{defn}

The corollary also implies the proposition below. 

\begin{prop}\label{assocprop}
 The identity below holds in the algebra $\tDHa(A_{m-1})$.
 \begin{description}
   \item[(S0)\namedlabel{s0:itm}{\lab{S0}}] Finger relation:
   $$[\z_{(b,c),n}, \z_{(a,b),n}]_q = \z_{(a,c),n}\conj{ for } 1\leq a < b < c\leq m.$$
 \end{description}
\end{prop}
\begin{proof}
The two sides of the equation differ by reassociation, see Corollary \ref{assoccor}.
\end{proof}  

This is called the finger relation because it states that if you push the
interior of a Lagrangian arc into a boundary point then it can be written as
a commutator of the two arcs into which it splits.

The theorem below establishes the most important relations satisfied by the
elements $z_{(a,b),n}$. In the image of the functor $\phi_i$ from Prop. \ref{dereqprop},
$z_{(a,b),n}$ becomes a Lagrangian arc between two marked intervals of the
disk. From this perspective, the relations \eqref{s2:itm} and \eqref{s3:itm}
state that the arcs must commute when they are disjoint. Relation \eqref{s1:itm}
states that two arcs which intersect admit a surgery and the relation in
Prop. \ref{assocprop} above corresponds to the observation that this surgery
is associative.

\begin{theorem}\label{relsthm}
Assume that $1\leq a<b<c<d \leq m$. Then the relations listed below hold among the elements $\z_{(a,b),n}$ in the algebra $\tDHa(A_{m-1})$.
  \begin{description}
  \item[(S1)\namedlabel{s1:itm}{\lab{S1}}] Skein relation:
$$ [\z_{(a,c),n}, \z_{(b,c),n-1}]_q = \z_{(a,b),n},$$
\item[(S2)\namedlabel{s2:itm}{\lab{S2}}] Interwoven commutativity:
$$[\z_{(a,d),n}, \z_{(b,c),n-1}]_1 = 0, $$ 
\item[(S3)\namedlabel{s3:itm}{\lab{S3}}] Distant commutativity:
$$[\z_{(a,b),n}, \z_{(c,d),k}]_1 = 0.$$
  \end{description}
\end{theorem}

Before proving the theorem, two lemmas are introduced. The first will serve as a base case for the inductive proofs of \eqref{s1:itm} and \eqref{eq:skeinleft}. The second lemma concerns some useful special cases of the omni-Jacobi identity \eqref{oj:itm}.
\begin{notation}\label{gradingnote}
  In many of the proofs notation is simplified by taking one of the gradings
  to be zero. Any other form of the relation can be obtained from
  suspension.
\end{notation}

\begin{lemma}\label{bases1lem}
When $1\leq a < m$, the two relations below hold in $\tDHa(A_{m-1})$.
\begin{align}
[[z_{a+1,n},z_{a,n}]_q,z_{a+1,n-1}]_q &= z_{a,n} \label{eq:S1a-base}\\
[z_{a,n+1},[z_{a+1,n},z_{a,n}]_q]_q &= z_{a+1,n} \label{eq:S1b-base}
\end{align}
\end{lemma}
\begin{proof}
  To prove the first equation, we use the omni-Jacobi relation $\eqref{oj:itm}[q,q^2,q^{-2}]$:
  \begin{equation}
[x,[y,z]_{q^{-1}}]_{q^3} + q[z,[x,y]_q]_{q^{-1}} + q^3 [ y,[z,x]_{q^{-2}}]_{q^{-2}} = 0.
    \label{eq:j3}
  \end{equation}
Using anti-symmetry \eqref{as:itm} to change the order of brackets in the second summand and setting $x = z_{a+1,0}$, $y = z_{a,0}$ and $z = z_{a+1,-1}$ shows that the left-hand side of Eqn. \eqref{eq:S1a-base} is equal to:
\begin{align*}
&= [z_{a+1,0},[z_{a+1,0},z_{a+1,-1}]_{q^{-1}}]_{q^3} + q^3 [z_{a,0}, [z_{a+1,-1}, z_{a+1,0}]_{q^{-2}} ]_{q^{-2}}\\ 
&= 0 + q^3 [z_{a,0}, \frac{q^{-1}}{q^2-1} ]_{q^{-2}} & \eqref{h2:itm}\\
&= z_{a,0}
\end{align*}

The second equation is argued similarly. Setting $q$ to $q^{-1}$ in Eqn. \eqref{eq:j3} and applying \eqref{as:itm} gives:
$[z,[y,x]_q]_q = q^2 [x,[y,z]_q]_{q^{-3}} + q^{-1} [ y, [z,x]_{q^2} ]_{q^2}$.
After setting $z = z_{a,1}$, $y = z_{a+1,0}$ and $x = z_{a,0}$ and applying \eqref{h2:itm}, we find that $0 + q^{-1}[z_{a+1,0}, -q^{2}\frac{q^{-1}}{q^2-1} ]_{q^2} = z_{a+1,0}$.\\
\end{proof}

\begin{proof}{(of Thm. \ref{relsthm})}
The proof of the theorem requires four steps. The first step is to
note that relation \eqref{s3:itm} follows from \eqref{h3:itm}, see
Lem. \ref{lemma:qcomm}. The second step is to prove \eqref{s1:itm} while
assuming that $c-b=1$. In Step 3, \eqref{s2:itm} is established. In the last
step, \eqref{s1:itm} is established in complete generality.

\noindent {\it Step 2}: Assuming $c-b=1$, we prove \eqref{s1:itm} by
induction on the quantity $c-a$. Note that $c-b=1$ implies $z_{(b,c),0} = z_{b,0}$. 
The base case is $c-a = 2$ which is \eqref{eq:S1a-base}. For the inductive step:
\begin{align*}
[z_{(a-1,c),1},z_{b,0}]_q &= [[z_{(a,c),1},z_{a-1,1}]_q,z_{b,0}]_q & \eqref{orderprop}\\
&= [[z_{(a,c),1},z_{b,0}]_q,z_{a-1,1}]_q&\eqref{eq:comm1}\,\&\,\eqref{s3:itm}\\
&= [z_{(a,b),1},z_{a-1,1}]_q&\textrm{(induction)}\\
&= z_{(a-1,b),1}&\eqref{orderprop}
\end{align*}

\noindent {\it Step 3}: We prove \eqref{s2:itm}. First assume $c-b=1$ so that $z_{(b,c),0} = z_{b,0}$. It follows that
\begin{align*}
[z_{(a,d),1},z_{b,0}]_1 &= -[z_{b,0},z_{(a,d),1}]_1 &\eqref{as:itm}\\
&= -[z_{b,0},[z_{(b+1,d),1},z_{(a,b+1),1}]_q]_1 & \eqref{s0:itm}\\
&= [z_{(a,b+1),1},[z_{b,0},z_{(b+1,d),1}]_q]_1 + [z_{(b+1,d),1},[z_{(a,b+1),1},z_{b,0}]_q]_1& \textrm{$\eqref{oj:itm}\lbrack 1,1,q\rbrack$}\\
&= [z_{(a,b+1),1},[z_{b,0},z_{(b+1,d),1}]_q]_1 + [z_{(b+1,d),1},z_{(a,b),1}]_1&\textrm{(Step 1)}\\
&= [z_{(a,b+1),1},[z_{b,0},z_{(b+1,d),1}]_q]_1 + 0& \eqref{s3:itm}\\
&= [z_{(a,b+1),1},[z_{b,0},[z_{(b+2,d),1},z_{b+1,1}]_q]_q]_1 &  \eqref{orderprop}\\
&= [z_{(a,b+1),1},[z_{(b+2,d),1},[z_{b,0},z_{b+1,1}]_q]_q]_1& \eqref{eq:comm2}\\
&= [z_{(a,b+1),1},[z_{(b+2,d),1},q^{-1}/(q^2-1)]_q]_1  & \eqref{s3:itm}\\
&=0
\end{align*}
The general case for \eqref{s2:itm} is implied by writing $z_{(b,c),0} = [z_{(b+1,c),0},z_{b,0}]_q$ and observing that both terms in the commutator commute with $z_{(a,d),1}$ by induction.

\noindent {\it Step 4}: The general case of \eqref{s1:itm} follows by the induction on the number $c-b$ below.
\begin{align*}
[z_{(a,c),1},z_{(b-1,c),0}]_q &= [z_{(a,c),1},[z_{(b,c),0},z_{b-1,0}]_q]_q & \eqref{orderprop}\\
&=-q[z_{(a,c),1},[z_{b-1,0},z_{(b,c),0}]_{q^{-1}}]_q & \eqref{as:itm}\\
&= -q[z_{b-1,0},[z_{(a,c),1},z_{(b,c),0}]_q]_{q^{-1}}&\textrm{\eqref{eq:comm2} \& \eqref{s2:itm}}\\
&=-q[z_{b-1,0},z_{(a,b)_1}]_{q^{-1}}&\textrm{(induction)}\\
&=[z_{(a,b),1},z_{b-1,0}]_q &\eqref{as:itm}   \\
&= z_{(a,b-1),1}&\textrm{(induction)}
\end{align*}
\end{proof}

The relation \eqref{s1:itm} generalizes Eqn. \eqref{eq:S1a-base} in Lem. \ref{bases1lem}. The proposition below generalizes the second relation in this lemma.

\begin{prop}\label{gensk1}
In the algebra $\tDHa(A_{m-1})$ the following relation holds.
 \begin{description}
   \item[(S1')\namedlabel{eq:skeinleft}{\lab{S1'}}] Opposite skein relation:
$$[\z_{(a,b),n+1},\z_{(a,c),n}]_q = \z_{(b,c),n} \conj{ for } 1\leq a < b < c \leq m.$$
 \end{description}

\end{prop}
\begin{proof}
The proof below will use two specializations of the omni-Jacobi identity \eqref{oj:itm}:
\begin{align*}
[x, [y,z]_q]_{q^{-1}}  + [z, [x,y]_1]_1 + q^{-1} [y, [z,x]_{q}]_q &= 0 
\\
[x,[y,z]_{q^{-1}}]_q + q[y,[z,x]_1 ]_{q^{-2}} + q^{-1} [z,[x,y]_q]_q &= 0 
\end{align*}
The first relation follows from the choices $a=1$, $b= q^{-1}$ and $c = q$,
while the second is obtained by setting $a=q^{-1}$, $b=q^2$ and $c=1$.

  We proceed by induction on the length $\ell = a - b$ of the element $\z_{(a,b),1}$. If $\ell = 1$ then $\z_{(a,b),1} = \z_{a,1}$ and we calculate
  \begin{align*}
    [z_{a,1},z_{(a,c),0}]_q &= [z_{a,1}, [z_{c-1,0}, z_{(a,c-1),0}]_q]_q & \eqref{orderprop}\\
                            &=-q[z_{c-1,0},[z_{(a,c-1),0},z_{a,1}]_{q^{-1}}]_q \\
 &\quad\quad\phantom{=} + q^2 [z_{(a,c-1),0}, [z_{a,1}, z_{c-1,0}]_1]_{q^{-2}} &\eqref{oj:itm}[q^{-1},q^2,1]\\
                              &=-q[z_{c-1,0},[z_{(a,c-1),0},z_{a,1}]_{q^{-1}}]_q + q^2 [z_{(a,c-1),0}, 0]_{q^{-2}} &\textrm{\eqref{h3:itm} \& \eqref{as:itm}}\\
                              &= [z_{c-1,0},z_{(a+1,c-1),0}]_q  &\textrm{\eqref{as:itm} \& \eqref{orderprop}}\\
    &=z_{(a+1,c),n}
    \end{align*}

    When the length, $\ell = a - b > 1$, we proceed again by induction.
    \begin{align*}
      [\z_{(a,b),1},\z_{(a,c),0}]_q &= [[z_{b,1},\z_{(a,b-1),1}]_q,\z_{(a,c),0}]_q & \eqref{orderprop}\\
                                       &= q[z_{(a,b-1),1},[z_{(a,c),0},z_{b,1}]_1]_1 + [z_{b,1},[z_{(a,b-1),1},z_{(a,c),0}]_q]_q & \eqref{oj:itm}[1,q^{-1},q]\\
                                       &= [z_{b,1},z_{(b-1,c),0}]_q + [z_{b,1},0]_q & \eqref{s2:itm}\\
      &= z_{(b,c),0} & \eqref{orderprop}
    \end{align*}
\end{proof}

The proposition below shows that the elements $z_{(a,b),n}$ satisfy \eqref{eq:z2}. This will be used in Lem. \ref{lemma:psihom}.

\begin{proposition}\label{prop:skeinselfext}
The elements $z_{(a,b),n}\in\tDHa(A_{m-1})$ satisfy the relation below.
 \begin{description}
   \item[(S4)\namedlabel{eq:skeinselfext}{\lab{S4}}] Self-skein relation:
$$ [z_{(a,b),n}, z_{(a,b),n+k}]_{q^{2(-1)^{k}}} = \delta_{k,1} \frac{q^{-1}}{q^2-1} \conj{ for } k \geq 1 \conj{ and }1\leq a < b \leq m.$$
 \end{description}

\end{proposition}
\begin{proof}
When $k=1$ and $b-a = 1$ the statement is equivalent to Eqn. \eqref{eq:z2}. If  $b-a > 1$ then the expression $[z_{(a,b),0},z_{(a,b),1}]_{q^{-2}}$ is equal to
\begin{align*}
&=   
-q^{-2}[z_{(a,b),1},z_{(a,b),0}]_{q^2} & \eqref{as:itm}\\
&= -q^{-2}[z_{(a,b),1},[z_{(a+1,b),0},z_{a,0}]_q]_{q^2}
& \eqref{orderprop}\\
&= -q^{-2}(-q^2[z_{a,0}, [z_{(a,b),1}, z_{(a+1,b),0}] - q^2[z_{(a+1,b),0}, [z_{a,0}, z_{(a,b),1}]_{q^{-1}}]_1)
& \textrm{$\eqref{oj:itm}[q^2,1,q^{-1}]$}\\
&= [z_{a,0},z_{a,1}]_{q^{-2}} + [z_{(a+1,b),1},[z_{a,0},z_{(a,b),1}]_{q^{-1}}]_1
&\eqref{s1:itm}\\
&= \frac{q^{-1}}{q^2-1} - q^{-1}[z_{(a+1,b),1},z_{(a+1,b),1}]_1 
&\textrm{\eqref{eq:z2} \& \eqref{eq:skeinleft}}\\
&= \frac{q^{-1}}{q^2-1}
\end{align*}

When $k \geq 2$, each monomial in $z_{(a,b),0}$ $q$-commutes with each monomial in $z_{(a,b),k}$. In particular,
 $[z_{i,0}, z_{j,k}]_{q^{(-1)^k a_{ij}}} = 0$ where the $A_{m-1} = (a_{ij})$ is the Cartan matrix, see Rmk. \ref{dhaamrmk}. Then Lemma \ref{lemma:qcomm} implies $[z_{(a,b),0},z_{(a,b),k}]_{q^{2(-1)^k}} = 0$ since $\sum_{i,j} a_{ij} = 2$.

\end{proof}

\vskip .25in

\section{Fukaya categories}\label{fuksec}

This section reviews the definition of the topological Fukaya categories; we
follow \cite{HKK}. The first part of Section \ref{surfsec} introduces
gradings on surfaces and curves. The second part of the section recalls
marked surfaces and arc systems. These two ideas are combined in Section
\ref{ainfsec} in order to define the Fukaya category $\F(S,A)$ of a surface
and its associated derived category $D^\pi\F(S,A)$. The last section
discusses the special case of the disk. The interested reader may consult
the references \cite{DK, Bok, Nadler}.

For convenience, all of the surfaces are assumed to be compact and oriented in what follows.

\subsection{Graded topology and marked surfaces}\label{surfsec}
The section begins with graded surfaces and graded curves. The key point is
that an intersection point of any two graded curves comes with an integer,
$i_p(c_1,c_2)$, called the intersection index. A shift operator $\s$, which
acts on graded curves is introduced, marked surfaces are introduced as
graded surfaces with corners, and arc systems on marked surfaces are
discussed.

If $X$ is a space then denote by $\Pi_1(X,x,y)$ the set of homotopy classes
of paths from $x$ to $y$ in $X$. There is a composition
$(\a,\b) \mapsto \a\cdot \b$ given by concatenation of paths.

\subsubsection{Graded surfaces and graded curves}\label{gradedsurfsec}
A {\em grading} $\eta\in \Ga(\PP(TS))$ of an oriented surface $S$ is a section
of the projectivized tangent bundle $\PP(TS)$. As $\eta_p\in\PP(T_pS)$ is a
line in the tangent space $T_pS$ at each point $p\in S$, a grading of $S$ is
a foliation of $S$ by lines.  A {\em graded surface} is a pair $(S,\eta)$
consisting of a oriented surface and a grading on $S$.

A {\em map $(S,\eta_S) \to (T,\eta_T)$ of graded surfaces} is a pair $(f,\tilde{f})$
where $f : S \to T$ is an orientation preserving local diffeomorphism and
$\tilde{f}$ is a homotopy from $f^*\eta_T$ to $\eta_S$; i.e.  
$\tilde{f}\in\Pi_1(\Ga(\PP(TS)), f^*\eta_T, \eta_S)$. If $(f,\tilde{f}) : S\to T$ and 
$(g,\tilde{g}) : R\to S$ are maps of graded surfaces then there is a composite map
$(f\circ g, (g^*\tilde{f})\cdot \tilde{g}) : R\to T$. 
There is a {\em shift automorphism} $\s = (1_S, \s) : S \to S$ where $\s$ is the generator of $\pi_1(\PP(T_pS))$ determined by the orientation at each point $p\in S$, see \cite[\S 2.1]{HKK}. 

A {\em graded curve} $c$ in a graded surface $(S,\eta)$ is a
triple $(I,c,\tilde{c})$ in which $I$ is a 1-manifold, $c : I\to S$ is an
immersion and $\tilde{c}$ is a path from $\eta_{c(t)}$ to the tangent line $\dot{c}(t)$
for each $t\in I$; i.e. $\tilde{c} \in \Pi_1(\Ga(c^*\PP(TS)),c^*\eta, \dot{c})$.

If $(I,c,\tilde{c})$ is a graded curve in a graded surface $(S,\eta)$ and $(f,\tilde{f}) : (S,\eta) \to (T,\rho)$ is a map of graded surfaces then the {\em pushforward $f_*(c)$ of $c$ by $f$} is the graded curve $(I, f\circ c, (c^*\tilde{f})\cdot \tilde{c})$ in $T$.
If $c$ is a graded curve in a graded surface then the {\em $n$-fold suspension of $c$} is $\s^n c = c[n] = (\s^n)_*(c)$ for $n\in \ZZ$.

If two graded curves $(I_1,c_1,\tilde{c}_2)$ and $(I_2,c_2,\tilde{c}_2)$ intersect at a point, $c_1(t_0) = p = c_2(t_1)$, then there is an {\em intersection index}:
\begin{equation}\label{indexeqn}
  i_p(c_1,c_2) := \tilde{c}_1(t_1) \cdot \kappa \cdot \tilde{c}_2(t_2)^{-1} \in \pi_1(\PP(T_pS))
  \end{equation}
where $\kappa$ is the shortest counterclockwise path from $\dot{c}_1(t_1)$ to $\dot{c}_2(t_2)$.
The intersection index satisfies the equations:
\begin{equation}\label{indexpropeqn}
  i_p(c_1,c_2) + i_p(c_2,c_1) = 1\conj{and} i_p(c_1[n],c_2[m]) = i_p(c_1,c_2) + n -m.
  \end{equation}

\subsubsection{Marked surfaces}\label{markedsec}

When a surface $S$ has corners, the smooth part of its boundary $\partial S = \partial_0 S \sqcup \partial_1 S$ is a disjoint union of closed
$1$-manifolds, $\partial_0 S$, and open intervals, $\partial_1 S$.
A {\em marked
  surface} is a surface $S$ with corners together with a subset
$M\subset\partial S$ of the boundary which contains each closed component,
$\partial_0S \subset M$, and every other component of $\partial_1 S$.

A marked surface is {\em finitary} when it is compact, has boundary and
there is at least one marked interval on each boundary component. 
Proposition \ref{finiteimpliesfiniteprop} shows that
finitary surfaces have finitary Fukaya categories in
the sense of Definition \ref{finitenessdef}.

An {\em arc} in a marked surface $S$ is a closed embedded interval which
intersects $M$ transversely at its endpoints and is not isotopic to an
interval in $M$.  Arcs are considered up to isotopies in which endpoints
may move within respective components of $M$.  A {\em boundary arc} is an
arc which is isotopic to the closure of a component of $\partial S
\backslash M$. An {\em internal arc} is an arc which is not a boundary
arc. 

An {\em arc system} $A$ in $S$ is a collection of pairwise disjoint
non-isotopic arcs. A {\em full arc system} is an arc system $A$ containing
all boundary arcs that cuts $S$ into a collection of disks.  A {\em map $f :
  (S,M) \to (T,N)$ of marked surfaces} is an orientation preserving
immersion which satisfies $f(M)\subset N$ and maps the boundary arcs of $S$
to disjoint non-isotopic arcs in $T$. 
Note that such maps are not necessarily
closed under composition. 
If $f$ also takes arcs in an arc system for $S$ to arcs in an arc system for $T$
then $f$ induces a strict $A_\infty$-functor
between associated Fukaya categories. For further discussion see \cite[\S 3]{HKK}.

\begin{center}\label{fig:rhcpdisk}
\begin{overpic}[scale=0.6] 
{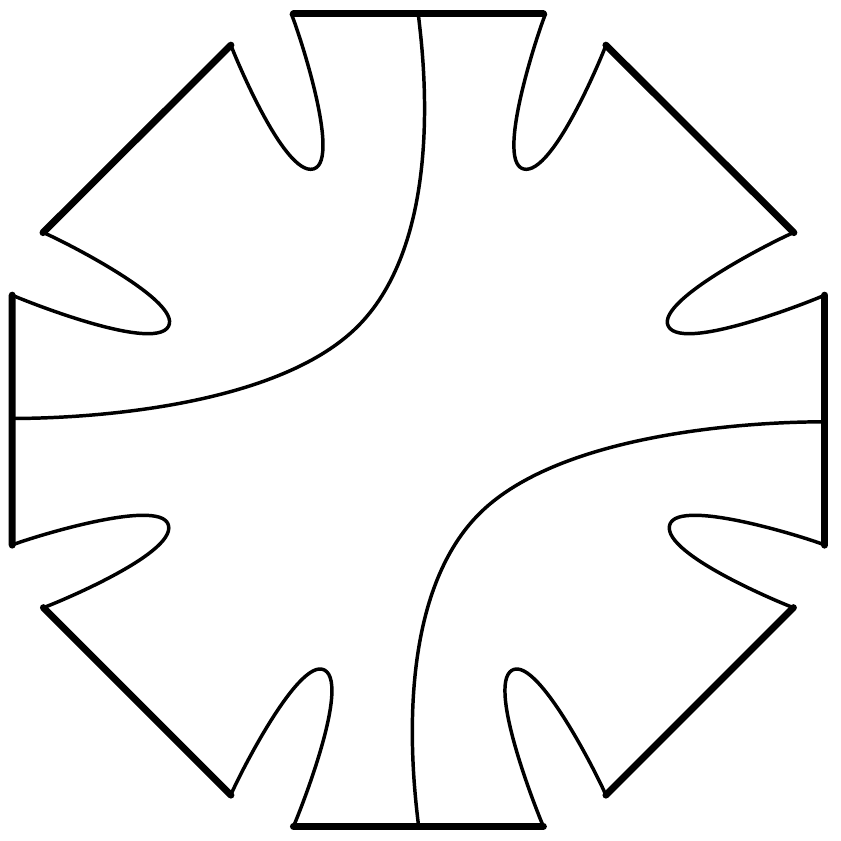}
\end{overpic}
\end{center}

The illustration above features a marked disk $D^2$.  Thin lines are used to
denote the arc system $A$. This arc system consists of eight boundary arcs
and two internal arcs. The thick lines can be thought of as marked intervals or
boundary paths, see Def. \ref{fukcatdef}.

\subsection{$A_\infty$-categories and the Fukaya category}\label{ainfsec}
The purpose of this section is to introduce an $A_\infty$-category $\F(S,A)$
associated to a graded marked surface $S$  with an arc system
$A$. This $A_\infty$-category determines a triangulated category
$D^\pi\F(S,A)$. The convention is to refer to both $\F(S,A)$
and $D^\pi\F(S,A)$ as Fukaya categories. The central concern of this paper
is to study the Hall algebra of the triangulated category
$D^\pi\F(S,A)$.  

\subsubsection{The Fukaya category}
\begin{defn}\label{ainfdef}
An {\em $A_{\infty}$-category} $\aC$ consists of a collection of objects $\Ob(\aC)$ and $\ZZ$-graded $k$-module of morphisms $\Hom(X,Y)$ for each pair of objects $X,Y\in \Ob(\aC)$ together with maps
$$\mu_d : \Hom(X_{d-1},X_{d})\ott  \cdots \ott \Hom(X_0,X_1)\to \Hom(X_0,X_d)[2-d], d\geq 1$$
which satisfy the relations
$$\sum_{m=0}^d \sum_{n=0}^{d-m} (-1)^{\ddagger_n} \mu_{d-m+1}(f_d,\ldots,f_{n+m+1},\mu_m(f_{n+m},\ldots,f_{n+1}),f_n,\ldots,f_1) = 0$$
where $\ddagger_n = \vnp{f_n} + \cdots + \vnp{f_1}-n$ and $d \geq 1$.

An $A_\infty$-category $\aC$ is said to be {\em strictly unital} when there is a unique degree zero morphism $1_X \in \Hom^0(X,X)$ for each $X\in \Ob(\aC)$ which satisfies:
\begin{align}\label{identityeqn}
\mu_2(f,1_X) = f, \quad\quad (-1)^{\vnp{g}}\mu_2(1_X,g) = g\\ \conj{ and }\mu_d(\ldots, 1_X, \ldots) = 0 \conj{ when } d \ne 2,\nonumber
\end{align}
for any maps $f : A \to X$ or $g : X\to B$ and any object $X\in\Ob(\aC)$.
\end{defn}

\begin{defn}{$(\F(S,A))$}\label{fukcatdef}
If $S$ is an oriented graded marked surface and $A$ is an arc system then there is an $\Ainf$-category $\F(S,A)$ with objects $\Ob(\F(S,A))=A$ given by the set of graded arcs in $A$. The morphisms in $\F(S,A)$ are $k$-linear combinations of boundary paths. 

Given two distinct arcs $X$ and $Y$ in $A$, a {\em boundary path} from $X$ to $Y$ is a non-constant path in $A$ which starts on an endpoint of $X$, follows the reverse orientation of the boundary and ends on an endpoint $Y$. When $X$ and $Y$ coincide, the trivial path $1_X$ is considered a boundary path. The {\em degree} of a boundary path $\gamma : [0,1] \to S$ from $X$ to $Y$ is given by
$$\vnp \ga := i_{\ga(0)}(X,\ga) - i_{\ga(1)}(Y,\ga)$$
for any grading of $\ga$. 

The $\Ainf$-structure on $\F(S,A)$ is defined below.
\begin{description}
\item[$(\mu_1)$] The map $\mu_1$ is always zero.
\item[$(\mu_2)$] The map $\mu_2$ is given by concatenation of boundary paths: if $a$ and $b$ can be concatenated 
then 
$$\mu_2(b,a) := (-1)^{\vnp a} a\cdot b,$$ 
otherwise, $\mu_2(b,a) := 0$.
\item[$(\mu_m)$] Suppose that $(D^2,\A)$ is a disk with $m$ marked intervals and $m \geq 3$. Let $\A$ be the boundary arcs and $\{c_1,\ldots,c_m\}$ the boundary paths between them ordered cyclically according to the reverse disk orientation. Then a {\em disk sequence} is a collection of boundary paths 
$\{f\circ c_1,\ldots,f \circ c_m\}$ in $(S,A)$ for some map $f : (D^2,\A) \to (S,A)$ of marked surfaces. 

If $a_1,\ldots,a_m$ is a disk sequence and $b$ is a boundary path then 
$$\mu_m(a_m,\ldots,a_1\cdot b) := (-1)^{\vnp b} b \conj{ or } \mu_m(b\cdot a_m,\ldots, a_1) := b$$
when $a_1\cdot b\ne 0$ or $b\cdot a_m\ne 0$ respectively; otherwise the map $\mu_m$ is defined to be zero.
\end{description}
\end{defn}

The definition above determines an $\Ainf$-category, see \cite[Prop. 3.1]{HKK}.

\subsubsection{The split-closed derived Fukaya category}
The remainder of this section is concerned with the construction of a
triangulated category $D^\pi (\aC)$ associated to an $A_\infty$-category
$\aC$. This category $D^\pi(\aC)$ is the homotopy category of the
$A_\infty$-category $\Pi\Tw(\Si\aC)$ of split-closed twisted complexes in the
additivization of $\aC$. Definition \ref{adddef} concerns the additiviation $\Si\aC$, Definition \ref{twdef} involves twisted complexes 
and Proposition \ref{karoubiprop} discusses split-closure.

The additivization of $\aC$ is determined by adding formal direct sums and grading shifts.

\begin{defn}{$(\Si\aC)$}\label{adddef}
  The additivization $\Si\aC$ of an $\Ainf$-category $\aC$ has objects $\Ob(\Si(\aC))$ given by formal direct sums:  $\oplus_{i=1}^n V_i \ott X_i$   where $V_i$ is a $\ZZ$-graded $k$-module and $X_i\in\Ob(\aC)$.
  Maps between formal direct sums are given by matrices of maps between summands:
  $$\Hom_{\Si\aC}(V\ott X, W\ott Y) := \Hom(V,W)\ott \Hom(X,Y).$$
  The $\Ainf$-structure is extended by:
  $$\mu_k(f_k\ott a_k, \ldots, f_1\ott a_1) = (-1)^{\sum_{i<j} \vnp{f_i}(\vnp{a_j}-1)}f_k\cdots f_1\ott\mu_k(a_k,\ldots, a_1).$$
\end{defn}  
 
The category of twisted complexes $\Tw(\aC)$ is determined by adding representatives for all mapping cones to the category $\aC$.

\begin{defn}{$(\Tw(\aC))$}\label{twdef}
A {\em twisted complex} $(X,\d_X)$ is a finite sum $X = \oplus_{i=1}^n X_i$ of formally graded objects in $\aC$ along with a strictly lower triangular matrix $\d_X = (\d_{ij})_{1\leq i,j \leq n}$, $\d_{ij} \in \Hom^1(X_j, X_i)$ that satisfies the Maurer-Cartan equation:
$$\sum_{n\geq 0} \mu_n(\d_X,\ldots,\d_X) = 0.$$
Maps between twisted complexes are given by matrices of maps between various summands:
$$\Hom_{\Tw(\aC)}((X,\d_X),(Y,\d_Y)) := \oplus_{i,j} \Hom(X_i,Y_j).$$
If $(X_0,\d_{X_0}), \ldots, (X_d,\d_{X_d})$ is a sequence of twisted complexes and $f_i = (f_i^{jk})$ where $(f_i^{jk}) \in \Hom_{\Tw(\aC)}((X_{i-1},\d_{X_{i-1}}), (X_i, \d_{X_i}))$ then the map $\mu_d^{\Tw(\aC)}$ is defined by
\begin{align*}
  \mu_d^{\Tw(\aC)}(f_d,\ldots,f_1) = \sum_{j_d,\ldots,j_0\geq 0} \mu_{d+j_0+\cdots+j_d}(\d_{X_d},\ldots,\d_{X_d}, f_d, \d_{X_{d-1}},\\\d_{X_{d-1}},\ldots,\d_{X_{d-1}},f_{d-1},\ldots,f_1, \d_{X_0},\ldots,\d_{X_0})
\end{align*}
where each group of $\d_{X_i}$ factors appear $j_i$ times. This determines the $A_\infty$-category $\Tw(\aC)$ of twisted complexes.
\end{defn}

The objects of $\Tw(\aC)$ are, not necessarily uniquely, built out of successive applications of the mapping cone construction.

In an $A_\infty$-category $\aC$, there is a degree $1$ map, $\mu_1 : \Hom(X_0,X_1)\to \Hom(X_0,X_1)$, for each pair of objects $X_0,X_1\in \Ob(\aC)$, which satisfies $\mu_1\circ\mu_1 = 0$; the simplest $A_\infty$-relation above. Taking homology everywhere with respect to these maps produces the homotopy category defined below.

\begin{defn}{$(\Ho(\aC))$}\label{hodef}
The {\em homotopy category} $\Ho(\aC)$ of an $A_\infty$-category $\aC$ is the $k$-linear category with the same objects as $\aC$ and morphisms given by homology classes of maps $[f] \in H^*(\Hom(X,Y),\mu_1)$ for each $X,Y\in \Ob(\aC)$. The composition is defined by
$$[f_2]\circ [f_1] := (-1)^{\vnp{f_1}} [\mu_2(f_2,f_1)].$$
\end{defn}

Although the homotopy category of the category of twisted complexes is a
triangulated category, in order to satisfy the technical assumptions made by
the derived Hall construction, it is necessary to take the split-closure
$\Pi\Tw(\aC)$ of the category of twisted complexes, see
Prop. \ref{finiteimpliesfiniteprop}. This adds to an $\Ainf$-category
objects representing the images of homotopy idempotents.

\begin{prop}\label{karoubiprop}
If $\aC$ is an $\Ainf$-category then a {\em split-closure $\iota : \aC \to
  \Pi\aC$} is an $\Ainf$-category such that
\begin{enumerate}
\item $\Ho(\Pi\aC)$ is split-closed,
\item $\iota$ induces a full and faithful functor on homotopy categories,
\item every object in the category $\Ho(\Pi \aC)$ is isomorphic to the image of an idempotent endomorphism in $\aC$,
\end{enumerate}
see \cite[I.(4c)]{Seidel}. The split-closure is unique up to homotopy and preserves triangularity. In more detail, every $\Ainf$-category $\aC$ has a split-closure $\aC \to \Pi\aC$. Any two split-closures $\iota_1 : \aC\to \Pi_1\aC$ and $\iota_2 : \aC \to \Pi_2\aC$ for a given $\Ainf$-category $\aC$ are quasi-equivalent via a functor $f : \Pi_1\aC\xto{\sim}\Pi_2\aC$ which satisfies $f\circ \iota_1 \simeq \iota_2$, see \cite[Lem. 4.7]{Seidel}. The split-closure of a triangulated $\Ainf$-category is triangulated, see \cite[Lem. 4.8]{Seidel}. 
\end{prop}

In particular, the last lemma cited above ensures that the homotopy category
$\Ho(\Pi \Tw(\aC))$ is a split-closed triangulated category in the usual
sense. The remark below summarizes two properties of the split-closure from Seidel.

\begin{rmk}\label{splclrmk}
  There is a homologically fully-faithful embedding of the split-closure $Z : \Pi\aC\to \aC\module$ inside of
  the category of $\Ainf$-functors from $\aC$ to chain
  complexes $\Ch_k$. 
If $Y\in\Ob(\aC)$ and $\wp_Y$ is a homotopy idempotent then one can associate a $\aC$-module $Z_Y : \aC\to \Ch_k$, defined on objects by $Z_Y(X) = \Hom_{\aC}(X,Y)[q]$ \cite[p.56]{Seidel}. By \cite[Lem. 4.4]{Seidel}, the homology of this module is given by $H(Z_Y)(X) = p_Y \Hom_{H(\aC)}(X,Y)$ where $p_Y = [\wp^1_Y]$. By \cite[Lem. 4.5]{Seidel}, if $M$ is any $\aC$-module then $\Hom_{H(\aC\module)}(Z_Y,M) = H(M)(Y)p_Y$. 

After setting $M=Z_X$, where $Z_X$ is the Yoneda module  associated to the homotopy idempotent $\wp_X$, the two lemmas combine to show:
\begin{align*}\Hom_{\Ho(\Pi\aC)}(Y,X) &= \Hom_{\Ho(\aC\module)}(Z_Y,Z_X)\\ &= H(Z_X)(Y)p_Y\\ &= p_X\Hom_{\Ho(\aC)}(Y,X)p_Y 
\end{align*}
where $p_X = [\wp_X^1]$. 
\end{rmk}

Sometimes this construction is called the idempotent completion or Karoubi envelope.

\begin{defn}{$(D^\pi(\aC))$}\label{trianote}
  If $\aC$ is an $A_\infty$-category then the {\em split-closed derived category
    $D^\pi(\aC)$ of $\aC$} is the homotopy category of the split-closed
  twisted complexes on the additivization of $\aC$.
  $$ D^\pi(\aC) := \Ho(\Pi\Tw(\Si\aC))$$
\end{defn}

\begin{rmk}
For any $\Ainf$-category $\aC$, there is a category of modules $\aC\module$ into which $\Pi\Tw(\Si\aC)$ embeds homologically fully faithfully. The split-closed derived category $D^\pi(\aC)$ is the smallest split-closed triangulated subcategory of $\Ho(\aC\module)$ containing the image of the Yoneda embedding.
\end{rmk}  

Definition \ref{trianote} above allows us to construct a triangulated category
$D^\pi\F(S,A)$ from the $A_\infty$-category introduced in
Def. \ref{fukcatdef} above. The majority of this paper concerns the
computation of the derived Hall algebras of these categories. The
proposition below verifies that Fukaya categories associated to finitary
surfaces are finitary triangulated categories. For a discussion of finitary
surfaces see Section \ref{markedsec}. See Definition \ref{finitenessdef} for a discussion of 
finitary triangulated categories.

\begin{prop}\label{finiteimpliesfiniteprop}
The derived category $D^\pi\F(S,A)$ of the Fukaya category of a finitary
surface $S$ is a finitary triangulated category.
\end{prop}  
\begin{proof}
  Since $S$ is finitary, there is a marked interval on each boundary
  component so that $\Hom(X,Y)$ in the Fukaya category $\F(S,A)$ is always
  finite dimensional. Maps between objects in $\Si\aC$ are finite
  matrices of maps between objects in $\aC$. 

  If the mapping spaces $\Hom_{\Ho(\aC)}(X,Y)$ are finite dimensional then the mapping spaces $\Hom_{\Ho(\Tw\aC)}(X,Y))$ must be finite dimensional. This is because any twisted complex $X$ is the total complex of a Postnikov system:
$$\tau_1 X \xto{f^X_1} \tau_2 X \xto{f^X_2} \tau_3 X \to \cdots \to \tau_{\ell_X} X = X$$
where the cone $C(f^X_n)\simeq C^X_n\in \Ob(\aC)$. So that if $Y\in \Ob(\aC)$
then, for each $n$, the long exact sequence:
$$\cdots \to \Ext^i(\tau_{n-1} X,Y) \to \Ext^i(\tau_{n} X, Y) \to Ext^i(C^X_{n-1},Y)\to \cdots$$
together with finite dimensionality of $\Ext^*(\tau_{n-1} X,Y)$ and $\Ext^*(C^X_{n-1},Y)$ implies finite dimensionality of $\Ext^*(\tau_{n} X,Y)$. If both $X, Y\in \Tw(\aC)$ then using the long exact sequence:
$$\cdots \from \Ext^i(X,\tau_m Y) \from \Ext^i(X, \tau_{m-1} Y) \from Ext^i(X,C^Y_{m-1})\from \cdots$$
together with finite dimensionality of $\Ext^*(X,\tau_{m-1}Y)$ and $\Ext^*(C^Y_{m-1},Y)$ implies finite dimensionality of $\Ext^*(X, \tau_{m} Y)$. Therefore, $H^*(\Hom_{\Tw\aC}(X,Y))$ is finite dimensional for all $X,Y\in \Ob(\Tw\aC)$.

If $\aC$ is an $\Ainf$-category and $\Ho(\aC)$ has finite dimensional $\Hom$-spaces then $\Ho(\Pi\aC)$ has finite dimensional $\Hom$-spaces since:
$$\dim p_X\Hom_{\Ho(\aC)}(Y,X)p_Y \leq \dim \Hom_{\Ho(\aC)}(Y,X)$$
see Remark \ref{splclrmk}.

  By the above argument $End(X)$ is finite dimensional for each
  $X\in\Ob(D^\pi\F(S,A))$. Finite dimensional $k$-algebras are examples of
  semiperfect rings. Since $D^\pi\F(S,A)$ is split-closed, the category
  $D^\pi\F(S,A)$ is Krull-Schmidt. In particular, when $X$ is indecomposable
  the algebra $\End(X)$ is local, 
  see \cite[Cor. 4.4, Prop. 4.9]{Krause}.

  Lastly, when $S$ is finitary, the Fukaya category is homologically proper
  by \cite[Cor. 3.1]{HKK}.  This stronger condition implies that
  $D^\pi\F(S,A)$ is left-homologically finite.
\end{proof}

\subsection{The Fukaya category of the disk}\label{disksec}

The special case of the disk $(D^2,m)$ with $m$ marked intervals is
discussed. When equipped with a foliation, there is a natural choice of a
minimal arc system $\A_m$ for $(D^2,m)$. Two propositions are stated after the
introduction of some notation. The first proposition states that any one arc
in $\A_m$ is isomorphic to a twisted complex consisting of the other arcs in
$D^\pi\F(D^2,\A_m)$. The second proposition identifies the category
$D^\pi\F(D^2,\A_m)$ with the derived category of $A_{m-1}$-quiver
representations.

\begin{notation}\label{disknotation}
Denote by $(D^2,m)$ the oriented disk $D^2$ with $m$ marked intervals equipped
with minimal full arc system $\A_m$ consisting only of boundary
arcs. Following the reverse orientation induces a cyclic ordering of the
boundary arcs so that $\A_m = \{E_i\}_{i\in \ZZ/m}$; the labelling
$i\in\ZZ/m$ of the arc $\E_i$ from the $i$th to the $i+1$st marked interval
is unique up to cyclic shift. The group $\ZZ = \inp{\s}$ acts on $\A_m$ by
suspension and so we set $\E_{i,n} = \s^n\E_i$ and $\ZZ \A_m = \{ \E_{i,n} : i\in\ZZ/m, n\in\ZZ\}$.  The subscript $m$ will be dropped from $\A_m$ when
it can be inferred by other means.

The Fukaya category uses the foliation $\eta$ to determine the degree $\vnp{a_i}$ of the unique map $a_i : \E_i \to \E_{i+1}$ for each $i\in\ZZ/m$. This information can be recorded more simply as a map $\h : \A_m \to \ZZ$, $\h(\E_i) = \vnp {a_i} = i_p(E_i,E_{i+1})$, which satisfies the equation
$$\sum_{\E_i\in\A_m} \h(\E_i) = m - 2.$$
Any such map will be called {\em foliation data}. We will consistently abbreviate notation by writing $\h(\E_i)=\h(i)$. 

\begin{remark}
For this equation to hold, it is essential for the disk to be a manifold with corners, see Section \ref{markedsec}. If a disk is realized as a subset of the plane so that each edge is a straight line then it must be a convex polygon.
\end{remark}

There is also a function $\inp{-} : \A_m \to \ZZ$ which will be used to  express the shift of $j$th arc, $\E_j$, which occurs when expressing the $i$th arc, $\E_i$, in terms of the other arcs: $\E_{i+1}, \E_{i+2}, \ldots, \E_{i+m-1}$, see Prop \ref{convprop}.
\begin{equation}\label{knotation}
\inp{k} := \sum_{j=1}^{k-1} (1-h(j))
\end{equation}
In this notation, $\inp 1 = 0$. When we wish to emphasize the dependence of the function $\inp{-} $ on the function $\h$, we will write $\inp{-}_\h$.

The permutation $\tau = (1,2,\ldots,m)$ acts on the arc system  $\A_m$ by $\tau^i(j) = i+j$. 
This induces an action of $\tau$ on the set of $\ZZ$-linear combinations of functions $\Hom(\A_m,\ZZ)$. 
Notice that there is an inclusion $\ZZ\hookrightarrow \Hom(\A_m,\ZZ)$ determined by taking an integer $n$ to the constant function on $\A_m$ with value $n$.
For example, $\tau^i(\h(j)) = h(i+j)$ and  $\tau^i\inp{k} = \sum (1-h(i+j))$. 
The action of $\tau$ is further extended to elements $E_{i,f(m)}$, where $f$ is a $\ZZ$-linear combination of functions on $\A_m$. For example,
$$\tau^i(E_{j,3 + \inp j})=E_{i + j,3 + \tau^i \inp j}.$$
\end{notation}

The two propositions below concern basic properties of the Fukaya categories
$\F(D^2,\A)$.  The first proposition says that any arc can be recovered from
the others in the Fukaya category of the disk.

\begin{prop}\label{convprop}
  Suppose that $\E_i\in \A=\{E_i\}_{i\in\ZZ/m}$ is an arc in a minimial arc system for the
  disk $(D^2,m)$ equipped with foliation data $\h : \A \to\ZZ$. Then the arc
  $\E_i$ is isomorphic to the following twisted complex of the remaining arcs
  $\A\backslash \{\E_i\}$ in the Fukaya category $\dfda$:
$$\E_{i+1,\tau^i\inp{1}} \xto{a_{i+1}} \E_{i+2,\tau^i\inp{2}} \xto{a_{i+2}} \E_{i+3,\tau^i\inp{3}} \to \cdots \xto{a_{i+m-2}} \E_{i+m-1,\tau^i\inp{m-1}}.$$
The maps $a_i : \E_{i} \to E_{i+1,\tau^i\inp{1}}$ and $a_{i+n-1} : E_{i+m-1,\tau^i\inp{m-1}}\to E_i$ establish this isomorphism, see \cite[\S 3.3]{HKK}.
\end{prop}

The proposition above states that the arc $E_i$ can be expressed as iterated
cones on the arcs $\A\backslash\{E_i\}$. The objects $E_j$ in subcategory of
$\F(D^2,\A)$ formed by the remaining arcs $\A\backslash\{E_i\}$ can be
identified with the simple modules $z_j$ associated to vertices of the
$A_{m-1}$-quiver. The next proposition states that this identification
induces an exact equivalence of categories.

\begin{prop}\label{dereqprop}
  Suppose that $k$ is a field and $\dfda$ is the $k$-linear Fukaya category of the disk with $m$ marked intervals equipped with minimal arc system $\A=\{\E_i\}_{i\in\ZZ/m}$ and foliation data  $\h : \A\to \ZZ$. If $\E_i\in A$ is a fixed arc then there is a exact equivalence:
$$\phi_i : D^b(\Rep_k(A_{m-1}))\to \dfda$$
determined by taking the simple module associated to the $j^{\textrm{th}}$ vertex $z_j$ to the arc $\E_{i+j,\tau^i\inp{j}}$.
\end{prop}
For more detailed discussion, see \cite[Prop. 6.7]{Nadler} or \cite[\S 6.2]{HKK}.

\begin{example}
  In the special case of the disk with three points, there is a exact triangle in the Fukaya category $\dfda$ of the form:
  $$\cdots \to E_{1,1} \to E_{2,1-h(1)} \to E_{3,h(3)} \to E_{1,0} \xto{a_1} E_{2,-h(1)} \to \cdots$$
shifting by $\h(1)-1$ and using the relation $\h(1)+\h(2)+\h(3) = 1$ makes the $E_2$ arc a term in degree zero, and shifting again gives the sequence corresponding to $E_3$. The three fundamental distinguished triangles are
\begin{align*}
  E_{2,1-h(1)} \to E_{3,h(3)} \to E_{1,0}, &\quad\quad\quad E_{3,1-h(2)} \to E_{1,h(1)} \to E_{2,0} \\
  &\quad\normaltext{ and } \quad\quad E_{1,1-h(3)} \to E_{2,h(2)} \to E_{3,0}.
\end{align*}  
They are related by iterated applications 
of the axioms (TR1) and
(TR2).

 Corresponding to each triangle, there are exact equivalences
$\phi_{i,n} : \dqt \to \dfda$, $i\in \ZZ/3$ and $n\in\ZZ$ determined by the
assignments
$$ z_{1} \mapsto E_{i,n} \conj{ and } z_{2} \mapsto E_{i+1,1-\h(i) + n}.$$
When $n=0$, we recover $\phi_{i}$ in Prop. \ref{dereqprop} above. 
\end{example}

\vskip .25in
\section{Hall algebras of surfaces}\label{halldisksec}
In this section we construct the Fukaya Hall algebra $\SAlg(S,M)$ of a finitary graded surface $S$ with marked intervals $M$. This algebra is determined by the property that there is a natural isomorphism
$$\Alg(S,A)\xto{\sim}\SAlg(S,M)$$
for any full arc system $A$ of $(S,M)$ where $\Alg(S,A)$ is the composition
subalgebra of the Hall algebra $\tDHa (D^\pi\aF(S,A))$ introduced in
Def. \ref{compalgdef}.  

The remainder of this section studies the Hall algebras of disks
$\Alg(D^2,A)$. In this case the distinction between composition subalgebra
and Hall algebra disappears. The main result is a presentation for
$\Alg(D^2,A)$.  This is accomplished by finding an explicit presentation for
the minimal arc system $\Alg(D^2,\A_m)$ and recovering $\Alg(D^2,A)$ from
gluing copies of the algebras $\Alg(D^2,\A_m)$ together.

\subsection{The Hall algebra of a surface}\label{hallsurfsec}
The section begins by introducing the composition subalgebra $\Alg(S,A)$
associated to a graded oriented finitary surface $S$ with full arc system
$A$. Then functoriality statements are proven from which it follows that
$\Alg(S,A)$ is functorial in $A$.  The Fukaya Hall algebra $\SAlg(S,M)$ is defined as 
the
colimit of these algebras over all possible full arc systems. The section
concludes by establishing functoriality results. In
particular, criteria in which the disk algebra $\SAlg(D^2,m)$ embeds in the
surface algebra $\SAlg(S,M)$ are introduced.

The composition subalgebra $\Alg(S,A)$ is the subalgebra of the derived Hall
algebra generated by all of the arcs in $A$ and their suspensions.

\begin{defn}{($\Alg(S,A)$)}\label{compalgdef}
  Suppose that $(S,A)$ is a finitary oriented graded surface equipped with a
  full arc system $A$. By Prop. \ref{finiteimpliesfiniteprop} and
  Thm. \ref{dhadef} there is a derived Hall algebra of the Fukaya category.
  Let $k$ be the field with $q$ elements. Then there is a canonical
  homomorphism: 
$$\kappa : F(\ZZ A) \to \tDHa(D^\pi\F(S,A))$$ 
from the free   $\QQ(\sqrt{q})$-algebra on the set $\ZZ A = \{ \s^n a : n\in \ZZ, a\in A\}$
  of suspensions of the arcs in $A$ to the derived Hall algebra of the
  Fukaya category $D^\pi\F(S,A)$. The {\em composition subalgebra}
  $\Alg(S,A)$ is the image of $\kappa$:
$$\Alg(S,A) := \frac{F(\ZZ A)}{\ker(\kappa)}$$
\end{defn}  

\begin{remark}
The terminology `composition subalgebra' is not necessarily standard in the derived setting.
\end{remark}

The next proposition shows that maps between marked surfaces which induce full
and faithful functors between Fukaya categories induce embeddings between
composition subalgebras.

\begin{prop}\label{mapprop}
Suppose that $f : (S,M) \to (S',M')$ is a map of marked surfaces equipped
with full arc systems $A$ and $A'$ respectively, and that $f$ takes arcs in $A$ to arcs in $A'$. If the induced functor
$\F(S,A)\to \F(S', A')$ is full and faithful then there is an algebra
monomorphism $f_* : \Alg(S,A)\to \Alg(S',A')$ which is uniquely determined
by the requirement that $f_*(a) = f(a)\in A'$ for all $a\in A$.

Moreover, this assignment is functorial in the sense that if 
$f: S \to S'$, $g: S' \to S''$ and
$g \circ f: S \to S''$ 
is a map of marked surfaces then $(g\circ f)_* = g_* \circ f_*$.
\end{prop}  

\begin{proof}
 Since $\mu_1=0$ in Def. \ref{fukcatdef}, a functor which is full and
 faithful must be homologically full and faithful. Any such functor extends
 to a full and faithful functor between additivizations of categories.
By Seidel \cite[Lem. 3.23 p. 46]{Seidel}, such a functor induces a homologically full and faithful
 functor between associated $A_\infty$-categories of twisted
 complexes. Moreover, a functor $\aC \to \aD$  which is homologically full and faithful
 induces a homologically full and faithful map $\Pi\aC \to \Pi\aD$ between
 split-closures. Using Remark \ref{splclrmk}, the maps 
$$(\Pi f)_{X,Y} : \Hom_{\Ho(\Pi\aC)}(X,Y) \to \Hom_{\Ho(\Pi\aD)}(f(X),f(Y))$$
can be written as 
$$f_{X,Y} : p_Y\Hom_{\Ho(\aC)}(X,Y)p_X \to f_{Y,Y}(p_Y)\Hom_{\Ho(\aC)}(X,Y)f_{X,X}(p_X).$$
The injectivity and surjectivity of $(\Pi f)_{X,Y}$ follow from that of $f_{X,Y}$ for each $X,Y\in\Ob(\Pi\aC)$. It follows that the induced map $D^\pi\F(S,A) \to D^\pi\F(S',A')$
 between homotopy categories is full and faithful. 

 Now the functoriality of the derived Hall algebra construction (see Remark \ref{dhafunrmk}) shows that there is a monomorphism
 $f_\natural : \tDHa (D^\pi\F(S,A))\to \tDHa (D^\pi\F(S',A'))$. By definition $f_\natural(a) = f(a)\in A'$ and $f_\natural$ commutes with the suspension $\s$, so $f_\natural$ restricts to a monomorphism
 $f_* : \Alg(S,A)\to \Alg(S',A')$ between composition subalgebras. 

If $f$ and $g$ are two composable maps which induce homologically full and
faithful functors then the composite $g\circ f$ is necessarily full and
faithful and the induced monomorphism $(g\circ f)_*$ is equal to the
composition $g_*\circ f_*$ because they agree on arcs $a\in A$ and commute with suspension.
\end{proof}

\begin{cor}{(Sheaf property)}\label{sheafpropcor}
  Suppose that $A$ and $A'$ are two full arc systems of a finitary surface $S$ such that $A' = A\backslash\{a\}$ for some arc $a\in A$. Then there are isomorphisms
  $$\a_{A',A} : \Alg(S,A) \rightleftarrows \Alg(S,A') : \b_{A,A'},$$
  which are functorial, in the sense that
  $$\b_{A,A''} = \b_{A,A'}\circ \b_{A',A''}\conj{ and } \a_{A'',A} = \a_{A'',A'}\circ \a_{A',A}$$
  where $A' = A\backslash \{a\}$ and $A'' = A'\backslash \{a'\}$ as above.
\end{cor}  
\begin{proof}
  The tautological map $\iota_{A,A'} : (S,A') \to (S,A)$ of marked surfaces
  which is the identity everywhere and identifies the full arc system $A'$
  within $A$. The inclusion $A'\subset A$ induces a Morita equivalence
  by \cite[Lem. 3.2]{HKK}, so the induced maps
  $\b_{A,A'} = (\iota_{A,A'})_*$ are isomorphisms. The functoriality
  statement for $\b$-maps follows from the associativity of inclusions:
  $(A'' \subset A') \subset A$ is $A'' \subset (A' \subset A)$, and the
  functoriality statement in Proposition \ref{mapprop}. The maps $\a_{A',A}$
  are uniquely determined by the requirement that they be inverse
  isomorphisms: $\a_{A',A} = \b_{A,A'}^{-1}$.
\end{proof}

\begin{defn}\label{modulidef}
  There is a category $\M(S,M)$ associated to a graded marked surface $S$
  with objects consisting of pairs $(S,A)$ of full arc
  systems $A$ parameterizing $S$ and maps $(S,A')\to (S,A)$ when
  $A$ is obtained by adding arcs to $A'$.
\end{defn}

When $S$ is a finitary graded marked surface, Def. \ref{compalgdef}
associates a $\QQ(\sqrt{q})$-algebra $\Alg(S,A)$ to each object $(S,A)$ in
$\M(S,M)$. By Cor. \ref{sheafpropcor}, there is a unique isomorphism of
algebras $\b_{A,A'} : \Alg(D^2,A')\to \Alg(D^2,A)$ associated to each map
$(S,A')\to (S,A)$ in $\M(S,M)$.  When $A' = A$, the map $\b_{A,A} = 1_A$ is
defined to be the identity homomorphism.  The $\b$-maps must satisfy
$\b_{A,A''} = \b_{A,A'} \circ \b_{A',A''}$ by the corollary above.  So these
assignments determine a functor $\M(S,M)\to \QQ(\sqrt{q})-Alg$. The Fukaya Hall
algebra is defined to be the colimit of this functor.

\begin{defn}\label{skdef} 
  If $(S,M)$ is a finitary graded surface with marked intervals $M$ then the
  \emph{Fukaya Hall algebra} $\SAlg(S,M)$ is the colimit of the composition
  subalgebras of the Hall algebras of the Fukaya categories associated to
  each full arc system:
  $$\SAlg(S,M) := \underset{(S,A)\in \M(S,M)}{\colim} \Alg(S,A).$$
\end{defn}

In words, the Hall algebra is generated by suspensions of graded arcs which
occur in any possible subdivision of the surface $S$.  Any two such arcs are
identified by a $\b$-map when they are contained in a common refinement.
By using the limit we have removed the dependence of our algebra on the arc
system. The theorem below captures the intent behind this construction.

\begin{thm}\label{thm:skeiniso}
  For any full 
  arc system $A$ of a finitary graded surface $S$ with marked
  intervals $M$, the natural map
$$\Alg(S,A)\xto{\sim} \SAlg(S,M)$$
from the composition subalgebra $\Alg(S,A)$ of the Hall algebra to the 
Fukaya Hall algebra $\SAlg(S,M)$ of the surface is an isomorphism.
\end{thm}
\begin{proof}
  By construction there is a natural map from the composition subalgebra to
  the Fukaya Hall algebra. This map is an isomorphism since the colimit in
  Definition \ref{skdef} is taken over a family of isomorphisms.
\end{proof}

The functoriality statements above extend to the Fukaya Hall algebras.

\begin{cor}{(of Prop. \ref{mapprop})}\label{skeinfuncor}
Suppose $f : (S,M) \to (S',M')$ is a map of connected marked surfaces endowed with full arc systems $A$ and $A'$ respectively. If $f$ induces a full and faithful functor $\F(S,A)\to \F(S',A')$ between Fukaya categories then $f$ induces a monomorphism $f_* : \SAlg(S,M)\to \SAlg(S',M')$ between Fukaya Hall algebras.
\end{cor}  
\begin{proof}
By Prop. \ref{mapprop}, the map $f$ induces a monomorphism $\Alg(S,A)\to\Alg(S',A')$. Since any two arc decompositions of a connected surface 
 can be connected by a sequence of subdivisions and refinements, this monomorphism can extended uniquely to a monomorphism
$f_* : \Alg(S,B) \to \Alg(S',B')$ for any two arc systems $(S,B)\in \M(S,M)$ and
$(S',B')\in \M(S',M')$ by precomposing or postcomposing with $\a$ or
$\b$-maps. The universal property of colimits implies the existence of an
induced homomorphism $f_* : \SAlg(S,M) \to \SAlg(S',M')$ between Fukaya Hall
algebras. For any full arc system $A$ of $S$, there is a commutative diagram:
\begin{equation*}\begin{tikzpicture}[scale=10, node distance=2.5cm]
\node (X) {$\Alg(S,A)$};
\node (B1) [right=1.25cm of X] {$\Alg(S',A')$};
\node (A1) [below=1.25cm of X]{$\SAlg(S,A)$};
\node (C1) [below=1.25cm of B1]{$\SAlg(S',A')$};
\draw[->] (X) to node {$f_*$} (B1);
\draw[->] (A1) to node  {$f_*$} (C1);
\draw[->] (X) to node   {} (A1);
\draw[->] (B1) to node {} (C1);
\end{tikzpicture}
\end{equation*}
in which the vertical arrows are isomorphisms and the map between composition subalgebras is a monomorphism. Therefore, the induced map $f_*$ between algebras must be a monomorphism.
\end{proof}

This section is concluded with a proposition 
that gives a condition which guarantees 
an embedding $f : (S,A) \to (S',A')$ 
induces a full and faithful functor $\F(S,A) \to \F(S',A')$
between associated Fukaya categories.  Corollary \ref{skdiskembcor} shows
that any such functor induces a monomorphism between Fukaya Hall algebras.

\begin{proposition}\label{ffprop}
Let $S$ and $S'$ be finitary graded marked surfaces with markings $M
\subset \partial S$ and $M' \subset S'$, and arc systems $A$ and $A'$
respectively. Suppose that $\iota: S \to S'$ is an embedding of marked surfaces such
that the induced map $\iota: \pi_0(M) \to \pi_0(M')$ between the sets of
path-components of marked intervals is injective. Then the functor $\iota_*:
\F(S, A) \to \F(S',A')$ is full and faithful.
\end{proposition}
\begin{proof}
Since $\iota$ is an embedding, 
the following  map of sets is injective:
\[
\iota_*: \{\textrm{boundary paths in } S \} \hookrightarrow  \{\textrm{boundary paths in }S'\}
\]
Since maps in $\F(S,A)$ are linear combinations of boundary paths in $M$
between arcs in $A$ and maps in $\F(S',A')$ are linear combinations of
boundary paths in $M'$ between arcs in $A'$, the injectivity of the set map
$\iota_*$ above implies that the functor $\iota_*$ is faithful.

The remainder of the proof will establish that $\iota_*$ is full.  First,
fix two arcs $X$ and $Y$ in $A$ and let $X' := \iota(X)$ and $Y' :=
\iota(Y)$ be their images in $A'$. 
Since maps between marked
surfaces are assumed to be orientation preserving, $\iota$ preserves the
orientation of boundary paths.
Therefore, in order to prove that $\iota_*$ is full, it suffices to prove the following implication: if $z'$ is a boundary path between $X'$ and $Y'$,
then there is a boundary path $z$ between $X$ and $Y$ in $S$ such that
$\iota(z) = z'$.

Suppose that $z'$ is a boundary path between the endpoint $X'(0)$ of $X'$
and the endpoint $Y'(0)$ of $Y'$, and that $z'$ is contained in the marked
interval $m' \subset M'$.  Since $\iota: \pi_0(M) \to \pi_0(M')$ is
injective, the preimage $\iota^{-1}(X'(0) \sqcup Y'(0))$ must be contained
in a single marked interval $m \subset M$. Since $S$ is finitary, there is a
unique boundary path $z \subset m$ between $X(0)$ and $Y(0)$. Since $S'$ is
finitary, there is a unique boundary path between $X'(0)$ and $Y'(0)$. Since
$\iota$ takes boundary paths to boundary paths, this shows $\iota(z) = z'$.
This establishes the claim in the previous paragraph and shows that
$\iota_*$ is full.
\end{proof}

Note that in the previous proposition there were no assumptions about the
number of arcs in $A$ or $A'$. The converse of this proposition is also
true, at least if $A$ has enough arcs. 
A full arc system always has enough arcs in this sense. 
The lemma below is included for completeness.

\begin{lemma}
Suppose that $\iota_*$ is full and every marked interval in $M$ intersects
an arc in $A$.  Then the induced map $\iota: \pi_0(M) \to \pi_0(M')$ is
injective.
\end{lemma}
\begin{proof}
Suppose that $m_1, m_2\subset M$ are two marked intervals such that
$\iota_*(m_1\sqcup m_2) \subset m'$, for some marked interval $m' \in M'$. 
By assumption there are arcs $X_1,X_2\in A$ such that $X_i(0) \in m_i$.
If $X_1 = X_2$ then there is a boundary path between the two endpoints of $\iota(X_1)$ that is not
  in the image of $\iota$, contradicting the assumption that $\iota$ is full. Assume that $X_1\ne X_2$. Then there is no boundary path in $M$ between $X_1(0)$ and $X_2(0)$, since they lie in different marked
intervals. However, there is a boundary path between $\iota(X_1(0))$ and
$\iota(X_2(0))$ in $M'$ since they lie in the same marked interval $m'$. 
So $\iota_*$ is not full which is a contradiction.
\end{proof}

\begin{cor}\label{fdiskembcor}
Suppose that $\iota : (S,M) \to (S',M')$ is a map of marked surfaces, $S$ and $S'$ equipped with full arc systems $A$ and $A'$ and $\iota$ satisfies the criterion of Prop. \ref{ffprop} above. Then there is an induced monomorphism $\iota_* : \Alg(S,A) \to \Alg(S',A')$ between associated Hall algebras.
\end{cor}
\begin{proof}
Prop. \ref{ffprop} ensures that the associated functor $\F(S,M) \to \F(S',M')$ is full and faithful. So Prop. \ref{mapprop} shows that there is a monomorphism $\Alg(S,A) \to \Alg(S',A')$.
\end{proof}  

\begin{cor}\label{skdiskembcor}
If $\iota : (S,M) \to (S',M')$ is a map of marked surfaces which satisfies the criterion of Prop. \ref{ffprop} above then there is an induced monomorphism $\iota_* : \SAlg(S,M) \to \SAlg(S',M')$ between associated Fukaya 
Hall algebras.
\end{cor}  
\begin{proof}
By the previous corollary, there is a monomorphism $\Alg(S,A) \to \Alg(S',A')$ and Cor. \ref{skeinfuncor} shows that this monomorphism induces the required monomorphism $\iota_* : \SAlg(S,M) \to \SAlg(S',M)$ between algebras.
\end{proof}

This concludes the present study of Hall algebras associated to general
surfaces. The next step in what follows is to determine a full presentation
for the algebras $\Alg(D^2,A)$ associated to disks. This together with the
Corollary \ref{skdiskembcor} above and the local computation in Lemma
\ref{skeinthmthm} will show that the HOMFLY-PT skein relation holds in the
Fukaya Hall algebra $\SAlg(S,M)$ of an arbitrary surface.

\subsection{The Hall algebra of disks with minimal arc systems}\label{mindisksec}
Here in Def. \ref{halldiskdef} we introduce an algebra $\Alg(D^2,\A)$
associated to a graded disk with $m$ marked intervals parameterized by a
minimal arc system $\A$. Theorem \ref{minarcthm} establishes an
isomorphism between this algebra and the derived Hall algebra of the
associated Fukaya category.

Recall from Notation \ref{disknotation} that the foliation data of the disk $(D^2,m)$ equipped with minimal arc system $\A$ is recorded by a function $\h : \A \to \ZZ$ which satisfies 
\begin{equation*}\label{eq:sumofdegrees}
\sum_{\E_i \in \A} \h(\E_i) = m - 2.
\end{equation*}
We will consistently identify the cyclically ordered sets $\A$ and $\ZZ/m$, setting $\h(\E_i)=\h(i)$. There is a function $\inp{-} : \A \to \ZZ$  given by the sum
\begin{equation*}
\inp{k} = \sum_{j=1}^{k-1} (1-h(j)).
\end{equation*}

There are a number of elementary identities satisfied by these functions which will be used in arguments surrounding the proof of Theorem \ref{minarcthm}. For example,
\begin{align}
\tau^i\inp{k+1} - \tau^i\inp{k} &= 1-\tau^ih(k) &\label{eq:shiftid1}\\
\tau^i\inp{m-1} - h(i)  &= h(i-1) \label{eq:shiftid2}\\
\tau^k\inp{m-k+1} +\inp{k} &= 1+h(k)\label{eq:shiftid3}
\end{align}

\subsubsection{Definition of a minimal arc algebra}\label{mindefsec} The algebra $\Alg(D^2,\A)$ associated to a minimal arc system $\A$ is introduced below. Theorem \ref{minarcthm} will establish an isomorphism with the Hall algebra of the Fukaya category associated to this arc system.

\begin{defn}\label{halldiskdef}
Suppose that the disk with $m$ marked intervals is equipped with a minimal arc system $\A$ and foliation data $\h : \A \to \ZZ$.  Then the {\em minimal arc algebra} $\Alg(D^2,\A)$ is the $\QQ(\sqrt{q})$-algebra generated by the suspensions of the graded boundary arcs: 
$$\ZZ \A = \{ \E_{i,n} : n\in\ZZ, \E_{i}\in \A\}\conj{ where } \E_{i,n} = \s^n\E_{i}$$
subject to the relations listed below:
  \begin{description}
  \item[(R1)\namedlabel{r1:itm}{\lab{R1}}] Self-extension:
\begin{align*}
[\E_{i,0},\E_{i,k}]_{q^{2(-1)^k}} &=\d_{k,1}\frac{q^{-1}}{q^2-1}\conj{ for } k \geq 1,
\end{align*}

  \item[(R2)\namedlabel{r2:itm}{\lab{R2}}] Adjacent commutativity and convolution:
\begin{align}[\E_{i+1,k},\E_{i,h(i)}]_{q^{(-1)^{k+1}}} &= 0  \conj{ for } k > 1 \label{eq:tmpchange},\\
[\E_{i+1,k},\E_{i,h(i)}]_{q^{(-1)^{k}}} &= 0 \conj{ for } k<1\notag, 
\end{align}
\begin{equation}
\E_{i,h(i)} = [\E_{i+m-1,\tau^i\inp{m-1}}, \ldots, \E_{i+2,\tau^i\inp{2}}, \E_{i+1,\tau^i\inp{1}}]_q \label{eq:cyc1},
\end{equation}

 \item[(R3)\namedlabel{r3:itm}{\lab{R3}}] Far-commutativity: 
\begin{align*}[\E_{i,0},\E_{j,k}]_1 &= 0\conj{ for } k\in\ZZ \normaltext{ and } \vnp{i-j} \geq 2.
\end{align*}
\end{description}
\end{defn}

\begin{rmk}
Note that all the relations above are invariant under the shift automorphism $\sigma$. The relations are closed under the 
action of $\ZZ/m = \inp{\tau}$. For instance, Eqn. \eqref{eq:cyc1} can be written as:
$$\tau^i\E_{0,h(0)} = \tau^i[\E_{m-1,\inp{m-1}}, \ldots, \E_{2,\inp{2}}, \E_{1,\inp{1}}]_q\conj{ for } i\in\ZZ/m.$$
\end{rmk}

The relations of the minimal arc algebra in Def. \ref{halldiskdef} imply a
family of relations analogous to relation \eqref{eq:cyc1}. Each relation in
this family is obtained by truncating the twisted complex from
Prop. \ref{convprop} in different places. Since each relation implies the
next, and since the final equation is equivalent to the first one, Eqn. \eqref{eq:cyc1} can be replaced by any of the relations in Lemma \ref{cyclicprop2}. Deriving these
relations from the presentation above is the purpose of the lemma below.

\begin{lemma}\label{cyclicprop2}
The {\em cyclic convolution relations} below, obtained by cyclically shuffling arcs to  $\E_j$ to the left-hand side of Eqn. \eqref{eq:cyc1} at the expense of a degree shift, are each equivalent to Eqn. \eqref{eq:cyc1} assuming other the relations in Def. \ref{halldiskdef} above. 
\begin{gather*}
\E_{i,h(i)} = [\E_{i+m-1,\tau^i\inp{m-1}}, \ldots, \E_{i+1,\tau^i\inp{1}}]_q \\
 [\E_{i+1,\tau^i\inp{1}+1},\E_{i,h(i)}]_q = [\E_{i+m-1,\tau^i\inp{m-1}}, \ldots, \E_{i+2,\tau^i\inp{2}}]_q \\
 [\E_{i+2,\tau^i\inp{2}+1}, \E_{i+1,\tau^i\inp{1}+1},\E_{i,h(i)}]_q = [\E_{i+m-1,\tau^i\inp{m-1}},\ldots, \E_{i+3,\tau^i\inp{3}}]_q\\
 \vdots\hspace{2in}\vdots \\
 [\E_{i+m-3,\tau^i\inp{m-3}+1},\ldots, \E_{i+1,\tau^i\inp{1}+1}, \E_{i,h(i)}]_q =   [\E_{i+m-1,\tau^i\inp{m-1}},\E_{i+m-2,\tau^i\inp{m-2}}]_q\\
 [\E_{i+m-2,\tau^i\inp{m-2}+1}, \ldots, \E_{i+1,\tau^i\inp{1}+1}, \E_{i,h(i)}]_q =   \E_{i+m-1,\tau^i\inp{m-1}}.
\end{gather*}
\end{lemma}

\begin{proof}
The proof is by induction. The first item in the list above is Eqn. \eqref{eq:cyc1}. 
Before proceeding to the inductive step, let's simplify notation by introducing the variables
$$ X_{k} := E_{i+k,\tau^i\inp{k}+1}\conj{and} Y_k := [E_{i+m-1,\tau^i\inp{m-1}},\ldots,E_{i+k,\tau^i\inp{k}}]_q.$$
Assume the $k$th cyclic convolution equation above holds:
\begin{equation}\label{eq:tmpeqk}
[X_{k},\ldots, X_1, \E_{i,\h(i)}]_q = Y_{k+1}.
\end{equation}
The induction step has two parts. First we show that
$[X_{k+1}, Y_{k+2}]_{q^{-1}} = 0$. The second part uses this assumption together with the omni-Jacobi
relation \eqref{oj:itm} to show that $[X_{k+1},Y_{k+1}]_q = Y_{k+2}$. The
latter relation suffices to complete the proof since applying the
operator $[X_{k+1}, -]_q$ to both sides of the $k$th cyclic convolution equation \eqref{eq:tmpeqk} produces the
$k+1$st cyclic convolution equation by  rebracketing  \eqref{orderprop}.

{\it Part 1:} 
Note that $Y_k = [Y_{k+1},E_{i+k,\tau^i\inp{k}}]_q$ by rebracketing, 
so $[Y_{k+2},X_{k+1}]_q$ becomes
\begin{align*}
&=[[Y_{k+3},E_{i+k+2,\tau^i\inp{k+2}}]_q,E_{i+k+1,\tau^i\inp{k+1}+1}]_q\notag\\
&= [Y_{k+3},[E_{i+k+2,\tau^i\inp{k+2}},E_{i+k+1,\tau^i\inp{k+1}+1}]_q]_q  &\textrm{\eqref{orderprop} \& \eqref{eq:tmpchange}}\notag\\
&=0\notag.
\end{align*}
The last equality uses  \eqref{eq:tmpchange} 
and the identity \eqref{eq:shiftid1}.
The relation $[X_{k+1}, Y_{k+2}]_{q^{-1}} = 0$ follows from  $[Y_{k+2},X_{k+1}]_q = 0$ and anti-symmetry \eqref{as:itm}.

{\it Part 2:}  The omni-Jacobi relation $\eqref{oj:itm}[q^3,q^{-2}q^{-2}]$ is
\begin{equation*}
[x,[y,z]_q]_q = -q^3[z,[x,y]_{q^{-1}}]_{q^{-3}} - q[y,[z,x]_{q^{-2}}]_{q^2}.
\end{equation*}
Now the expression $[X_{k+1}, Y_{k+1}]_q$ can be written as
\begin{align*}
&= [X_{k+1},[Y_{k+2},E_{i+k+1,\tau^i\inp{k+1}}]_q]_q & \eqref{orderprop}\\
&= -q^3[E_{i+k+1,\tau^i\inp{k+1}},[X_{k+1},Y_{k+2}]_{q^{-1}}]_{q^{-3}} \\
&\quad\quad\phantom{=}- q[Y_{k+2},[E_{i+k+1,\tau^i\inp{k+1}},X_{k+1}]_{q^{-2}}]_{q^2}&\textrm{$\eqref{oj:itm}[q^3,q^{-2}, q^{-2}]$}\\
&= 0-q[Y_{k+2},q^{-1}/(q^2-1)]_{q^2}   &\textrm{(Step 1) \& \eqref{r1:itm}}\\
&= Y_{k+2}.
\end{align*}

\end{proof}

\subsubsection{The isomorphism theorem}\label{minarcproofsec}
In the theorem below, the Hall algebra of a disk with $m$ marked intervals is
computed. The gist of the argument is to show that a set of relations
\eqref{r1:itm}, \eqref{r2:itm} and \eqref{r3:itm} generate the kernel of the
natural map $\kappa$ from the free $\QQ(\sqrt{q})$-algebra on 
arcs to the Hall algebra of the Fukaya category. This is accomplished by
comparing the presentation in Def. \ref{halldiskdef} to the Hall algebra
$\tDHa(A_{m-1})$ in Def. \ref{dhaanprop} using the equivalence
$\phi = \phi_0$ from Prop. \ref{dereqprop}.

\begin{thm}\label{minarcthm}
  The natural map $\kappa : F(\ZZ \A) \to \tDHa (D^\pi\F(D^2,\A))$ from the
  free $\QQ(\sqrt{q})$-algebra on suspensions of a minimal arc system to the
  derived Hall algebra of the Fukaya category associated to $(D^2,\A)$ with
  disk with $m$ marked intervals induces an isomorphism
  $$\bar{\kappa} : \Alg(D^2,\A)\xto{\sim} \tDHa (D^\pi \F(D^2,\A))$$
from the minimal arc algebra to the derived Hall algebra.
\end{thm}

The proof is established by a series of lemmas which constitute the majority
of this section.  Recall that Prop. \ref{dereqprop} introduces an equivalence 
\[
\phi := \phi_0 : D^b(\Rep_k(A_{m-1})) \xto{\sim} D^\pi\F(D^2,\A)
\] 
of triangulated categories. The proof of the theorem is structured to leverage the induced isomorphism
$$\phi_* : \tDHa(A_{m-1}) \xto{\sim} \tDHa(D^\pi\F(D^2,\A)).$$
In order to show that the kernel of the map $\kappa$ is generated by the
relations \eqref{r1:itm}, \eqref{r2:itm} and \eqref{r3:itm}, it suffices to define mutually inverse isomorphisms $\phi : \tDHa(A_{m-1}) \to \Alg(D^2,\A)$ and $\psi : \Alg(D^2,\A)\to \tDHa(A_{m-1})$ in such a way as to make the diagram below commute.
\begin{equation}\label{thmdiag}\begin{tikzpicture}[scale=10, node distance=2.5cm]
\node (X) {$\Alg(D^2,\A)$};
\node (A1) [below=1.25cm of X]{$\tDHa(A_{m-1})$};
\node (B1) [right=1.25cm of X] {$\tDHa (D^\pi\F(D^2,\A))$};
\draw[->,bend right=15] (A1) to node [swap] {$\phi_*$} (B1);
\draw[->,bend left=15] (A1) to node   {$\varphi$} (X);
\draw[->,bend left=15] (X) to node {$\psi$} (A1);
\draw[->] (X) to node {$\bar{\kappa}$} (B1);
\end{tikzpicture}
\end{equation}

Since the choice $\phi_*$ is fixed, commutativity of the diagram forces us to define the map $\varphi$ as 
\begin{equation}\label{eq:phidef}
  \varphi(z_{i,0}) := \E_{i,\inp i} \conj{ for } 1\leq i < m
\end{equation}  
and to define the map $\psi$ as
\begin{equation}\label{eq:psidef}
\psi(\E_{i,0}) := 
\left\{ \begin{array}{cl} 
z_{i,-\inp{i}} & \textrm{ when } 1 \leq i < m, \\
{[z_{m-1,-\h(m)},\ldots,z_{1,-\h(m)}]_q} & \textrm{ when } i=m,
\end{array}\right.
\end{equation}
where the iterated bracket notation in Def. \ref{bracketnotationdef} is used for $\psi$ when $i=m$.

\subsubsection{Proof of the isomorphism theorem}
The proof consists of three steps. First we argue that $\varphi$ is a
homomorphism. The second step is to show that $\psi$ is a
homomorphism. Lastly, we check that $\varphi$ and $\psi$ are mutually
inverse isomorphisms. Each step constitutes a lemma in what follows.

\begin{rmk}
  In what follows, since both $\varphi$ and $\psi$ commute with shifts in
  homological degree in either algebra, we will simplify notation by proving
  the relations above when the homological degree $n$ is zero.
\end{rmk}

\begin{lemma}{(Step 1)}\label{lemma:phihom}
The assignment Eqn. \eqref{eq:phidef}  determines a homomorphism
$$\varphi : \tDHa(A_{m-1}) \to \Alg(D^2,\A).$$
\end{lemma}
\begin{proof}
First observe that $\varphi$ preserves the first part of \eqref{eq:z1} by 
the relation
\eqref{r3:itm}. The interesting part of the proof consists of checking the Serre relations in \eqref{eq:z1}. In the first case, when $1\leq i \leq m-2$, the expression $\varphi([z_{i+1,0},[z_{i+1,0},z_{i,0}]_q]_{q^{-1}})$ is equal to
\begin{align*}
&= [E_{i+1,\inp {i+1}},[E_{i+1,\inp{i+1}},E_{i,\inp i}]_q]_{q^{-1}} & \eqref{eq:phidef}\\
&=  \sigma^{\inp{i+1}-1}[E_{i+1,1},[E_{i+1,1},E_{i,h(i)}]_q]_{q^{-1}} &\eqref{eq:shiftid1}\lbrack i=1\rbrack \\
&= -q^{-1}\sigma^{\inp{i+1}-1}[[E_{i+1,1},E_{i,h(i)}]_q,E_{i+1,1}]_{q} & \eqref{as:itm}\\
&= -q^{-1}\sigma^{\inp{i+1}-1}[[E_{i+m-1,\tau^i\inp{m-1}},\ldots,E_{i+2,\tau^i\inp {2}}]_q,E_{i+1,1}]_q&\textrm{(Lem. \ref{cyclicprop2})}\\
&= -q^{-1}\sigma^{\inp{i+1}-1}[
\ldots,E_{i+3,\tau^i\inp 3},[E_{i+2,\tau^i\inp {2}},E_{i+1,1}]_q]_q & \eqref{orderprop}\\
&= -q^{-1}\sigma^{\inp{i+1}-1}[
\ldots,E_{i+3,\tau^i\inp 3},[E_{i+2,1-h(i+1)},E_{i+1,1}]_q]_q &\textrm{(Def. of } \tau^i\inp 2\textrm{)}\\
&= -q^{-1}\sigma^{1-h(i+1)}\sigma^{\inp{i+1}-1}[\ldots,[E_{i+2,0},E_{i+1,h(i+1)}]_{q}]_q \\
&=0&\textrm{\eqref{r2:itm}}
\end{align*}
Similarly, for the other Serre relation, $\varphi([z_{i,0},[z_{i,0},z_{i+1,0}]_{q^{-1}}]_{q})$ is 
\begin{align*}
&= [E_{i,\inp i}, [E_{i, \inp i},E_{i+1,\inp{i+1}}]_{q^{-1}}]_q & \eqref{eq:phidef}\\
&= -q^{-1}[E_{i,\inp i},[E_{i+1,\inp{i+1}},E_{i,\inp i}]_q]_q & \eqref{as:itm}\\
&= -q^{-1}\sigma^{\inp{i+1}-1}[E_{i,h(i)},[E_{i+1,1},E_{i,h(i)}]_q]_q & \eqref{eq:shiftid1}\lbrack i=1\rbrack \\
&= -q^{-1}\sigma^{\inp{i+1}-1}[E_{i,h(i)},[E_{i+n-1,\tau^i\inp{n-1}},\ldots,E_{i+2,\tau^i\inp 2}]_q]_q& \textrm{(Lem. \ref{cyclicprop2})}\\
&=-q^{-1}\sigma^{\inp{i+1}-1}
[[E_{i,h(i)},E_{i+n-1,\tau^i\inp{n-1}}]_q,\ldots]_q&  \eqref{orderprop}\\
&=-q^{-1}\sigma^{\inp{i+1}-1}\sigma^{-h(i)}[[E_{i,0},E_{i+n-1,h(i-1)}]_{q},\ldots]_q & \eqref{eq:shiftid2} \\
&=0& \eqref{r2:itm}
\end{align*}

The second half of the proof consists of checking that $\varphi$ preserves the \eqref{eq:z2} relation
\[
[z_{i,0},z_{j,k}]_{q^{(-1)^k i\cdot j}} = \delta_{i,j}\delta_{k,1}\frac{q^{-1}}{q^2-1} \conj{ for } k \geq 1.
\]
There are three cases corresponding to the three possibilities $i\cdot j\in\{0,-1,2\}$.
If $i =j$ then $i\cdot j = 2$ and \eqref{r1:itm} implies the result. If $\vnp{i-j} \geq 2$ then $i\cdot j = 0$ and \eqref{r3:itm} shows that the claim holds. In the remaining case, if $j=i+1$ then set $p(k) = (-1)^k$ and observe that
\begin{align*}
\varphi([z_{i,0},z_{i+1,k}]_{q^{p(k+1)}}) &= 
[E_{i,\inp i},E_{i+1,\inp{i+1}+k}]_{q^{p(k+1)}} & \eqref{eq:phidef}\\
&= -q^{p(k+1)}[E_{i+1,\inp{i+1} + k},E_{i,\inp i}]_{q^{p(k)}} & \eqref{as:itm}\\
&= -q^{p(k+1)}\sigma^{\inp{i+1}-1}[E_{i+1,1+k},E_{i,h(i)}]_{q^{p(k)}} &\eqref{eq:shiftid1}\lbrack i=1\rbrack\\
&= -q^{p(\ell)}\sigma^{\inp{i+1}-1} [E_{i+1,\ell},E_{i,h(i)}]_{q^{p(\ell-1)}} & (\ell = k+1)\\
&=0&\textrm{\eqref{r2:itm}}
\end{align*}
The $j=i-1$ case is omitted because its proof is almost identical.

\end{proof}

\begin{lemma}{(Step 2)}\label{lemma:psihom}
The assignment Eqn. \eqref{eq:psidef}  determines a homomorphism
$$\psi : \Alg(D^2,\A) \to \tDHa(A_{m-1}).$$
\end{lemma}  
\begin{proof}
The relation \eqref{eq:z2} implies that $\psi$ preserves \eqref{r3:itm} and \eqref{r1:itm} when $i \ne n$. If $i=n$ then $\psi$ preserves \eqref{r1:itm} by Proposition \ref{prop:skeinselfext}. 

In order to complete the proof it suffices to show that $\psi$ preserves \eqref{r2:itm}. This relation consists of three equations. The longest relation
\begin{equation}\label{eq:wewant}
E_{i,h(i)} = [E_{i+m-1,\tau^i\inp{m-1}},\ldots,E_{i+1,\tau^i\inp 1}]_q,
\end{equation}
will be examined first. Before proceeding to the proof observe that 
\begin{equation}
  \psi(E_{m,n}) = z_{(1,m-1),-h(m)+n}\label{eq:psiid0}
\end{equation}
by \eqref{eq:psidef} and Def. \ref{usefulcor}. When $i\ne m$, the same assumptions imply
\begin{align}
  \psi([E_{i-1,\tau^i\inp{-1}},\cdots,E_{1,\tau^i\inp{1-i}}]_q) &= z_{(1,i),\tau^i\inp{-i+1}}   \label{eq:psiid1}\\
  \psi([E_{m-1,\tau^i\inp{-i-1}},\cdots,E_{i+1,\tau^i\inp{1}}]_q) &= z_{(i+1,m),-2+\tau^i\inp{-i+1}} \label{eq:psiid2}
\end{align}

Now when $i=m$, Eqn. \eqref{eq:psiid0} shows that $\psi$ preserves \eqref{eq:wewant}. So assuming $i\ne m$, applying $\psi$ to the left-hand side of \eqref{eq:wewant} produces $\psi(E_{i,h(i)}) = z_{i,h(i)-\inp i}$ by \eqref{eq:psidef}. Then applying $\psi$ to the right-hand side of \eqref{eq:wewant} is equal to
\begin{align*}
&= \psi([E_{i-1,\tau^i\inp{-1}},\ldots,E_{1,\tau^i\inp{1-i}},E_{m,\tau^i\inp{-i}},E_{m-1,\tau^i\inp{-i-1}},\ldots, E_{i+1,\tau^i\inp 1}]_q)
 & \textrm{($\inp{m+k} = \inp k$)}\\
&= [z_{(1,i),\tau^i\inp{-i+1}},z_{(1,m-1),-h(m)+\tau^i\inp{-i}},z_{(i+1,m),-2+\tau^i\inp{-i+1}}]_q
&\textrm{\eqref{eq:psiid1} \& \eqref{eq:psiid2} }
\\
&=\sigma^{\tau^i\inp{-i+1}} [z_{(1,i),0},[z_{(1,n-1),-1},z_{(i+1,n-1),-2}]_q]_q&\textrm{\eqref{orderprop} \& \eqref{eq:shiftid1}}\\
&=\sigma^{\tau^i\inp{-i+1}}[z_{(1,i),0},z_{(1,i+1),-1}]_q  &\eqref{s1:itm}\\
&=\sigma^{\tau^i\inp{-i+1}}z_{i,-1}&\textrm{(Prop. \ref{gensk1})}\\
&=\sigma^{1+h(i)-\inp i}z_{i,-1}
&\eqref{eq:shiftid3}\\
&= z_{i,h(i)-\inp i}
\end{align*}

In order to complete the proof, it suffices to check that $\psi$ preserves relation \eqref{r2:itm} when $k\ne 1$.
If $i \not\in \{m,m-1\}$ and $k>1$ then
\begin{align*}
\psi([E_{i+1,k},E_{i,h(i)}]_{q^{p(k+1)}}) &=  [z_{i+1,k-\inp{i+1}},z_{i,h(i)-\inp{i}}]_{q^{p(k+1)}} & \textrm{$(p(k) = (-1)^k)$} \\
&= \sigma^{h(i)-\inp{i}}[z_{i+1,k+1},z_{i,0}]_{q^{p(k+1)}} & \eqref{eq:shiftid1}\\
&=-q^{p(k+1)}\sigma^{h(i)-\inp{i}}[z_{i,0},z_{i+1,k+1}]_{q^{-p(k+1)}}  & \textrm{\eqref{as:itm} \& \eqref{eq:z2}}\\
&=0
\end{align*}

Before proceeding to the $i=m$ and $i=m-1$ cases, we show that $[z_{1,k},z_{(2,n-1),0}]_{q^{(-1)^k}} = 0$. Observe that $\eqref{oj:itm}[q^{1+p(k)},q^{-1},q^{-p(k)}]$ is
$$[x,[y,z]_q]_{q^{p(k)}} + q^{1+p(k)}[z,[x,y]_1]_{q^{-1-p(k)}} - q^{p(k)}[y,[z,x]_{q^{-p(k)}}]_q=0.$$
So the commutator $[z_{1,k},z_{(2,n-1),0}]_{q^{p(k)}}$ of interest is equal to
\begin{align*}
&= [z_{1,k},[z_{(3,n-1),0},z_{2,0}]_q]_{q^{p(k)}} &  \eqref{orderprop} \\
&= -q^{1+p(k)}[z_{2,0},[z_{1,k},z_{(3,n-1),0}]_1]_{q^{-1-p(k)}} \\
&\quad\quad\phantom{=} - q^{p(k)}[z_{(3,n-1),0},[z_{2,0},z_{1,k}]_{q^{-p(k)}}]_q & \textrm{\eqref{oj:itm}$\lbrack q^{1+p(k)},q^{-1},q^{-p(k)}\rbrack$}\\
&= -q^{1+p(k)}[z_{2,0},0]_{q^{-1-p(k)}} - q^{p(k)} [z_{(3,n-1),0},0]_{q} & \textrm{ \eqref{s3:itm} \& \eqref{eq:z2} }\\
&=0
\end{align*}

Note that $\eqref{oj:itm}[q^{1-2p(k)},q^{-1+p(k)},q^{2p(k)}] $ is\\
$$
[x,[y,z]_q]_{q^{p(k+1)}} + q^{1-2p(k)} [z, [x,y]_{q^{p(k)}}]_{q^{2p(k)-1}} + q^{-p(k)}[y,[z,x]_{q^{2p(k)}}]_{q^{1-p(k)}}=0.
$$

Now  $\psi$ is applied to \eqref{r2:itm} with $i=m$ and $k>1$.  After setting $p(k) = (-1)^k$, the expression $\psi([E_{1,k},E_{m,h(m)}]_{q^{p(k+1)}})$ is equal to
\begin{align*}
&= [z_{1,k}, z_{(1,m-1),0}]_{q^{p(k+1)}} & \eqref{eq:psidef}\\
&=[z_{1,k},[z_{(2,m-1),0},z_{1,0}]_{q}]_{q^{p(k+1)}} & \eqref{orderprop}\\
&= -q^{1-2p(k)}[z_{1,0},[z_{1,k},z_{(2,m-1),0}]_{q^{p(k)}}]_{q^{2p(k)-1}} \\ 
& \quad\quad\phantom{=}- q^{-p(k)} [z_{(2,m-1),0},[z_{1,0},z_{1,k}]_{q^{2p(k)}}]_{q^{1-p(k)}}&\eqref{oj:itm}[q^{1-2p(k)},q^{-1+p(k)},q^{2p(k)}] \\
&=-q^{1-2p(k)}[z_{1,0},0]_{q^{2p(k)-1}} - q^{-p(k)} [z_{(2,m-1),0},0]_{q^{1-p(k)}} & \textrm{(Previous) \& \eqref{eq:z2}}
\\
&=0
\end{align*}

When $i=m$ and $k=0$, the expression $\psi([E_{1,0},E_{m,h(m)}]_{q^{p(0)}})$ is equal to
\begin{align*}
&= [z_{1,0},z_{(1,m-1),0}]_q & \textrm{\eqref{eq:psidef} \& \eqref{eq:psiid0}}\\
&= -q[[z_{(3,m-1),0},[z_{2,0},z_{1,0}]_q]_q,z_{1,0}]_{q^{-1}}&\textrm{\eqref{orderprop} \& \eqref{as:itm}} \\
&=-q[z_{(3,m-1),0},[[z_{2,0},z_{1,0}]_q,z_{1,0}]_{q^{-1}}]_{q}&\eqref{eq:comm4usedlaterdonotdelete}\\
&=-q[z_{(3,m-1),0},[[z_{2,0},z_{1,0}]_q,z_{1,0}]_q]_q & \eqref{eq:z1} \\
&=0
\end{align*}
The proofs of the remaining cases, $i=m$ and $k<0$ case, as well as the $i=m-1$ cases, are nearly identical and are omitted.
\end{proof}

\begin{lemma}{(Step 3)}\label{lemma:invisos}
  The homomorphisms $\psi$ and $\varphi$ are mutually inverse isomorphisms
$$\psi\varphi = 1_{\tDHa(A_{m-1})}\conj{ and } \varphi\psi = 1_{\Alg(D^2,\A)}.$$
\end{lemma}
\begin{proof}
By definition $\psi\circ \varphi = 1_{\tDHa(A_{m-1})}$ and $\varphi(\psi(E_{i,n})) = E_{i,n}$ for $1 \leq i < m$ and $n\in\ZZ$. In the remaining case, when $i=m$,
\[
\varphi(\psi(E_{m,h(m)})) = \varphi([z_{m-1,0},\ldots,z_{1,0}]_q) = [E_{m-1,\inp{m-1}},\ldots,E_{1,\inp{1}}]_q=E_{m,h(m)}
\]
where the final equality follows from \eqref{r2:itm} and the identity $\tau^m=1$.
\end{proof}

Our discussion is concluded with a summary.

\begin{proof}{(of Thm. \ref{minarcthm})}
  By Lemmas \ref{lemma:phihom} and \ref{lemma:psihom}, there are maps
  $\varphi$ and $\psi$ between the minimal arc algebra $\Alg(D^2,\A)$ and the derived
  Hall algebra $\tDHa(A_{m-1})$ of the $A_{m-1}$-quiver. By Lemma
  \ref{lemma:invisos}, the map $\psi$ is an isomorphism.  On the other hand,
  there is an isomorphism
  $\phi_* : \tDHa(A_{m-1}) \xto{\sim} \tDHa (D^\pi\F(D^2,\A))$ obtained from
  Prop. \ref{dereqprop}. Composing these two maps and using the
  commutativity of Diagram \eqref{thmdiag}, $\bar{\kappa} = \phi_* \circ \psi$,
  implies that the map $\bar{\kappa}$ is an isomorphism.
\end{proof}

\subsection{Gluing minimal arc systems in the disk}\label{diskgluesec}
In this section we give a presentation for the Hall algebra of the Fukaya
category of a disk $D^2$ with an arbitrary arc system $A$.  This is
accomplished by cutting $D^2$ along the internal arcs of $A$ until one
obtains a collection of disks $(D^2_i,\A_{n_i})$ for which the arc system is
minimal. A gluing theorem shows that the algebra $\Alg(D^2,A)$ can be
assembled from $\aprod$-products of the algebras $\Alg(D^2_i,\A_{n_i})$ defined below.

\begin{defn}\label{topgluedef}
  Suppose that $(D^2_1,\A_n)$ and $(D^2_2,\A_m)$ are two disks with minimal
  arc systems $\A_n = \{E_k\}_{k\in\ZZ/n}$ and $\A_m=\{F_k\}_{k\in\ZZ/m}$. Up to homotopy, the foliations of each disk can be chosen to be tangent to boundary arcs. So if $E_i\in\A_n$ and $F_j\in\A_m$ then there is a gluing defined by the quotient
$$(D^2_1,\A_n)\sqcup_{i,j} (D^2_2,\A_m) = (D^2_1 \sqcup_{i,j} D^2_2, \A_n \sqcup_{i,j} \A_m).$$
The glued disk $D^2_1 \sqcup_{i,j} D^2_2 = (D^2_1 \sqcup D^2_2) / (E_i\sim F_j)$ is formed
by the quotient which identifies the two boundary arcs on each disk. The
disk supports a minimal arc system
$\A_n\sqcup_{i,j}\A_m = (\A_n\backslash \{E_i\}) \sqcup
(\A_m\backslash\{F_j\})$.

\vspace{.15in}
\begin{center}\label{fig:tmpgluedisk}
\begin{overpic}[scale=0.7]
{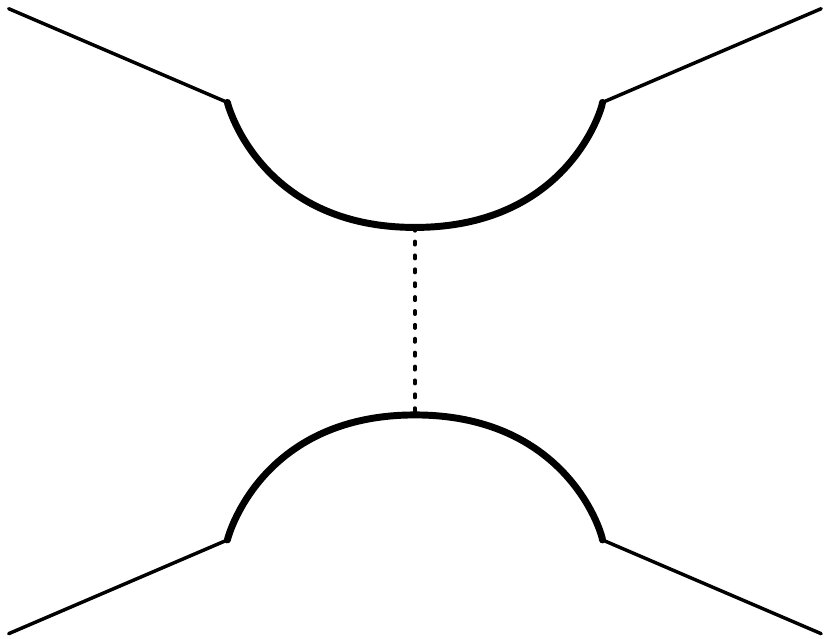}
\put(65,60){$E_i$}
\put(89,60){$F_j$}
\put(10,130){$E_{i-1}=G_{n-1}$}
\put(10,-5){$E_{i+1}=G_1$}
\put(110,130){$F_{j+1}=G_n$}
\put(110,-5){$F_{j-1}=G_{0}$}
\end{overpic}
\end{center}
\vspace{.15in}

The glued disk $D^2_1 \sqcup_{i,j} D^2_2$ supports a minimal arc system $\A_n\sqcup_{i,j}\A_m = \{G_k\}_{k\in\ZZ/(n+m-2)}$ with foliation data given by
\begin{equation}\label{gluefoleq}
g(G_k) = \left\{
\begin{array}{cl} 
e(F_i) + f(F_{j-1}) & k = 0 \\
e(E_{i+k}) & 1\leq k \leq n-2\\ 
e(E_{i-1}) + f(F_j) & k = n-1\\
f(F_{j+k-n+1}) & n \leq k \leq n+m-3.
\end{array}\right.
\end{equation}
\end{defn}  

Theorem \ref{gluethm} will show that the $\aprod$-product, introduced below,
computes the Hall algebra of the Fukaya category of a gluing.

\begin{defn}\label{alggluedef}
Suppose that $(D^2_1,\A_n)$ and $(D^2_2,\A_m)$ are graded disks and $(D^2_1,\A_n) \sqcup_{i,j} (D^2_2,\A_m)$ is a gluing. Then the $\aprod_{i,j}$-product algebra is defined by 
\begin{equation}\label{aproddef}
  \Alg(D^2_1,\A_n)\aprod_{i,j} \Alg(D^2_2,\A_m) := \frac{\Alg(D^2_1,\A_n)\ast \Alg(D^2_2,\A_m)}{\inp{(G1), (G3)}},
  \end{equation}
given by the free product of the algebras $\Alg(D^2_1,\A_n)$ and $\Alg(D^2_2,\A_m)$ subject to the relations listed below.
\begin{description}
\item[(G1)\namedlabel{g1:itm}{\lab{G1}}] Gluing: 
$$F_{j,s} = E_{i,s}\conj{ for all } s\in\ZZ,$$
\item[(G3)\namedlabel{g3:itm}{\lab{G3}}] Far-commutativity: 
$$[E_{k,s}, F_{\ell,t}]_1 = 0 \conj{ for all } (k,\ell)\not\in Int \normaltext{ and } s, t\in\ZZ,$$
where $Int =\{ (i+1,j-1), (i,j-1), (i,j+1), (i-1,j), (i-1,j+1), (i+1,j)\}$.
\end{description}

Relation \eqref{g1:itm} says that the arcs that are glued are
identified within the product $\Alg(D^2,\A_n)\aprod_{i,j} \Alg(D^2,\A_m)$,
and relation \eqref{g3:itm} says that all pairs which are not neighboring
must commute. Relations between adjacent edges, corresponding to \eqref{r2:itm}, are not imposed.
\end{defn}  

The theorem below shows that the topological $\sqcup_{i,j}$-gluing of disks
is compatible with the $\aprod_{i,j}$-product introduced above. The utility
of Eqn. \eqref{glueeq} below is that the algebra on the right-hand side is
naturally associated to the topological gluing of disks and the algebra on
the left-hand side contains a representative for the internal arc along
which the gluing was performed.

\begin{thm}\label{gluethm}
The Hall algebra of the disk obtained by a gluing is isomorphic to the $\aprod$-product of Hall algebras of the disks. There are isomorphisms:
\begin{equation}\label{glueeq}
  \a : \Alg(D^2_1 \sqcup_{i,j} D^2_2, \A_n \sqcup_{i,j} \A_m) \rightleftarrows \Alg(D^2_1,\A_n)\aprod_{i,j} \Alg(D^2_2,\A_m) : \b.
  \end{equation}
\end{thm}  

\newcommand{\pE}{\Alg(E)}
\newcommand{\pF}{\Alg(F)}
\newcommand{\pG}{\Alg(G)}

\vspace{.15in}
The remainder of this section contains the proof of the theorem.
Here is an outline of things to come. First, in Step 1, the statement of
Theorem \ref{gluethm} is refined, the assignments $\a$ and $\b$ are
introduced and shown to be mutually inverse on generators.  In Step 2, the
map $\b$ is shown to be a homomorphism. In Step 3, the map $\a$ is shown
to be a homomorphism.

Before getting started, a few notational preparations will simplify the discussion later. The auxiliary function
$$\inp{k}_e = \sum_{\ell=1}^{k-1} (1-e(\ell))\conj{ and } \inp{1}_e = 0,$$
from Notation \ref{disknotation} admits a mild generalization below.
\begin{notation}\label{peternotation}
\begin{equation*}
\fol j k e := \sum_{\ell = j}^{k-1} (1-e(\ell)).
\end{equation*}
By definition, $\inp{k}_e = \fol 1 k e$. In particular, $\fol 1 n e = 1 + e(n)$.
This function also satisfies the following identities:
\begin{equation*}\label{eq:splitpaths}
\fol j k e = \fol j {k'} e + \fol {k'} k e  \conj{if} j \leq k' \leq k,
\end{equation*}
$$\fol j k e + \fol k j e = 2.$$
\end{notation}
See Prop. \ref{folprop} below for properties of this function which hold in the context of the gluing argument.

\vspace{.15in}
{\it Step 1: Setup} Abbreviate $D^2_3 = D^2_1 \sqcup_{i,j} D^2_2$ and let $A$ be the arc system of the $D^2_3$ given by the union $\A_n\cup \A_m$. Then $A$ contains the minimal arc system plus one additional internal arc corresponding to the gluing.  Cor. \ref{fdiskembcor} implies that the inclusions:
$\iota_1 : \Alg(D^2_1,\A_n) \hookrightarrow \Alg(D^2_3,\A_n \sqcup_{i,j} \A_m)$ and
$\iota_2 : \Alg(D^2_2,\A_m) \hookrightarrow \Alg(D^2_3, \A_n \sqcup_{i,j} \A_m)$ are
monomorphisms. Moreover, the subset $\A_n \sqcup_{i,j}\A_m \subset A$ determines an 
isomorphism $\eta = \a_{\A_n \sqcup_{i,j}\A_m, A}$ by Corollary \ref{sheafpropcor}.
These observations lead to the diagram below.

  \begin{equation*}
  \begin{tikzpicture}[scale=10,node distance=2.5cm]
  \node (A1) {$\Alg(D^2_1,\A_n)\ast \Alg(D^2_2,\A_m)$};
  \node (A3) [right=2.5cm of A1] {$\Alg(D^2_3,\A_n \sqcup_{i,j} \A_m)$};
\node (B1) [below=1.5cm of A1] {$\Alg(D^2_1,\A_n) \aprod_{i,j} \Alg(D^2_2,\A_m)$};
   \draw[->] (A1) to node {$\xi$} (A3);
   \draw[->>] (A1) to node {$\pi$} (B1);
   \draw[bend left=10,->] (B1) to node {$\b$} (A3);
   \draw[bend left=10,->] (A3) to node {$\a$} (B1);
  \end{tikzpicture}
  \end{equation*}  
  The map $\pi$ is the quotient map, and the map $\xi$ is defined as
  $\eta \circ (\iota_1\ast \iota_2)$ using the maps which are
  described in the previous paragraph.  The maps $\a$ and $\b$ are
  determined on generators by the requirement that the diagram above commute
  and that they must be mutually inverse.

Before introducing $\a$ and $\b$, we'll simplify the notation.  The names of algebras will be abbreviated as follows
$$\pE = \Alg(D^2_1,\A_n),\quad \pF = \Alg(D^2_2,\A_m)\quad\normaltext{ and }\quad \pG = \Alg(D^2_1 \sqcup_{i,j} D^2_2, \A_n \sqcup_{i,j} \A_m).$$
The cyclic sets $\A_n$, $\A_m$ and $\A_n\sqcup_{i,j}\A_m$ will be indexed in
the following way:
\begin{gather}
  \A_n  \cong \{E_1,E_2,\ldots,E_n\},\quad \A_m \cong \{F_{n-1},F_{n}, \ldots,F_{n+m-2}\}\label{cycideq}\\
 \conj{ and } \A_n\sqcup_{i,j}\A_m \cong \{ G_1,G_2,\ldots, G_{n+m-2}\}.\notag
\end{gather}
Note the choice for $\A_m$. The advantage of this convention is that the $i$th boundary arc $G_i$ in $\pG$ is identified with either $E_i$ or $F_i$ in $\pE$ or $\pF$ respectively, see Eqns. \eqref{eq:gluingalphadef} and \eqref{eq:gluingbetadef} below.
As a consequence, the $n$th edge in $\pE$ is attached to the $n-1$st edge in $\pF$.

The algebra $\pE \aprod \pF$ is now the quotient of the free product $\pE \ast \pF$ by the gluing relations \eqref{g1:itm} and \eqref{g3:itm}:
\begin{gather}
E_{n,s} = F_{n-1,s}  \conj{ for all }  s\in\ZZ, \label{eq:gluingproofg1}\\
[E_{k,s}, F_{\ell,t}]_1 = 0  \conj{ when }\,\, (k,\ell) \not\in Int \normaltext{ and } s,t\in\ZZ \label{eq:gluingproofg2}
\end{gather}
where $Int = \{(1,n-1),(1,m+n-2),(n-1,n-1),(n-1,n),(n,n), (n,m+n-2)\}$. 
The foliation data for the glued disk is written as follows:
\begin{equation}\label{eq:gluingfoliation}
g(i) = 
\left\{ 
\begin{array}{cl}
e(i)&\textrm{if } 1 \leq i \leq n-2\\
e(n-1) + f(n-1) & \textrm{if } i=n-1\\
f(i)&\textrm{if } n \leq i \leq n+m-3\\
e(n) + f(n+m-2) & \textrm{if } i = n+m-2.
\end{array}
\right.
\end{equation}

\begin{prop}\label{folprop}
This data satisfies the relations below.
\begin{align}
\fol i k g &= \fol i k e & \conj{ for } & 1\leq i < k \leq n-1 \notag\\
\fol i k g &= \fol i k f & \conj{ for } & n\leq i < k \leq n+m-2 \notag\\
\fol i n g &= \fol i n e - f(n-1) & \conj{ for } & 1 \leq i < n \label{eq:gluingfolid2}
\end{align}
\end{prop}
\begin{proof}
The first and second equations follow from the definition of $g$ above.
The last equation follows from
\begin{align*}
\fol i n g &= \fol i {n-1} g + \fol {n-1} n g\\
&= \fol i {n-1} e + \fol {n-1} n g\\
&= \fol i n e  - f(n-1)
\end{align*}
\end{proof}

Having fixed the notation for the algebras and their foliation data, we now introduce the maps $\a$ and $\b$. The map $\alpha$ is defined by
\begin{equation}\label{eq:gluingalphadef}
\alpha(G_{i,s}) := 
\left\{ 
\begin{array}{cl}
E_{i,s}&\textrm{if } 1 \leq i \leq n-1\\
F_{i,s}&\textrm{if } n \leq i \leq m+n-2
\end{array}
\right.
\end{equation}
The map $\b$ is given by:
\begin{equation}\label{eq:gluingbetadef}\begin{array}{lll}
\b(E_{k,s}) &:=  G_{k,s} & \textrm{if } 1 \leq k \leq n-1, \\
\b(F_{k,s}) &:=  G_{k,s} & \textrm{if } n \leq k \leq n+m-2.\\
\end{array}\end{equation}
The cyclic relations in $\pE$ and $\pF$ determine the value of $\b$ on the glued generators $E_{n,1}$ and $F_{n-1,1}$ to be the following:
\begin{align}
\b(E_{n,1}) &= [G_{n-1,\fol n {n-1} e},\ldots, G_{1,\fol n 1 e}]_q\notag\\
&= \sigma^{\fol n 1 e}[G_{n-1,\fol 1 {n-1} e},\ldots, G_{1,\fol 1 1 e}]_q\notag\\
&= \sigma^{1-e(n)}[G_{n-1,\fol 1 {n-1} g},\ldots,G_{1,\fol 1 1 g}]_q\label{eq:gluingbe}\\
\b(F_{n-1,1}) &= [G_{n+m-2,\fol {n-1} {n+m-2} f},\ldots,G_{n,\fol {n-1} n f}]_q\notag\\
&= \sigma^{\fol {n-1} n f}[G_{n+m-2,\fol n {n+m-2} f},\ldots,G_{n,\fol n n f}]_q\notag\\
&= \sigma^{1-f(n-1)}[G_{n+m-2,\fol n {n+m-2} g},\ldots,G_{n,\fol n n g}]_q\label{eq:gluingbf}
\end{align}
Materials from Notation \ref{peternotation} and Proposition \ref{folprop} were used above.

With these definitions the proposition below is straightforward.

\begin{prop}\label{mutualinvprop}
  The maps $\a$ and $\b$ are mutually inverse on generators.
\end{prop}
\begin{proof}
  Observe that $\b\a = 1_{\pG}$ by virtue of the definition of each map on
  generators.  For the same reason, $\a\b(E_{k,s}) = E_{k,s}$ when $k\ne i$ and
   $\a\b(F_{k,s}) = F_{k,s}$ when $k \ne j$.  The exceptional cases
  follow from the \eqref{r2:itm} relation. For instance
\begin{align*}
  \a\b(E_{n,1}) &= \a [G_{n-1,\fol n {n-1} e},\ldots, G_{1,\fol n 1 e}]_q\\
                   & = [E_{n-1,\fol n {n-1} e},\ldots, E_{1,\fol n 1 e}]_q\\
                   &=E_{n,1}.
\end{align*}                     
The computation $\a\b(F_{n-1,1})$ is similar.
\end{proof}

\vspace{.15in}
To complete the proof it is necessary to show that the assignments $\a$ and $\b$ are
homomorphisms. Step 2 shows that $\b$ is a homomorphism and Step
3 shows that $\a$ is a homomorphism.

\vspace{.15in}
{\it Step 2.} In order to check that the map $\b$ is a homomorphism, it
suffices to show that the relations \eqref{eq:gluingproofg1} and \eqref{eq:gluingproofg2} are
contained in the kernel of the map $\xi$ from Step 1 above. The relations \eqref{eq:gluingproofg2} are in the
kernel by the far-commutativity relations \eqref{r3:itm} in the algebra $\pG$. The lemma below addresses the relations \eqref{eq:gluingproofg1}.

\begin{lemma}\label{lemma:gluingapp}
The following identity holds.
$$\b(E_{n,1}) = \b(F_{n-1,1})$$
\end{lemma}
\begin{proof}
The cyclic relation for the algebra $\pG$ reads
\begin{align*}
[G_{n-1,\fol 1 {n-1} g},\ldots,G_{1,\fol 1 1 g}]_q 
&= \sigma^{-1}[G_{n+m-2,\fol 1 {m+n-2} g},\ldots,G_{n,\fol 1 n g}]_q\\
&= \sigma^{-1+\fol 1 n g}[G_{n+m-2,\fol n {m+n-2} g},\ldots,G_{n,\fol n n g}]_q\\
&= \sigma^{e(n)-f(n-1)}[G_{n+m-2,\fol n {m+n-2} g},\ldots,G_{n,\fol n n g}]_q
\end{align*}
where in the last step we used the identity $\fol 1 n g = 1+e(n)-f(n-1)$, see Eqn. \eqref{eq:gluingfolid2} and Not. \ref{peternotation} above. This calculation, combined with the equations \eqref{eq:gluingbe} and \eqref{eq:gluingbf} above, completes the proof.
\end{proof}

\vspace{.15in}
{\it Step 3.} 
In order to prove that the assignment $\a$, introduced by Eqn. \eqref{eq:gluingalphadef}, is an algebra homomorphism, it is necessary to show that $\a$ respects the
relations defining the algebra $\pG$. The self-extension relations
\eqref{r1:itm} follow immediately from the \eqref{r1:itm} relations imposed in
either $\pE$ or $\pF$ respectively.  The far-commutativity relations
\eqref{r3:itm} hold by virtue of far-commutativity relations in either 
algebra $\pE$ or $\pF$ and the gluing relations \eqref{eq:gluingproofg2}.

The substance of the proof is to show that $\a$ respects the \eqref{r2:itm} relations. In the notation of this section, these appear as follows:
\begin{align}
[\E_{i+1,k},\E_{i,h(i)}]_{q^{p(k+1)}} &= 0  &  \normaltext{ for } & k > 1 \label{eq:gluingr2a}\\
[\E_{i+1,k},\E_{i,h(i)}]_{q^{p(k)}} &= 0   & \normaltext{ for } & k<1\notag\\
[\E_{i+m-1,\fol i {i+m-1} e}, \cdots, \E_{i+k, \fol i {i+k} e},\cdots, \E_{i+1, \fol i {i+1} e}]_q &=\E_{i,1} & & \label{eq:gluingr2b}
\end{align}
Since there are two types of these relations the proof will contain two
independent arguments.

The relations \eqref{eq:gluingr2a} are examined first.
Now if $1 \leq i \leq n-2$ or $n \leq i \leq m+n-3$ then $\alpha$ respects these relations 
because they are imposed in the individual disk algebras $\pE$ and $\pF$ and 
the foliation data \eqref{eq:gluingfoliation} agrees.
By symmetry, switching $\pE$ and $\pF$, the cases $i=n-1$ and $i=m+n-2$ are
equivalent, so it suffices to check only the first. 
The two relations in \eqref{eq:gluingr2a} will be addressed simultaneously using the function $a = a(\ell)$ below.
\begin{equation}\label{eq:gluingadef}
a := \left\{ 
\begin{array}{cl}
q^{-p(\ell)} & \textrm{if } \ell > 1\\
q^{p(\ell)} & \textrm{if } \ell < 1.
\end{array}\right.
\end{equation}

Now, in order  to show that $\a$ satisfies Eqns. \eqref{eq:gluingr2a} above,  we compute $\alpha\left([G_{n,\ell},G_{n-1,g(n-1)}]_{a}\right)$ to be
\begin{align*}
&= [F_{n,\ell},E_{n-1,e(n-1)+f(n-1)}]_{a} 
& \eqref{eq:gluingfoliation}\\
&= [F_{n,\ell},\sigma^s E_{n-1,0}]_{a} 
& (\mathrm{defines}\,s)\\
&= [F_{n,\ell},\sigma^{-1+s} 
    [E_{n-2, \fol {n-1} {n-2} e},\cdots ,E_{1,\fol {n-1} 1 e},E_{n,\fol {n-1} n e}]_q]_{a}
    & \eqref{eq:gluingr2b}\\
&= [F_{n,\ell},\sigma^{-1+s} 
    [[E_{n-2, \fol {n-1} {n-2} e},\cdots ,E_{1, \fol {n-1} 1 e}]_q,E_{n,\fol {n-1} n e}]_q]_{a}
    & \eqref{orderprop}\\
&= [F_{n,\ell},[y,\sigma^{-1+s}E_{n, \fol {n-1} n e}]_q]_{a}
&(\mathrm{defines}\,y)\\
&= [F_{n,\ell},[y,\sigma^{-1+e(n-1)+f(n-1)}E_{n,1-e(n-1)}]_q]_{a}\\
&= [F_{n,\ell},[y,E_{n,f(n-1)}]_q]_{a}\\
&= [F_{n,\ell},[y,F_{n-1,f(n-1)}]_q]_{a}
&\eqref{eq:gluingproofg1}\\
&= [x,[y,z]_q]_{a}
&(\mathrm{defines}\,x\, \&\, z)\\
&= [y,[x,z]_{a}]_q
&\eqref{eq:comm2}\,\&\, \eqref{eq:gluingproofg2}\\
&= [y,[F_{n,\ell},F_{n-1,f(n-1)}]_{a}]_q\\
&= 0
&\eqref{r1:itm}\,\&\,\eqref{eq:gluingadef}
\end{align*}
Here we used the fact that all the generators involved in the definition of
$y$ commute with $F_{n,\ell}$ when we applied \eqref{eq:comm2}. We also used
the definition of the constant $a$ when we applied \eqref{r1:itm} in
the last step. This shows that $\a$ satisfies equations \eqref{eq:gluingr2a} above.

\vspace{.10in}
Before checking the next relation, recall that the cyclic relations become
\begin{equation}\label{eq:gluingcyc}
\sigma [E_{i + \ell, \fol i {i+\ell} e },\cdots E_{i, \fol i i e }]_q = 
[E_{i+m-1, \fol i {i+m-1} e}, \cdots ,E_{i+\ell+1,\fol i {i+\ell+1} e}]_q
\end{equation}
The same argument from Lemma \ref{cyclicprop2} shows that for any $i$ the relation \eqref{eq:gluingr2b} is equivalent to the relation \eqref{eq:gluingcyc}. Therefore, in order to prove each instance of the relation \eqref{eq:gluingr2b}, it suffices to prove an equivalent version of the cyclic relation.

It is time to prove that $\a$ respects the relation \eqref{eq:gluingr2b}.
The cases $1 \leq i \leq n-1$ and $n \leq i \leq n+m-2$ are equivalent, by switching $\pE$ and $\pF$, so it suffices to check only the first. We will prove the following identity:
\begin{equation}\label{eq:gluingtoprove}
\alpha\left( \sigma [G_{n-1,\fol i {n-1} g },\cdots ,G_{i,\fol i i g}]_q\right) \stackrel{?}{=} \alpha\left([G_{i-1,\fol i {i-1} g},\cdots,G_{n, \fol i n g}]_q \right).
\end{equation}

The elements $X_i$ and $Y$ below will play a key role in intermediate steps of the computation.
\begin{align*}
X_i &:= [G_{i-1,
\fol 1 {i-1} g},\cdots,G_{1,
\fol 1 1 g}]_q\\
Y &:= [G_{n+m-2,
\fol n {n+m-2} g},\cdots,G_{n,
\fol n n g}]_q
\end{align*}
Since the shifts in the elements defining $X_i$ and $Y$ are entirely contained $\pE$ and $\pF$ respectively, one can compute the effect of $\alpha$ on $X_i$ and $Y$ as follows:
\begin{align}
\alpha(X_i) &= [E_{i-1,
\fol 1 {i-1} e},\cdots, E_{1,
\fol 1 1 e}]_q\label{eq:gluingax}\\
\alpha(Y) &= [F_{n+m-2,
\fol n {n+m-2} f},\cdots,F_{n,
\fol n n f}]_q\notag\\
&= \sigma^{-
\fol {n-1} n f}[F_{n+m-2,
\fol n {n+m-2} f},\cdots,F_{n,
\fol {n-1} n f}]_q \notag\\
&= \sigma^{-(1-f(n-1))}F_{n-1,1}
&\eqref{eq:gluingr2b}\notag\\
&= \sigma^{f(n-1)}E_{n,0}
&\eqref{eq:gluingproofg1} \label{eq:gluingay}
\end{align}

We now simplify the right-hand side of the equation \eqref{eq:gluingtoprove} before applying $\alpha$. Observe that the right-hand side can be written as
\begin{align*}
&= [G_{i-1,
\fol i {i-1} g},\cdots,G_{n,
\fol i n g}]_q \\
&= \sigma^{
\fol i n g}[G_{i-1,
\fol n {i-1} g},\cdots,G_{n,
\fol n n g}]_q\\
&= \sigma^{
\fol i n g}[G_{i-1,
\fol n {i-1} g},\cdots,
G_{1,
\fol n 1 g},G_{n+m-2,
\fol n {n+m - 2} g},\cdots,G_{n,
\fol n n g}]_q\\
&= \sigma^{
\fol i n g}[\sigma^{
\fol n 1 g} X_i, Y]_q\\
&= \sigma^{
\fol i n g}[\sigma^{f(n-1) + \fol n  1 e} X_i, Y]_q
&\eqref{eq:gluingfolid2}
\end{align*}

Using \eqref{eq:gluingax} and \eqref{eq:gluingay}, one can compute $\a$ applied to the right-hand side of the equation \eqref{eq:gluingtoprove}
\begin{align*}
\a(\sigma^{\fol i n g}[\sigma^{f(n-1) + \fol n  1 e} X_i, Y]_q)
&= \sigma^{
\fol i n g}[\sigma^{f(n-1)+
\fol n 1 e}\alpha(X_i), \sigma^{f(n-1)}E_{n,0}]_q\\
&= \sigma^{
\fol i n g + f(n-1)}[[E_{i-1,
\fol n {i-1} e},\cdots,E_{1,
\fol n 1 e}]_q,E_{n,0}]_q\\
&= \sigma^{
\fol i n g+f(n-1)}[E_{i-1,
\fol n {i-1} e},\cdots,E_{n,
\fol n n e}]_q
&\eqref{orderprop}\\
&= \sigma^{
\fol i n e} [E_{i-1,
\fol n {i-1} e},\cdots,E_{n,
\fol n n e}]_q
&\eqref{eq:gluingfolid2}\\
&= [E_{i-1,
\fol i {i-1} e},\cdots,E_{n,
\fol i n e}]_q\\
&= \sigma[E_{n-1,
\fol i {n-1} e},\cdots,E_{i,
\fol i i e}]_q
&\eqref{eq:gluingcyc}
\end{align*}
The last line above is the left-hand side of the equation \eqref{eq:gluingtoprove}.

This completes the proof of Theorem \ref{gluethm}.

\subsubsection{Associativity of gluing}\label{asssec}
\newcommand{\dsk}[1]{\Alg(D^2_{#1}, \A_{n_{#1}})}
\newcommand{\adsk}[1]{(D^2_{#1}, \A_{n_{#1}})}
\newcommand{\dk}[1]{\Alg_{#1}}
\newcommand{\psz}{1.823} 

When a collection of disks is glued to form a larger disk, an iterated $\aprod$-product of algebras $\dsk{i}$ must be performed. The next proposition shows that the result is independent of the order in which the $\aprod$-products are performed. 

\begin{prop}\label{assprop}
Suppose that $\adsk{1}\sqcup_{i_1,j_1} \adsk{2} \sqcup_{i_2,j_2} \adsk{3}$ is a gluing of three disks and set $\dk{i} = \dsk{i}$. Then there are isomorphisms
  $$a_{1,2,3} : \left(\dk{1}\aprod_{i_1,j_1}\dk{2} \right) \aprod_{i_2,j_2} \dk{3} \xto{\sim} \dk{1}\aprod_{i_1,j_1}\left(\dk{2}\aprod_{i_2,j_2}\dk{3}\right).$$
For any gluing of four disks, the following pentagon diagram commutes:
  \begin{equation*}
  \begin{tikzpicture}[scale=10,node distance=2.5cm]
  \node (A1) {$(\dk{1}\aprod\dk{2})\aprod(\dk{3}\aprod\dk{4})$};
  \node (B) [below=\psz*0.691 cm of A1] {};
  \node (A2) [left=\psz*.951cm of B] {$( (\dk{1}\aprod\dk{2}) \aprod \dk{3}) \aprod \dk{4}$};
  \node (A5) [right=\psz*.951cm of B] {$\dk{1}\aprod (\dk{2} \aprod (\dk{3}\aprod \dk{4}))$};
  \node (C) [below=\psz*1.118cm of B] {};
  \node (A3) [left=\psz*.5879cm of C] {$(\dk{1} \aprod (\dk{2}\aprod \dk{3}))\aprod \dk{4}$};
  \node (A4) [right=\psz*.5879cm of C] {$\dk{1} \aprod ((\dk{2}\aprod\dk{3})\aprod \dk{4})$};
  \draw[->] (A2) to node {} (A1);
  \draw[->] (A1) to node {} (A5);
  \draw[->] (A4) to node {} (A5);
  \draw[->] (A3) to node {} (A4);
  \draw[->] (A2) to node {} (A3);
  \end{tikzpicture}
  \end{equation*}  
\end{prop}
\begin{proof}
  The algebras $\left(\dk{1}\aprod \dk{2} \right) \aprod \dk{3}$ and $\dk{1}\aprod \left(\dk{2}\aprod\dk{3}\right)$ are isomorphic to the algebra $\dk{1}\aprod \dk{2} \aprod \dk{3}$ with generators $\ZZ\inp{\A_{n_1} \sqcup \A_{n_2} \sqcup \A_{n_3}}$ subject to relations \eqref{r1:itm}, \eqref{r2:itm} and \eqref{r3:itm}. The map $a_{1,2,3}$ is given by composing one isomorphism with the inverse of the other. 

The commutativity of the pentagon diagram follows from the observation that each vertex is isomorphic to the algebra with generators $\ZZ\inp{\A_{n_1} \sqcup \A_{n_2} \sqcup \A_{n_3}\sqcup \A_{n_4}}$ and these isomorphisms commute with the maps forming the boundary of the pentagon.
\end{proof}

It is possible to show that the maps $\a$ and $\b$ introduced by the gluing Thm. \ref{gluethm} commute with the associators introduced above. This has been omitted as it will not be needed in what follows.

\subsection{The Hall algebra of the disk}\label{skdisk}
Here the gluing theorem from the previous section is used to define an
algebra $\Alg(D^2,A)$ associated to the Hall algebra of
$D^\pi\F(S,A)$.

\begin{defn}\label{genalgdef}
Suppose that $(D^2,m)$ is a disk with $m$ marked intervals equipped with
grading $\eta$ and $A$ is a full arc system. By definition, the internal arcs of $A$ cut the disk into a collection of disks, so
$$(D^2, A) = (D^2_1,\A_{m_1}) \sqcup_{i_1,j_1} (D^2_2,\A_{m_2})
\sqcup_{i_2,j_2} \cdots \sqcup_{i_{\ell-1},j_{\ell-1}}
(D^2_\ell,\A_{m_\ell})$$ with minimal arc systems equipped with foliation
data $h_k : \A_{m_k} \to \ZZ$. 
The numbers $m_k$ satisfy $\sum_{k=1}^\ell m_k = m + 2p$, where $p$ is the number
of internal arcs in $A$. Since the disk $(D^2,A)$ can be reassembled by
gluing the $k$ disks together, identifying the $i_k$th arc in
$(D^2_k,\A_{m_k})$ with the $j_k$th arc in $(D^2,\A_{m_{k+1}})$, the gluing
theorem allows us to construct the presentation of {\em algebra $\Alg(D^2,A)$ associated to the
  arc system $A$} as
$$\Alg(D^2,A) = \Alg(D^2,\A_{m_1}) \aprod_{i_1,j_1} \Alg(D^2_2,\A_{m_2}) \aprod_{i_2,j_2} \cdots \aprod_{i_{\ell-1},j_{\ell-1}} \Alg(D^2_\ell,\A_{m_\ell}).$$
Prop. \ref{assprop} ensures that the algebra is independent of the order in which the products are taken.
\end{defn}

The next proposition records that the algebras $\Alg(D^2,A)$ compute
the Hall algebras of the Fukaya categories $D^\pi\F(D^2,A)$.

\begin{prop}\label{fullarcprop}
Suppose that $(D^2,m)$ is a graded disk with $m$ marked intervals equipped 
with a full arc system $A$. Then the map $\kappa' : F(\ZZ A)\to \tDHa (D^\pi\F(D^2,A))$ induces an isomorphism $\bar{\kappa}' : \Alg(D^2,A) \xto{\sim} \tDHa( D^\pi\F(D^2,A))$.
\end{prop}
\begin{proof}

This follows from Thm. \ref{minarcthm} and Cor. \ref{sheafpropcor}.

\end{proof}

\subsection{A conjectural presentation of composition subalgebras}
The gluing construction in Section \ref{genalgdef} can be applied to an arbitrary marked surface. In some cases, the resulting algebra is not the composition subalgebra $\Alg(S,A)$.
However, we conjecture that when the surface $(S,M)$ has enough marked intervals, the algebra obtained by gluing is the composition subalgebra. This constitutes a broad class of surfaces. In particular, by subdividing marked intervals any surface can be refined to satisfy this criteria.

\begin{definition}\label{def:naive}
Suppose that $(S,M)$ is a graded marked surface and $A$ is a full arc system. By definition, the internal arcs of $A$ cut the surface into a collection of disks:
$$(S, A) = (D^2_1,\A_{m_1}) \sqcup_{i_1,j_1} (D^2_2,\A_{m_2}) \sqcup_{i_2,j_2} \cdots \sqcup_{i_{\ell-1},j_{\ell-1}} (D^2_\ell,\A_{m_\ell})$$ 

Since the surface $(S,A)$ can be reassembled by
gluing the $k$ disks together, identifying the $i_k$th arc in
$\A_{m_k}$ with the $j_k$th arc in $\A_{m_{k+1}}$, the gluing
theorem allows us to construct the presentation of the {\em naive algebra} $\NAlg(S,A)$:
$$\NAlg(S,A): = \Alg(D^2,\A_{m_1}) \aprod_{i_1,j_1} \Alg(D^2_2,\A_{m_2}) \aprod_{i_2,j_2} \cdots \aprod_{i_{\ell-1},j_{\ell-1}} \Alg(D^2_\ell,\A_{m_\ell}).$$
In words, $\NAlg(S,A)$ is the free product of the disk algebras $\Alg(D^2_k,\A_{m_k})$  subject to the \eqref{g1:itm} relation   that arcs $a$ and $a'$ which are equal in $S$ are equal in $\NAlg(S,A)$ and the \eqref{g3:itm} relation that arcs $a$ and $a'$ that don't end on the same marked interval commute in $\NAlg(S,A)$.
\end{definition}
\begin{definition}\label{def:enough}
The marked surface $(S,M)$ above is said to have {\em enough marked intervals} when, for each $k$, the inclusion $\iota_k : (D^2_k,M_k)\to (S,M)$ 
is injective on connected components.
\begin{equation}\label{eq:embcond}
(\iota_k)_*: \pi_0(M_k) \to \pi_0(M) \textrm{ is injective}
\end{equation}
\end{definition}

When this criteria is satisfied, there is a map from the naive algebra to the composition subalgebra.

\begin{thm}\label{prop:naivemaps}
Suppose that  $(S,M)$ is a graded marked surface with a full arc system $A$. If $(S,M)$ has enough marked intervals then there is a surjective map 
$$\ga : \NAlg(S,A) \to \Alg(S,A).$$
\end{thm}
\begin{proof}
The map $\ga$ is determined on generators by mapping each arc $a\in A$ to $a\in A$. Since this is so, if $\ga$ is a homomorphism then it is surjective.

Since $(S,M)$ has enough marked intervals, Cor. \ref{fdiskembcor} implies
that each map $\Alg(D_k,\A_k)\to \Alg(S,A)$ is a
monomorphism. Since each such map factors through $\ga$ above, it suffices
to check that the relations \eqref{g1:itm} and \eqref{g3:itm} from
Def. \ref{alggluedef} hold in $\Alg(S,A)$. The \eqref{g1:itm} relations hold
because arcs which are equal in the gluing are equal in the surface by
definition. The \eqref{g3:itm} relations hold because if $X$ and $Y$ are
arcs in $A$ which don't end on the same marked interval in $M$ then
$$\Hom^*(X,Y) = 0 \conj{ and } \Hom^*(Y,X) = 0$$
and so the equations for $F^L_{X,Y}$ in Def. \ref{toenalgdef} show that
$X$ and $Y$ commute in $\Alg(S,A)$. 
\end{proof}

Proposition \ref{fullarcprop} proves the following conjecture in the case that $S$ is a disk.
\begin{conjecture}\label{conjconj}
If the surface $(S,M)$ has enough marked intervals then the map $\ga$ is an isomorphism.
\end{conjecture}

\begin{example}
When the surface $S$ is an annulus $S^1\times [0,1]$ with two marked intervals on each boundary component, there is an arc system $A$ consisting of two boundary arcs intervals $E_1$, $F_1$  and $E_3,F_3$ between each of the two marked intervals and two internal arcs $E_2=F_4$ and $E_4=F_2$. This is pictured below. 
\vspace{.1in}
\begin{center}\label{fig:annuluspresex}
\begin{overpic}[scale=1]
{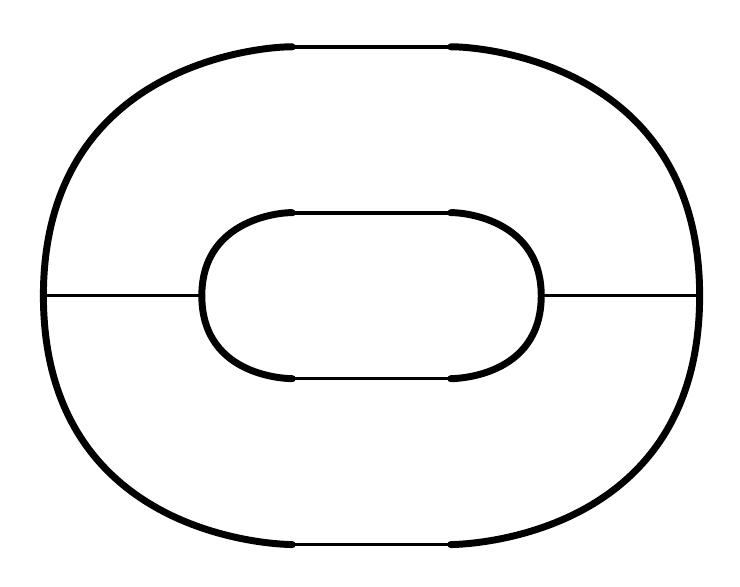}
\put(102,160){$E_1$}
\put(102,0.5){$F_1$}

\put(102,95.5){$E_3$}
\put(102,65.5){$F_3$}

\put(31,88){$E_4$}
\put(31,73){$F_2$}

\put(173,88){$E_2$}
\put(173,73){$F_4$}
\end{overpic}
\end{center}
\vspace{.1in}



Cutting the annulus along the arc system produces two disks $(D^2_1,\A_4)$ and $(D^2_2,\A'_4)$ with four marked intervals $\A_4 = \{E_1,E_2,E_3,E_4\}$ and $\A'_4 = \{F_1,F_2,F_3,F_4\}$. Since there are enough arcs, the conjecture above implies that the composition algebra $\Alg(S,A)$ 
has a presentation generated by suspensions the arcs $E_{i,n} = \s^nE_i$ and $F_{j,m}= \s^mF_j$ for $1\leq i,j\leq 4$ and $n,m\in\ZZ$ subject to the \eqref{g1:itm} and \eqref{g3:itm} relations:
\begin{align*}
E_{2,n}=F_{4,n}  \conj{ and } E_{4,n} = F_{2,n} & \textnormal{ for all } n \in \ZZ,\\
[E_{2,n},F_{2,m}]_1 = 0  \conj{ and }  [F_{2,n},E_{2,m}]_1 = 0 & \textnormal{ for all } n,m\in\ZZ,
\end{align*}
as well as the relations which hold within each disk algebra $\Alg(D^2_1,\A_4)$ and $\Alg(D^2_2,\A'_4)$, see Def. \ref{halldiskdef}.

\end{example}

\vskip .25in
\section{The HOMFLY-PT skein relation}\label{skeinrelsec}

Since the Fukaya Hall algebra $\SAlg(S,M)$ of a finitary surface combines the
composition subalgebras $\Alg(S,A)$ of many different arc systems $A$, there
is the possibility for relations between graded arcs $X$ and $Y$ which
appear in different arc systems $A$ and $A'$.  If there is an
arc system in which $X$ and $Y$ are disjoint then $X$ and $Y$ must commute
in the Fukaya Hall algebra.  When $X$ and $Y$ intersect at a point $p$, they
cannot be contained in the same arc system and the purpose of this section
is to establish Theorem \ref{skeinrelcor}, which shows that they must satisfy
a commutation relation that is determined by the Lagrangian surgery at
$p$. This relation is a graded version of the HOMFLY-PT
skein relation, which is
$$\CPPic{ncross} -  \CPPic{pcross} = (q-q^{-1}) \CPPic{orres}$$

This section consists of two parts. The first contains a direct computation in the algebra $\Alg(D^2,\A_4)$. In the second part, Proposition \ref{cor:invariantskein} shows that this computation can be written using intersection indices of graded curves.  Theorem \ref{skeinrelcor} leverages these materials together with the embedding criteria in Corollary \ref{skdiskembcor} to show that this graded skein relation holds in the Fukaya Hall algebras of arbitrary surfaces.

\subsection{A local calculation}\label{localcalcsec}
This section contains a local calculation in the disk algebra $\Alg(D^2,\A_4)$ which will be used by Theorem \ref{skeinrelcor} to prove a graded skein relation on surfaces.

\begin{lemma}\label{skeinthmthm}
In the disk algebra $\Alg(D^2,\A_4)$,
\begin{equation}\label{eq:homflyskein}
[X,\s^\ell Y]_1 = (q-q^{-1})\delta_{\ell,1} E_{2,1}E_{4,h(4)+h(1)} + (q^{-1}-q)\delta_{\ell,0}E_{1,h(1)}E_{3,1-h(2)},
\end{equation}
where $\ell\in\ZZ$, $\A_4=\{E_1,E_2,E_3,E_4\}$ are the boundary arcs,  
$$X := [E_{2,1}, E_{1,h(1)}]_q \conj{ and } Y := [E_{3,1-h(2)}, E_{2,0}]_q.$$
\end{lemma}

\begin{proof}

In order to simplify the expression $[ X,\s^\ell Y]_1 = [X, [E_{3,1-h(2)+\ell}, E_{2,\ell}]_q]_1$, write the omni-Jacobi identity \eqref{oj:itm}$[a,a^{-1},qa^{-1}]$:
$$[x,[y,z]_{q}]_1 = -a[z,[x,y]_{qa^{-1}}]_{a^{-1}} - [y,[z,x]_{qa^{-1}}]_{a}$$
Applied to $[X,\sigma^\ell Y]$, this gives
\begin{equation}\label{eq:skeincomm1}
[X,\s^\ell Y]_1 = 
-a[E_{2,\ell}, [X, E_{3,1-h(2)+\ell}]_{qa^{-1}}]_{a^{-1}}  -[E_{3,1-h(2)+\ell},[E_{2,\ell}, X]_{qa^{-1}}]_{a}.
\end{equation}
Recall that $p(\ell) = (-1)^\ell$. In what follows, the function $a=a(\ell)$ is defined as
\begin{equation}\label{eq:aspecial}
a := \left\{ \begin{array}{cl}
q^{1+p(\ell+1)}&\textrm{if }\ell \geq 1\\
q^{1+p(\ell)}&\textrm{if }\ell < 1
\end{array}\right. 
\end{equation}
Let $P$ and $Q$ be the first and second summand of  \eqref{eq:skeincomm1} respectively.

{\it Term P:} If
$P=-a[E_{2,\ell}, P']_{a^{-1}}$ then 
the inner bracket $P'$ is
\begin{align*}
  P' &= [X, E_{3,1-h(2)+\ell}]_{qa^{-1}} \\
  &= \s^\ell [[E_{2,1-\ell},E_{1,h(1)-\ell}]_q,E_{3,1-h(2)}]_{qa^{-1}}\\
  &= \s^\ell [[E_{2,1-\ell},E_{3,1-h(2)}]_{qa^{-1}},E_{1,h(1)-\ell}]_q
  & \eqref{eq:comm1}\\
  &= -qa^{-1}\s^\ell [[E_{3,1-h(2)}, E_{2,1-\ell}]_{q^{-1}a},E_{1,h(1)-\ell}]_q\\
  &= -qa^{-1}\s^\ell [\s^{1-h(2)-\ell}[E_{3,\ell}, E_{2,h(2)}]_{q^{-1}a},E_{1,h(1)-\ell}]_q
\end{align*}

By \eqref{r2:itm} and our choice of $a$ in \eqref{eq:aspecial}, 
the inner bracket in this last expression is zero unless $\ell=1$. 
If $\ell = 1$ then $a = q^2$, and in this case
\begin{align*}
P' &= -q^{-1}\s^{1} [\s^{-h(2)}[E_{3,1}, E_{2,h(2)}]_{q},E_{1,h(1)-1}]_q\\
&= -q^{-1} \s^{h(1)} [E_{3,2-h(1)-h(2)}, E_{2,1-h(1)}, E_{1,0} ]_q& \eqref{orderprop} \\
&=-q^{-1} E_{4,h(4)+h(1)}
\end{align*}

Since $E_{2,1}$ and $E_{4, h(4)+h(1)-1}$ commute, this implies
\begin{equation}\label{eq:terma}
P = (q-q^{-1})\delta_{\ell,1} E_{2,1}E_{4,h(4)+h(1)}.
\end{equation}

{\it Term Q:} The cyclic equations in Lemma \ref{cyclicprop2} include
$$  X = [E_{2,1}, E_{1,h(1)}]_q = [E_{4,\tau^1\inp{3}}, E_{3,\tau^1\inp{2}}]_q  = [E_{4,2-h(2)-h(3)}, E_{3,1-h(2)}]_q.$$
Now if $Q=-[E_{3,1-h(2)+\ell},Q']_{a}$ then the innermost bracket $Q'$ is
\begin{align*}
  Q' &= [E_{2,\ell}, X]_{qa^{-1}} \\
&= \s^\ell [ E_{2,0}, [E_{4,2-h(2)-h(3)-\ell}, E_{3,1-h(2)-\ell}]_q]_{qa^{-1}}\\
&= \s^\ell [ E_{4,2-h(2)-h(3)-\ell}, [E_{2,0}, E_{3,1-h(2)-\ell}]_{qa^{-1}}]_q & \eqref{eq:comm2}\\
&= -qa^{-1} \s^{\ell} [ E_{4,2-h(2)-h(3)-\ell}, [E_{3,1-h(2)-\ell},E_{2,0}]_{q^{-1}a}]_q\\
&= -qa^{-1} \s^{\ell} [ E_{4,2-h(2)-h(3)-\ell}, \s^{-h(2)}[E_{3,1-\ell},E_{2,h(2)}]_{q^{-1}a}]_q
\end{align*}
Set $j=1-\ell$. If $j>1$ then $\ell < 0$ and $q^{-1}a = q^{-1}q^{1+p(\ell)} = q^{p(j+1)}$
so \eqref{r2:itm} shows that the inner bracket vanishes. Similarly, if $j < 1$ then
$\ell \geq 1$, so $q^{-1}a = q^{-1}q^{1+p(\ell+1)} = q^{p(j)}$ and \eqref{r2:itm} again
shows that the inner bracket vanishes. If $j=1$ then $\ell=0$ and
$q^{-1}a = q$, and in this case
\begin{align*}
Q' &= -q^{-1} [E_{4,2-h(2)-h(3)},\s^{-h(2)}[E_{3,1}, E_{2,h(2)}]_{q}]_{q}\\
&= -q^{-1} [E_{4,2-h(3)-h(2)}, E_{3,1-h(2)}, E_{2,0}]_q & \eqref{orderprop} \\
&= -q^{-1} E_{1,h(1)}
& \eqref{r2:itm}
\end{align*}
Lastly, since $E_{1,h(1)}$ and $E_{3,-h(2)}$ commute, 
\[
Q = q^{-1}[E_{3,1-h(2)}, Q']_{q^2} = (q^{-1}-q)\delta_{\ell,0}E_{1,h(1)}E_{3,1-h(2)}.
\]

\end{proof}

\begin{remark}\label{rmk:quiverskein}
Using the isomorphism in Theorem \ref{minarcthm}, the skein relations can be translated to the Hall algebra $\tDHa(A_{m-1})$, where they become
\begin{align*}
[z_{(a,c),0},z_{(b,d),0}]_1 &= (q-q^{-1})z_{(a,d),0}z_{(b,c),0},\\
[z_{(a,c),1},z_{(b,d),0}]_1 &= (q^{-1}-q)z_{(a,b),1}z_{(c,d),0},\\
[z_{(a,c),k},z_{(b,d),0}]_1 &= 0 & \textrm{ when } \vnp k \geq 2.
\end{align*}
\end{remark}

\subsection{The graded HOMFLY-PT skein relation}\label{grhomflyptsec}
\renewcommand{\fo}{e_0}
\newcommand{\pcap}{\pitchfork}

The purpose of this section is to translate Eqn.  \eqref{eq:homflyskein}
into a relation which depends only on the the intersection index $i_p(X,Y)$
of the two curves at the point $p$ of the intersection.

In order to write Lemma \ref{skeinthmthm} entirely in terms of intersection
indices, the lemmas below will use a number of explicit choices.  This {\em standard form} makes all of the materials very concrete and down-to-earth.

\begin{defn}\label{stdef}
  A foliation of the disk $(D^2,4)$, viewed as a subset of the plane,
  is in {\em standard form} when all of the lines are
  vertical. The illustration below depicts the basic setup.

\vspace{.15in}
\begin{center}\label{fig:ppic3}
\begin{overpic}[scale=0.6]
{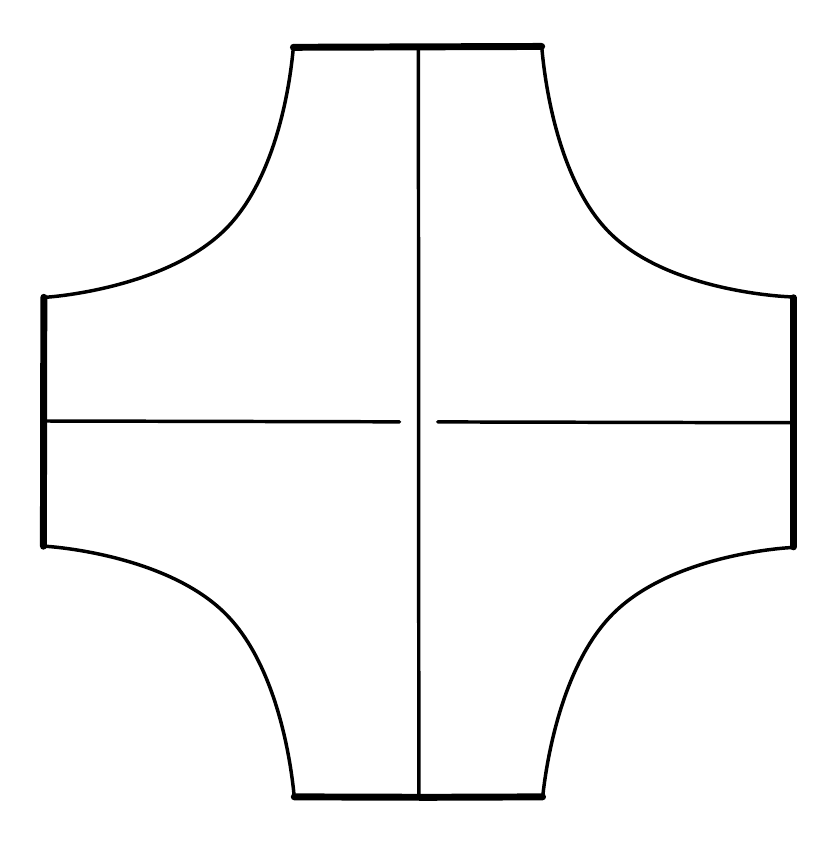}
 \put(68,142){$a_2$}
 \put(68, 0){$a_4$}
 \put(-5,70){$a_1$}
 \put(140,70){$a_3$}
 \put(105,115){$E_3$}
 \put(25,115){$E_2$}
 \put(25,20){$E_1$}
 \put(105,20){$E_4$}
 \put(77.5,120){$X'$}
 \put(17.5,60){$Y'$}
\end{overpic}
\end{center}
\vspace{.15in}

The picture features boundary arcs $E_i$, boundary paths $a_i : E_i\to E_{i+1}$ and two internal arcs: a vertical arc $X'$ and a horizontal arc $Y'$. The foliation is in standard form: north-south and parallel to the arc $X'$. 
The gradings of $X'$, $a_1$ and $a_3$ are given by the trivial paths
everywhere. The gradings of $Y'$ and the boundary arcs $E_i$ are given
pointwise by the shortest counterclockwise path from the foliation $\eta$ to
the tangent vector to the curve. The intersection indices can be computed
using Eqn. \eqref{indexeqn}:
\begin{align*}
i(E_1,a_1)&=1, & i(E_2,a_1) &= 1, &  i(E_2,a_2)&=1,& i(E_3,a_2)&=0\\
i(E_3,a_3)&=1,& i(E_4,a_3)&=1,& i(E_4,a_4)&=1,&i(E_1,a_4)&=0.
\end{align*}
The associated foliation data is as follows:\\
\begin{equation}\label{eq:efol}
\fo(1)=0,\quad \fo(2)=1,\quad \fo(3)=0\conj{and} \fo(4)=1.
\end{equation}
\end{defn}

Before proceeding it is important to see why one can assume the foliation is in standard form without loss of generality. For each $i\in\ZZ/4$, the suspension map
$$\s_i(E_j) = \left\{ \begin{array}{cl} \s E_i &\textrm{if } i = j \\ E_j &\textrm{if } i \ne j \end{array}\right.$$
induces an equivalence $\s_i : \F(D^2,\A_4,h) \xto{\sim} \F(D^2,\A_4,\s_ih)$ where $h : \A_4\to\ZZ$ is arbitrary foliation data and $\s_ih$ is the suspended foliation data
$$(\s_ih)(j) = \left\{ \begin{array}{cl} h(j) + 1 &\textrm{if } i = j + 1 \\ h(j) - 1  &\textrm{if } i = j \\ h(j) &\textrm{if } i\ne j,j+1.\end{array}\right.$$
If $h : \A_4 \to \ZZ$ is arbitrary foliation data then the equation
$$h = \s_4^{1-h(4)} \s_3^{1-h(4)-h(3)} \s_2^{h(1)}\fo,$$
together with the bijection between homotopy classes of foliations and foliation data \cite{HKK} shows that, up to isomorphism, the foliation of $(D^2,4)$ is in standard form.

\begin{notation}\label{iamthelizardkingicandoanythingnotation}
If $X$ and $Y$ are curves and there is only one point of intersection
$X\pcap Y = \{p\}$ then the point $p$ will be dropped from notation. This is
to say, $i(X,Y) = i_p(X,Y)$ and $X\#Y = X\#_pY$ below. If $X$ and $Y$ are
graded arcs whose endpoints share a marked interval then $i(X,Y) = \vnp\ga$
will be used to denote the degree of the boundary path $\ga$ between them,
see Def. \ref{fukcatdef}.
\end{notation}

With the foliation in standard form, Lagrangian surgery is recalled, 
see \cite{Abo08,HKK}.

\begin{definition}\label{def:resolutions}
Suppose that $X$ and $Y$ are graded curves which intersect transversely at a
point $p$ such that $i_p(X,Y) = 1$. Then there is a graded curve $X\#_p Y$
called the {\em Lagrangian surgery} or the {\em resolution} of $X$ and $Y$ at $p$.
\vspace{.15in}
\begin{center}\label{fig:ppic1}
\begin{overpic}[scale=0.6]
{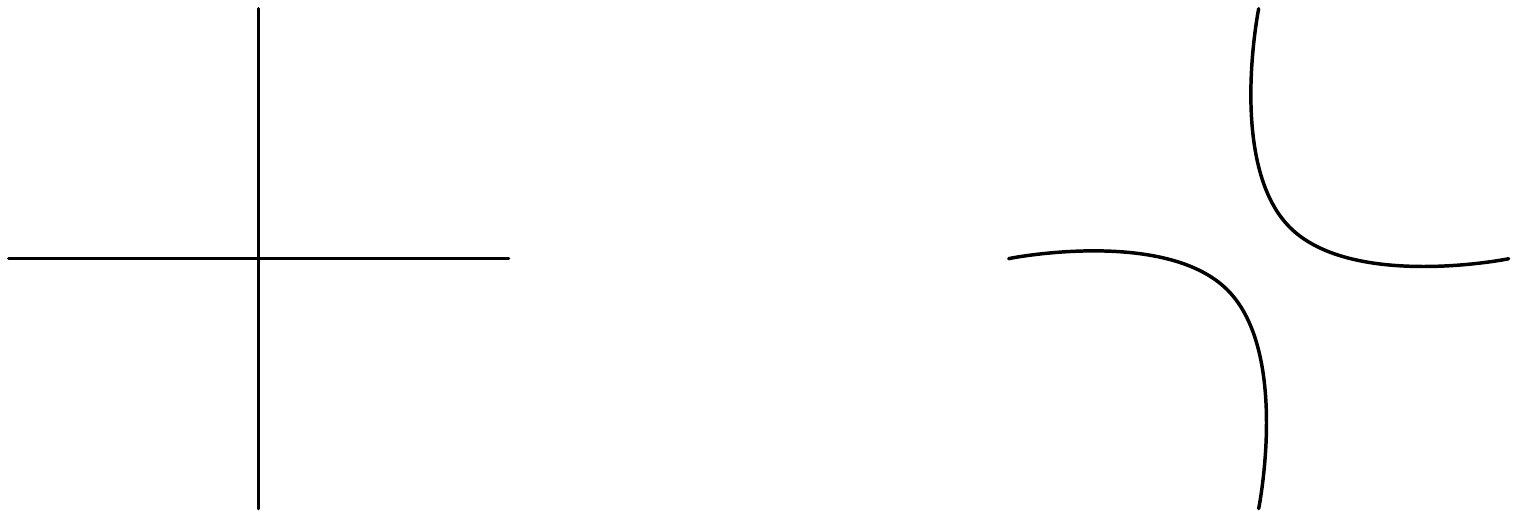}
\put(120,42.5){$\rightsquigarrow $}
\put(50,75){$Y$}
\put(3,30){$X$}
\put(225,75){$X\#_pY$}
\put(50,30){$p$}
\end{overpic}
\end{center}
\vspace{.15in}
The curve $X\#_pY$ is obtained from $X\sqcup Y$ by removing a neighborhood
of $p$ and connecting the endpoints.  The choice of resolution is determined
by the rule that if the intersection index satisfies $i_p(X,Y)=1$ then the
$X$ curve {\em turns right} as it approaches $p$ in the resolution.  This
condition on the intersection index ensures that the resolution has a
canonical grading. If $i_p(X,Y) = 0$ then $i_p(Y,X) = 1$ and $X$ {\em turns left} in the resolution.

For an explicit example, let the foliation in the above diagram be vertical, grade $Y$ trivially, and grade
$X$ with the shortest counterclockwise path\footnote{This grading is continuous because the curve $X$ is never tangent to the foliation.} from the foliation to the tangent line $\dot{X}$.  Then the grading of $X\#_p Y$ is given by a continuous
family of paths smoothly transitioning from the trivial path at the top of the pictured portion of $X\#_p Y$ to the counterclockwise rotation by $\pi/2$ on the rightmost part of the pictured portion. 

The graded resolution is also defined for curves which intersect the same marked boundary interval.
\begin{center}\label{fig:ppic2}
\begin{tikzpicture}
\node at (-2.25,0) {\CPPic{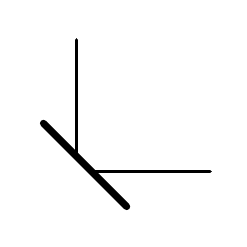}};
\node at (2.25,0) {\CPPic{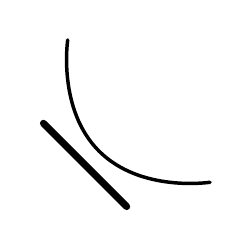}};
\node at (.3,0) {$\rightsquigarrow$};
\node at (2.7,.825) {$X\#Y$};
\node at (-2.075,.825) {$X$};
\node at (-1,-.825) {$Y$};
\end{tikzpicture}
\end{center}
  For a boundary intersection of arcs $Y$ and $X$, with $X$ clockwise
  from $Y$, there is only one choice of resolution and it has a canonical
  grading only when $i_p(Y,X) = 1$.

\end{definition}

The next two Lemmas use the choices above to identify $X'$ both as a
resolution of $E_1$ and $E_2$ and as a commutator in the Hall algebra.

Before proceeding, recall that the shift $\sigma$ acts by $\s E_{i,k} =
E_{i,k+1}$, or more generally, on a graded curve $c$ by pointwise
composition of the path $\tilde{c}$ from $\eta$ to the tangent line $\dot c$
with \emph{clockwise} rotation by $\pi$.

\begin{lemma}\label{lemma:xpyp}
In the Hall algebra $\SAlg(D^2,4)$,
\[
X' = [E_{2,1},E_{1,0}]_q \conj{ and } Y' = [E_{3,0},E_{2,0}]_q.
\]
\end{lemma}
\begin{proof}
The identity $X' = [E_{2,1},E_{1,0}]_q$, must hold in the triangle
\begin{equation}\label{eq:defsf}
F_1 := E_1,\quad F_2 := E_2\conj{ and } F_3 := X'
\end{equation}
which is contained in the standard form of the disk $(D^2,4)$, see Def. \ref{stdef}. The foliation data is given by
\[
f(1)=0,\quad f(2)=1 \conj{ and } f(3)=0.
\]
Within this disk $(D^2,3)$, Theorem \ref{minarcthm} implies the convolution relation
\begin{equation}\label{eq:idsf}
F_{3,1} = [F_{2,2-f(3)-f(1)},F_{1,1-f(3)}]_q = [F_{2,2},F_{1,1}]_q.
\end{equation}
By Cor. \ref{fdiskembcor}, the inclusion of the triangle into the disk $(D^2,4)$ induces a monomorphism determined on generators by Eqn. \eqref{eq:defsf}. The image of the identity \eqref{eq:idsf} under this map is $X' = [E_{2,1},E_{1,0}]_q$ after shifting by $\sigma$.
 The proof that $Y' = [E_{3,0},E_{2,0}]_q$ is similar and omitted.
\end{proof}

\begin{lemma}\label{lemma:gcurves}
The graded resolutions of $X'$ and $Y'$ are as follows.
\begin{equation}
Y'\# X' = E_{3} \sqcup E_{1} \conj{ and }\sigma Y' \# X' = \sigma E_2 \sqcup \sigma E_4
\end{equation}
\end{lemma}

The proposition below combines the above results with the local computation in Lemma \ref{skeinthmthm}.
\begin{prop}\label{cor:invariantskein}
In the Hall algebra  $\SAlg(D^2,4)$, the curves $X'$ and $Y'$ satisfy the relation below.
\begin{equation}\label{eq:skeinint}
[X',Y']_1 = \left\{
\begin{array}{cl}
(q-q^{-1}) X'\#Y' & \mathrm{if }\,\, i(X',Y')=1\\
-(q-q^{-1}) X'\#Y' & \mathrm{if }\,\, i(X',Y')=0\\
0&\mathrm{otherwise}
\end{array}\right.
\end{equation}
\end{prop}
\begin{proof}
  Using the foliation data in standard form Eqn. \eqref{eq:efol}, Lemma \ref{skeinthmthm} becomes
  \begin{equation}\label{stdformeqn}
    [X,\s^\ell Y]_1 = (q-q^{-1})\d_{\ell,1} E_{2,1} E_{4,1} - (q-q^{-1})\d_{\ell,0} E_{1,0} E_{3,0}.
\end{equation}    
By Lemma \ref{lemma:xpyp}, $X' = [E_{2,1}, E_{1,0}]_q$ and $Y' = [E_{3,0}, E_{2,0}]_q$ so that $X = X'$ and $Y=Y'$ by the assumptions of Lemma \ref{skeinthmthm}. Therefore $[X',Y']_1 = [X,\s^0Y]_1 = - (q-q^{-1}) E_{1,0} E_{3,0}$ using Eqn. \eqref{stdformeqn}. On the other hand,  $- (q-q^{-1}) E_{1,0} E_{3,0} = -(q-q^{-1})Y'\# X'$ by Lemma \ref{lemma:gcurves}. Since $i(X',Y')=0$ in standard form, the second equation in \eqref{eq:skeinint} above must hold. 

The rest of the argument is similar after suspending $Y'$ and using Eqn. \eqref{indexpropeqn}.

\end{proof}

There is also boundary skein analogue of the proposition above.

\begin{corollary}\label{cor:boundaryskein}
Suppose $X$ and $Y$ are graded arcs in a disk which intersect the same marked interval and that $X$ is clockwise from $Y$. Then in the algebra $\SAlg(D^2,m)$,
\begin{equation}
XY - q^{r(Y,X)} YX = 
\left\{ \begin{array}{cl}
X\# Y & \mathrm{if }\,\, i(Y,X)=1\\
0     &  \mathrm{if }\,\, i(Y,X)\ne 1\\
\end{array}\right.
\end{equation}
where 
\begin{equation}\label{reqn}
r(Y,X) := \left\{ \begin{array}{cl} 
(-1)^{i(Y,X)}&\mathrm{if }\,\, i(Y,X) \geq 1 \\
(-1)^{1+i(Y,X)}&\mathrm{if }\,\, i(Y,X) < 1. 
\end{array}\right.
\end{equation}
\end{corollary}
\begin{proof}
  The details are the same as Proposition \ref{cor:invariantskein} above except the equations
\[
X' =  E_1 \# \sigma E_2\conj{ and } Y' = E_3\# E_2
\]
are used in place of Lemma \ref{lemma:gcurves}. To see that they hold, one checks that $\iota(E_1,\sigma E_2) = 1$ and observes that the resolution in Definition \ref{def:resolutions} is isotopic to $X'$ as graded curves. 

\end{proof}

\begin{remark}\label{rmk:skeinskein}
We expect that the relations in Proposition \ref{cor:invariantskein} and Corollary \ref{cor:boundaryskein} can be used to define a graded version of the HOMFLY-PT skein algebra, but we will leave details to later work.
\end{remark}

Having established the graded skein relation, Eqn. \eqref{eq:skeinint}, in
the disk $(D^2,4)$, the theorem below extends this result to an arbitrary
finitary surface using the embedding criteria established earlier in Section
\ref{hallsurfsec}.

\begin{thm}\label{skeinrelcor}
Suppose that $a$ and $b$ are graded arcs 
whose endpoints intersect four distinct marked intervals in a finitary marked surface $(S,M)$. 
If $a$ and $b$ intersect uniquely at a point $p$ and $i_p(a,b) = 1$ then the graded HOMFLY-PT skein
relation
$$[a,b]_1 = (q-q^{-1})a\#_p b$$
holds in the Fukaya Hall algebra $\SAlg(S,M)$. 
\end{thm}

\begin{proof}
  The regular neighborhood $N(a\cup b)\subset S$ of the union of the arcs  
  $a$ and $b$ is a graded disk $(D^2,\A_4)$. The four marked intervals of
  this disk contain the endpoints of $a$ and $b$. This
  observation determines an inclusion
  $\iota : (N(a\cup b), 4, \A_4) \to (S, M, A)$ of marked surfaces where $A$
  is any arc system containing the boundary arcs of $N(a\cup b)$. Since this map
takes the four marked intervals of the disk to four distinct marked intervals of $M$, it induces an injection between the path components of the marked intervals.
Corollary \ref{skdiskembcor} implies that there is an
  associated monomorphism 
  $\iota_* : \SAlg(N(a\cup b),4) \to \SAlg(S,M)$. 
  Proposition \ref{cor:invariantskein}
  above shows that the skein relation holds in the Hall algebra
  $\SAlg(N(a\cup b),4)$ of the disk, so it must also hold in the Fukaya Hall algebra
  $\SAlg(S,M)$ of the surface $S$.
\end{proof}

\section*{Appendices}\label{appsec}
Appendix I relates the elements $z_{(a,b),n}\in \tDHa(A_{m-1})$ to Lusztig's
PBW basis. Appendix II contains an 
index of notation, a summary of
relations which hold among the elements $z_{(a,b),n}$ and some remarks about
$q$-commutators and $q$-Jacobi relations.

\subsection*{Appendix I: Relation to the PBW basis}\label{lusztigsec}
\def\U{U}
Lusztig has defined a Poincar\'e-Birkhoff-Witt (PBW) basis of the quantum
group $\U_v(\g)$ using a braid group action on $\U_v(\g)$ together with a
choice of reduced expression for the longest word in the Weyl group (see
 \cite{LQG}). In this section we recall an
alternative description Brundan gave for this basis in type $A$, for a
particular choice of reduced expression (see \cite[Sec. 3]{Bru98}). This
description shows that in the interpretation of $\U_v(\mathfrak{sl}_m)^-$ as
the (degree 0 part of the) Hall algebra of the Fukaya category of a disk,
Lusztig's PBW basis generators correspond to arcs between marked segments of
the boundary. We show that the HOMFLY skein relations in the Hall algebra of
the Fukaya category of the disk allow Lusztig's PBW generators to be
reordered, which gives a simple presentation for
$\U_v(\mathfrak{sl}_m)^-$. Some of these relations have appeared in the
literature before (see \cite{Bru98, Ros89}), but we include them since they
have a nice interpretation in terms of the Fukaya category, and since they
look particularly simple in the present notation.

Brundan uses the following notation for the quantum group $\U_v(\mathfrak{sl}_m)^-$: it is generated by elements $F_i$ for $1 \leq i \leq m$, with the relations
\begin{align*}
F_i^2F_j - (v+v^{-1})F_iF_jF_i + F_jF_i^2 = 0 &\conj{if} \lvert i-j\rvert =1,\\\quad \quad F_jF_i=F_iF_j &\conj{if} \lvert i-j\rvert \geq 2
\end{align*}
We now recall Brundan's characterization of Lusztig's PBW generators.
\begin{lemma}[{\cite[Lemma 3.5]{Bru98}}]\label{lemma:brundan}
Let $F'_{i,j}$ (for $1 \leq i < j \leq m$) be Lusztig's PBW generators, and rescale them by $F_{i,j} := v^{1-j+i}F'_{i,j}$. Then the $F_{i,j}$ are uniquely characterized by the identities
\begin{align*}
F_{i,i+1} &= F_i\\
F_{i,j} &= [F_{i+1,j},F_i]_{v^{-1}}&\mathrm{for }\,\, 1 \leq i < j-1 < m
\end{align*}
\end{lemma}
\begin{remark}
Three remarks are in order here.
\begin{enumerate}
\item For brevity we have suppressed the choice of reduced word needed in the above statement  (for details see \cite[Sec.\ 3]{Bru98}). 
\item Lusztig's braid group action does \emph{not} preserve the subalgebra $\U_v({\mathfrak{sl}_m})^- \subset \U_v(\mathfrak{sl_m})$. It is a nontrivial part of the above lemma that the elements $F_{i,j}$ actually are contained in the subalgebra $\U_v(\mathfrak{sl}_m)^-$, even though they are originally defined in terms of the braid group action. 
\item We have rescaled Lusztig's PBW generators to simplify the powers of $v$ in the relations they satisfy.
\end{enumerate}
\end{remark}

\begin{proposition}[{\cite[4.5]{Lus90}}]\label{prop:lusztigpbw}
The set 
\[
\left\{ \prod_{1\leq i < j \leq m} F_{i,j}^{N_{(i,j)}} \,\,\mid\,\, N_{(i,j)} \in \NN \right\}
\]
is a basis for $\U_v(\mathfrak{sl}_m)^-$ over $\ZZ[v,v^{-1}]$. (The order of the terms of the product is the lexicographic order: $(i,j) < (i',j')$ if $i < i'$ or if $i=i'$ and $j < j'$.)
\end{proposition}

We now show that the HOMFLY-PT skein relations give a PBW presentation of $\U_q(\mathfrak{sl}_m)^-$ in terms of the generators $F_{i,j}$. In other words, these relations allow arbitrary products of the $F_{i,j}$ to be reordered lexicographically in an obvious way. The first relation is the HOMFLY-PT skein relation in the interior of the disk, and the last three are the boundary skein relations.
\begin{lemma}
The algebra $\U_q(\mathfrak{sl}_m)^-$ has a presentation with generators $F_{i,j}$ for $1 \leq i < j \leq m$ and the following relations for $a < b < c < d$:
\begin{align}
[F_{a,c}, F_{b,d}]_1 &= (v^{-1}-v) F_{a,d}F_{b,c}\label{eq:Lsa}\\
[F_{a,b}, F_{c,d}]_1 &= 0\label{eq:Lsb}\\
[F_{a,d}, F_{b,c}]_1 &= 0\label{eq:Lsc}\\
[F_{b,c}, F_{a,b}]_{v^{-1}} &= F_{a,c}\label{eq:Lsdhahaha}\\
[F_{a,c},F_{b,c}]_{v^{-1}} &= 0\label{eq:Lse}\\
[F_{a,c}, F_{a,b}]_{v^{-1}} &=0\label{eq:Lsf}
\end{align}
\end{lemma}
\begin{proof}
We only  sketch the proof since most of the relations have appeared previously. To prove the listed relations we use the map $\varphi: \U_v(\mathfrak{sl}_m)^- \to \tDHa(A_{m})$ defined by
\[
\varphi(v) = q^{-1},\quad \quad \varphi(F_i) = z_{i,0}
\]
It is easy to see that $\varphi$ is injective,
and from Lemma \ref{lemma:brundan} and Proposition \ref{assocprop}, it follows that 
\[
\varphi(F_{i,j}) = z_{(i,j),0}
\]
Relation \eqref{eq:Lsa} is the first relation in Remark \ref{rmk:quiverskein}, and relations \eqref{eq:Lsb} and \eqref{eq:Lsc} follow from the fact that disjoint arcs in the Fukaya category commute in the Hall algebra. Relation \eqref{eq:Lsdhahaha} is Proposition \ref{assocprop} (or Lemma \ref{lemma:brundan}), and relations \eqref{eq:Lse} and \eqref{eq:Lsf} can be proved using straightforward modifications of arguments in Theorem \ref{relsthm}. 

Let $U$ be the quotient of the free algebra on the $F_{i,j}$ modulo the listed relations. So far we showed that there is an algebra map $\varphi: U \to \U_v(\mathfrak{sl}_m)$. 
The listed relations show that any product of the $F_{i,j}$ can be written as a linear combination of products of the $F_{i,j}$ in lexicographical order. Proposition \ref{prop:lusztigpbw} shows that (images under $\varphi$ of) these products form a basis of $\U_v(\mathfrak{sl}_m)$, so $\varphi$ bijects a spanning set to a basis and is therefore an isomorphism. 
\end{proof}

\begin{rmk}
Alternatively, one can use Ringel's theorem to produce an identification
between the PBW basis elements and the generators $z_{(a,b),0}$. The Hall
algebra elements $z_{(a,b),0}$ are precisely those associated to
indecomposable modules in the abelian category $Rep_{\FF_q}(A_{m-1})$.
\end{rmk}

\subsection*{Appendix II: Notation and Miscellanea}

\subsubsection*{Glossary of notation}\label{kcglossarysec}
\noindent
\begin{multicols}{2}
\begin{list}{}{
  \renewcommand{\makelabel}[1]{#1\hfil}
}
\item[$\#_p$] Def. \ref{def:resolutions}
\item[$\aprod_{i,j}$] Def. \ref{alggluedef}
\item[$\fol j k e$] Not. \ref{peternotation}
\item[$\inp{k},\,\inp{k}_h$] Eqn. \ref{knotation}
\item[$\sqcup_{i,j}$] Def. \ref{topgluedef}
\item[$\vnp{\ga}$] Def. \ref{fukcatdef}
\item[${[}a_n,\ldots, a_1{]}_q$] Def. \ref{bracketnotationdef}
\item[$a_{1,2,3}$] Prop. \ref{assprop}
\item[arc] see $(I,c,\tilde{c})$, $I=[0,1]$
\item[arc system] \S\ref{markedsec}
\item[$\a$] Thm. \ref{gluethm}
\item[$\a_{A',A}$] sheaf map, Cor. \ref{sheafpropcor}
\item[$A$] arc system, \S\ref{markedsec}
\item[$A_{m-1}$] Rmk. \ref{dhaamrmk}
\item[boundary arc] \S\ref{markedsec}
\item[boundary path] Def. \ref{fukcatdef}
\item[$\b$] Thm. \ref{gluethm}
\item[$\b_{A,A'}$] sheaf map, Cor. \ref{sheafpropcor}
\item[$\aC$] $\Ainf$-category, Def. \ref{ainfdef}
\item[disk sequence] Def. \ref{fukcatdef}
\item[$(D^2,m)$] See \S\ref{disksec}
\item[$D^\pi(\aC)$] Def. \ref{trianote}
\item[$D^\pi\F(S,A)$] Def. \ref{fukcatdef} \& Def. \ref{trianote}
\item[$\DHa$] Thm. \ref{dhadef}
\item[$\tDHa$] Def. \ref{dha2def}
\item[$\tDHa(A_{m-1})$] Rmk. \ref{dhaamrmk}
\item[$\tDHa(Q)$] Eqn. \eqref{tdhaqdef}
\item[$e$] Eqns. \eqref{gluefoleq}\,\&\,\eqref{eq:gluingfoliation}
\item[$\fo$] foliation data for standard form, \eqref{eq:efol} 
\item[$\eta$] grading, \S\ref{gradedsurfsec}
\item[$\E_i,\,\E_{i,n}$] Not. \ref{disknotation}
\item[$f$] Eqns. \eqref{gluefoleq}\,\&\,\eqref{eq:gluingfoliation}
\item[$\Alg(S,A)$] composition subalgebra, Def. \ref{compalgdef}
\item[foliation data] Not. \ref{disknotation}
\item[full arc system] \S\ref{markedsec}
\item[$\phi$] Proof of Thm. \ref{minarcthm}
\item[$\phi,\,\phi_i$] Prop. \ref{dereqprop}
\item[$F$] free algebra
\item[$F^L_{X,Y}$] Def. \ref{toenalgdef}
\item[$\SAlg(S,M)$] Fukaya Hall algebra, Def. \ref{skdef}
\item[$\F(S,A)$] Fukaya category, Def. \ref{fukcatdef}
\item[$g$] Eqns. \eqref{gluefoleq}\,\&\,\eqref{eq:gluingfoliation}
\item[grading] \S\ref{gradedsurfsec}
\item[$h : \A\to\ZZ$] Not. \ref{disknotation}
\item[$\Ho(\aC)$] Def. \ref{hodef}
\item[$\Hom(X,L)_Y$] Not. \ref{homnote}
\item[$i_p(c,d)$] intersection index, \S\ref{gradedsurfsec}
\item[internal arc] \S\ref{markedsec}
\item[$I$] 1-manifold, \S\ref{gradedsurfsec}
\item[$I$] vertices, \S\ref{latticealgsec}
\item[$(I,c,\tilde{c})$] curve, \S\ref{gradedsurfsec}
\item[$k$] field, often $k=\FF_q$
\item[$\kappa,\,\bar{\kappa}$] Def. \ref{compalgdef} \& Thm. \ref{minarcthm}
\item[$\A,\,\A_m$] Not. \ref{disknotation}
\item[$\mu_d$] Def. \ref{ainfdef}
\item[$M$] marked intervals, \S\ref{markedsec}
\item[$\M(S,M)$] arc system category, Def. \ref{modulidef}
\item[$p(k)$] $(-1)^k$
\item[$\psi$] Proof of Thm. \ref{minarcthm}
\item[$\Pi$] Prop. \ref{karoubiprop}
\item[$\Pi_1$] paths, \S\ref{surfsec}
\item[$Q$] quiver, \S\ref{latticealgsec}
\item[$r(Y,X)$] Cor. \ref{cor:boundaryskein} 
\item[$\Rep_k(Q)$] \S\ref{latticealgsec}
\item[$\s,\,\s^n$] suspension, \S\ref{gradedsurfsec}
\item[$S$] surface, \S\ref{markedsec}
\item[$\partial S$] boundary, \S\ref{markedsec}
\item[$(S,A)$] surface \& arc system, \S\ref{markedsec}
\item[$(S,M)$] surface \& marked intervals, \S\ref{markedsec}
\item[$\Si$] Def. \ref{adddef}
\item[$\tau$] \S\ref{disksec}
\item[$\aT$] triangulated category
\item[$\Tw$] Def. \ref{twdef}
\item[$\ZZ X$] orbit of $X$ under $\ZZ$ action
\item[$z_i,\,z_{i,n}$] Prop. \ref{dhaanprop}
\item[$z_{(a,b),n}$] Def. \ref{usefulcor}
\end{list}
\end{multicols}

\subsubsection*{Arc relations}\label{arcrelssec}
The table below summarizes the key results of Section \ref{algsec}. In each case there are elements $\z_{(a,b),n} \in \tDHa(A_{m-1})$, see Def. \ref{usefulcor}, and the endpoints of the elements are always in alphabetical order: $1 \leq a < b < c < d \leq m$.
\[
\begin{array}{lll}
\eqref{s0:itm} & [\z_{(b,c),n}, \z_{(a,b),n}]_q = \z_{(a,c),n} &  \textrm{Finger relation}\\
\eqref{s1:itm} &  [\z_{(a,c),n}, \z_{(b,c),n-1}]_q = \z_{(a,b),n} & \textrm{Boundary skein relation 1}\\
\eqref{eq:skeinleft} & [\z_{(a,b),n+1},\z_{(a,c),n}]_q = \z_{(b,c),n} & \textrm{Boundary skein relation 1'}\\
\eqref{s2:itm} & [\z_{(a,d),n}, \z_{(b,c),n-1}]_1 = 0  & \textrm{Interwoven commutativity}\\
\eqref{s3:itm} & [\z_{(a,b),n}, \z_{(c,d),k}]_1 = 0 & \textrm{Distant commutativity}\\
\eqref{eq:skeinselfext} & [z_{(a,b),n}, z_{(a,b),n+k}]_{q^{2(-1)^{k}}} = \delta_{k,1} q^{-1}/(q^2-1) & \textrm{Self-skein relation}\\
\end{array}  
\]

\subsubsection*{$q$-Analogues of the Lie bracket and Lemmas}
If $A$ is a $\ZZ[q]$-algebra and $x,y\in A$ then the $q$-analogue of the Lie bracket $[x,y]_q\in A$ is defined by the equation below.
\[
  [x,y]_q := xy - qyx
\]

This $q$-commutator satisfies a number of elementary algebraic identities which are used throughout the paper above. 

\begin{prop}\label{qalgprop}
  The $q$-analogue of the Lie bracket satisfies the two algebraic identities below.
  \begin{description}
\item[(AS)\namedlabel{as:itm}{\lab{AS}}] $q$-antisymmetry:
$$[x,y]_q = -q [y,x]_{q^{-1}}.$$
\item[(OJ)\namedlabel{oj:itm}{\lab{OJ}}] omni-Jacobi:
$$[x,[y,z]_{ac}]_{ab} + a[z,[x,y]_{abc}]_{a^{-1}} + ab[y,[z,x]_c]_{b^{-1}} = 0\conj{ where } a,b,c\in \QQ(q).$$
    \end{description}
\end{prop}

All instances of anti-commuting brackets and Jacobi relations in this paper
can be derived from the equations above. Below we introduce a shorthand
notation for the application of the omni-Jacobi relation.

\begin{notation}
Within the body of a proof the shorthand $\eqref{oj:itm}[u,v,w]$ will be used to denote the relation obtained from the omni-Jacobi relation after performing the substitution $a=u$, $b=v$ and $c=w$. For example, $\eqref{oj:itm}[1,1,q]$ is the equation
  $$[x,[y,z]_q]_1 + [z,[x,y]_q]_1 + [y,[z,x]_q]_1 = 0.$$
\end{notation}  

The $q$-Jacobi relation will sometimes be used in the shortened form featured below.

\newcommand{\pta}{f} 
\newcommand{\ptb}{g} 
\begin{lemma}\label{lemma:commids}
If $A$ is a $\ZZ[q]$-algebra then the following identities hold:
\begin{align}
[[x,y]_\pta,z]_\ptb = [[x,z]_\ptb , y]_\pta &\conj{ if }[y,z]_1=0\label{eq:comm1}\\
[x,[y,z]_\pta]_\ptb = [y,[x,z]_\ptb]_\pta   &\conj{ if } [x,y]_1=0\label{eq:comm2}\\
[x,[y,z]_\pta]_\ptb = [[x,y]_\ptb,z]_\pta  &\conj{ if } [x,z]_1=0\label{eq:comm4usedlaterdonotdelete}
\end{align}
\end{lemma}
\begin{proof}
The arguments for these identities are essentially
  identical so we only include the proof of Equation \eqref{eq:comm2}, which follows from the omni-Jacobi relation $\eqref{oj:itm}[\pta\ptb,\pta^{-1},\ptb^{-1}]$,
\begin{align*}
  [x,[y,z]_\pta]_\ptb &= \pta\ptb [z,[x,y]_1]_{\pta^{-1}\ptb^{-1}} - \b [y,[z,x]_{\ptb^{-1}}]_\pta  \\
&= 0 + [y,[x,z]_\ptb]_\pta.
\end{align*}  
\end{proof}

The lemma below states that if the elements of two words in an algebra individually $q$-commute then the two words must $q$-commute.

\begin{lemma}\label{lemma:qcomm}
  Suppose that $A$ is a $\ZZ[q]$-algebra and two elements: $x = x_1 x_2\cdots x_m$, $y = y_1 y_2 \cdots y_n$ satisfy $[x_i,y_j]_{a_{ij}} = 0$ for $1 \leq i \leq m$ and $1 \leq j \leq n$. Then $[x,y]_{\rho} = 0$ where $\rho = \prod_{i,j} a_{ij}$.
\end{lemma}
\begin{proof}
By assumption one can move each term in $y$ past each term in $x$ at the expense of multiplying by a factor of $a_{ij}$. Commuting all of the $y$ terms past all of the $x$ terms gives a product $\rho$ of all $a_{ij}$.
\end{proof}

\pagebreak
\bibliography{sillyman1}{}

\providecommand{\bysame}{\leavevmode\hbox to3em{\hrulefill}\thinspace}
\providecommand{\MR}{\relax\ifhmode\unskip\space\fi MR }
\providecommand{\MRhref}[2]{%
  \href{http://www.ams.org/mathscinet-getitem?mr=#1}{#2}
}
\providecommand{\href}[2]{#2}
\begin{thebibliography}{BGSLX16}

\bibitem[Abo08]{Abo08}
Mohammed Abouzaid, \emph{On the {F}ukaya categories of higher genus surfaces},
  Adv. Math. \textbf{217} (2008), no.~3, 1192--1235.

\bibitem[BGSLX16]{BGSX16}
Francois Bergeron, Adriano Garsia, Emily Sergel~Leven, and Guoce Xin,
  \emph{Compositional {$(km,kn)$}-shuffle conjectures}, Int. Math. Res. Not.
  (2016), no.~14, 4229--4270.

\bibitem[BHMV95]{BHMV}
C.~Blanchet, N.~Habegger, G.~Masbaum, and P.~Vogel, \emph{Topological quantum
  field theories derived from the {K}auffman bracket}, Topology \textbf{34}
  (1995), no.~4, 883--927.

\bibitem[Boc16]{Bok}
Raf Bocklandt, \emph{Noncommutative mirror symmetry for punctured surfaces},
  Trans. Amer. Math. Soc. \textbf{368} (2016), no.~1, 429--469.

\bibitem[Bri13]{Bri13}
Tom Bridgeland, \emph{Quantum groups via {H}all algebras of complexes}, Ann. of
  Math. (2) \textbf{177} (2013), no.~2, 739--759.

\bibitem[Bru98]{Bru98}
Jonathan Brundan, \emph{Modular branching rules and the {M}ullineux map for
  {H}ecke algebras of type {$A$}}, Proc. London Math. Soc. (3) \textbf{77}
  (1998), no.~3, 551--581.

\bibitem[BS12]{BS}
Igor Burban and Olivier Schiffmann, \emph{On the {H}all algebra of an elliptic
  curve, {I}}, Duke Math. J. \textbf{161} (2012), no.~7, 1171--1231.

\bibitem[CF94]{CraneFrenkel}
Louis Crane and Igor~B. Frenkel, \emph{Four-dimensional topological quantum
  field theory, {H}opf categories, and the canonical bases}, J. Math. Phys.
  \textbf{35} (1994), no.~10, 5136--5154.

\bibitem[Che13]{Che13}
Ivan Cherednik, \emph{Jones polynomials of torus knots via {DAHA}}, Int. Math.
  Res. Not. (2013), no.~23, 5366--5425.

\bibitem[DK13]{DK}
Tobias Dyckerhoff and Mikhail Kapranov, \emph{Triangulated surfaces in
  triangulated categories}, arXiv:1306.2545v3 (2013).

\bibitem[FT11]{FT11}
B.~L. Feigin and A.~I. Tsymbaliuk, \emph{Equivariant {$K$}-theory of {H}ilbert
  schemes via shuffle algebra}, Kyoto J. Math. \textbf{51} (2011), no.~4,
  831--854.

\bibitem[GN15]{GN15}
Eugene Gorsky and Andrei Negut, \emph{Refined knot invariants and {H}ilbert
  schemes}, J. Math. Pures Appl. (9) \textbf{104} (2015), no.~3, 403--435.

\bibitem[Gol86]{Goldman}
William~M. Goldman, \emph{Invariant functions on {L}ie groups and {H}amiltonian
  flows of surface group representations}, Invent. Math. \textbf{85} (1986),
  no.~2, 263--302.

\bibitem[Hen14]{Krause}
Krause Henning, \emph{{K}rull-{S}chmidt categories and projective covers},
  arXiv:1410.2822 (2014).

\bibitem[HKK14]{HKK}
Fabian Haiden, Ludmil Katzarkov, and Maxim Kontsevich, \emph{Flat surfaces and
  stability structures}, arXiv:1409.8611 (2014).

\bibitem[HL15]{HL}
David Hernandez and Bernard Leclerc, \emph{Quantum grothendieck rings and
  derived hall algebras}, J. reine angew. Math. \textbf{701} (2015), 77--126.

\bibitem[Kap98]{Kapranov}
Mikhail Kapranov, \emph{Heisenberg doubles and derived categories}, J. Algebra
  \textbf{202} (1998), no.~2, 712--744.

\bibitem[KKOY09]{Kapustin}
Anton Kapustin, Ludmil Katzarkov, Dmitri Orlov, and Mirroslav Yotov,
  \emph{Homological mirror symmetry for manifolds of general type}, Centr. Eur.
  J. Math. \textbf{571} (2009).

\bibitem[KL09]{KL09}
Mikhail Khovanov and Aaron~D. Lauda, \emph{A diagrammatic approach to
  categorification of quantum groups. {I}}, Represent. Theory \textbf{13}
  (2009), 309--347.

\bibitem[LP17]{Lekili}
Yanki Lekili and Alexander Polishchuk, \emph{Auslander orders over nodal stacky
  curves and partially wrapped fukaya categories}, arXiv:1705.06023 (2017).

\bibitem[Lus90]{Lus90}
George Lusztig, \emph{Finite-dimensional {H}opf algebras arising from quantized
  universal enveloping algebra}, J. Amer. Math. Soc. \textbf{3} (1990), no.~1,
  257--296.

\bibitem[Lus10]{LQG}
\bysame, \emph{Introduction to quantum groups}, Modern Birkh\"auser Classics,
  Birkh\"auser/Springer, New York, 2010.

\bibitem[MS17]{MortonSamuelson}
Hugh Morton and Peter Samuelson, \emph{The {HOMFLYPT} skein algebra of the
  torus and the elliptic {H}all algebra}, Duke Math. J. \textbf{166} (2017),
  no.~5, 801--854.

\bibitem[Nad15]{Nadler}
David Nadler, \emph{Cyclic symmetries of {$A\sb n$}-quiver representations},
  Adv. Math. \textbf{269} (2015), 346--363.

\bibitem[Neg14]{Neg14}
Andrei Negut, \emph{The shuffle algebra revisited}, Int. Math. Res. Not.
  (2014), no.~22, 6242--6275.

\bibitem[PZ01]{PZ}
Alexander Polishchuk and Eric Zaslow, \emph{Categorical mirror symmetry in the
  elliptic curve}, Winter {S}chool on {M}irror {S}ymmetry, {V}ector {B}undles
  and {L}agrangian {S}ubmanifolds, AMS/IP Stud. Adv. Math., vol.~23, Amer.
  Math. Soc., 2001, pp.~275--295.

\bibitem[Rin90]{Ringel}
Claus~Michael Ringel, \emph{Hall algebras and quantum groups}, Invent. Math.
  \textbf{101} (1990), 583--592.

\bibitem[Ros89]{Ros89}
Marc Rosso, \emph{An analogue of {P}.{B}.{W}. theorem and the universal
  {$R$}-matrix for {$U_h{\rm sl}(N+1)$}}, Comm. Math. Phys. \textbf{124}
  (1989), no.~2, 307--318.

\bibitem[{Rou}08]{Rou08}
R.~{Rouquier}, \emph{{2-Kac-Moody algebras}}, arXiv:0812.5023 (2008).

\bibitem[Sch16]{Sch16}
Olivier Schiffmann, \emph{Indecomposable vector bundles and stable {H}iggs
  bundles over smooth projective curves}, Ann. of Math. (2) \textbf{183}
  (2016), no.~1, 297--362.

\bibitem[Sei08]{Seidel}
Paul Seidel, \emph{Fukaya categories and {P}icard-{L}efschetz theory}, Zurich
  Lectures in Advanced Mathematics, European Mathematical Society (EMS),
  Z\"urich, 2008.

\bibitem[STZ14]{STZ14}
Nicol\`o Sibilla, David Treumann, and Eric Zaslow, \emph{Ribbon graphs and
  mirror symmetry}, Selecta Math. (N.S.) \textbf{20} (2014), no.~4, 979--1002.

\bibitem[SV13a]{SV13agt}
O.~Schiffmann and E.~Vasserot, \emph{Cherednik algebras, {W}-algebras and the
  equivariant cohomology of the moduli space of instantons on {$\bold{A}^2$}},
  Publ. Math. Inst. Hautes \'Etudes Sci. \textbf{118} (2013), 213--342.

\bibitem[SV13b]{SV13}
Olivier Schiffmann and Eric Vasserot, \emph{The elliptic {H}all algebra and the
  {$K$}-theory of the {H}ilbert scheme of {$\mathbb A^2$}}, Duke Math. J.
  \textbf{162} (2013), no.~2, 279--366.

\bibitem[To{\"e}06]{Toen}
Bertrand To{\"e}n, \emph{Derived {H}all algebras}, Duke Math. J. \textbf{135}
  (2006), no.~3, 587--615.

\bibitem[Tur91]{Turaevskein}
Vladimir~G. Turaev, \emph{Skein quantization of {P}oisson algebras of loops on
  surfaces}, Ann. Sci. \'Ecole Norm. Sup. (4) \textbf{24} (1991), no.~6,
  635--704.

\bibitem[Tur16]{Turaevbook}
\bysame, \emph{Quantum invariants of knots and 3-manifolds}, De Gruyter Studies
  in Mathematics, vol.~18, De Gruyter, Berlin, 2016.

\bibitem[VV11]{VV}
Michela Varagnolo and Eric Vasserot, \emph{Canonical bases and {KLR}-algebras},
  J. Reine Angew. Math. \textbf{659} (2011), 67--100.

\bibitem[Wit89]{Witten}
Edward Witten, \emph{Quantum field theory and the jones polynomial}, Comm.
  Math. Phys. \textbf{121} (1989), 351--399.

\bibitem[XX08]{XX}
Jie Xiao and Fan Xu, \emph{Hall algebras associated to triangulated
  categories}, Duke Math. J. \textbf{143} (2008), no.~2, 357--373.

\end{thebibliography}
\bibliographystyle{amsalpha}

\end{document}